\documentclass[11pt]{amsart}

\usepackage{amssymb,amsmath}
\usepackage{amsthm,enumerate}
\usepackage{graphicx,rotating}
\usepackage{qsymbols}
\usepackage{latexsym}
\usepackage{amsbsy}
\usepackage{mathbbol}
\usepackage{amsfonts}
\usepackage{enumitem}
\usepackage{txfonts}
\usepackage{dsfont}
\usepackage{bbm}
\usepackage{dsfont,txfonts}

\usepackage{color}

\newtheorem{theorem}{Theorem}[section]
\newtheorem{lemma}[theorem]{Lemma}

\newtheorem{definition}[theorem]{Definition}

\newtheorem*{comments}{Comments}

\newtheorem*{remarks}{Remarks}

\newcommand{\tw}{w'}
\newcommand{\tz}{z'}
\newcommand{\ttz}{{z''}}
\newcommand{\V}{\mathbb{V}}
\newcommand{\X}{\mathcal{X}}
\newcommand{\Y}{\mathcal{Y}}
\newcommand{\M}{\mathcal{M}_+}
\newcommand{\B}{{\rho}}
\renewcommand{\P}{{P}}
\newcommand{\Pc}{\mathcal{P}}
\newcommand{\supp}{{\rm supp}}

\newcommand{\R}{\mathbb{R}}
\newcommand{\N}{\mathbb{N}}
\newcommand{\Z}{\mathbb{Z}}

\newcommand{\1}{\mathds{1}}
\newcommand{\Pb}{\mathbb{P}}
\newcommand{\E}{\mathbb{E}}

\begin{document}

\title[Entropic curvature on graphs.]
{Entropic curvature on graphs  \\along  Schr\"odinger bridges at zero temperature.}
\author{ Paul-Marie Samson}

\date{\today}

\address{P.-M. Samson, Univ Gustave Eiffel, Univ Paris Est Creteil, CNRS, LAMA UMR8050 F-77447 Marne-la-Vallée, France }
\email{ paul-marie.samson@univ-eiffel.fr}
\thanks{This research is partly funded by the B\'ezout Labex, funded by ANR, reference ANR-10-LABX-58. The  author is supported by a grant of the Simone and Cino Del Duca Foundation.}

\begin{abstract}
 Lott-Sturm-Villani theory of curvature on geodesic spaces has been extended to discrete graph spaces by C. L\'eonard by replacing $W_2$-Wasserstein geodesics by Schr\"odinger bridges in the definition of entropic curvature \cite{Leo13, Leo17,Leo16}.  As a remarkable fact, as a temperature parameter goes to zero, these Schr\"odinger bridges are supported by  geodesics  of the space. We analyse this property on  discrete graphs to reach  entropic curvature  on  discrete spaces. Our approach  provides   lower bounds for the entropic curvature for several examples of graph spaces: the lattice $\Z^n$ endowed with the counting measure, the discrete cube endowed with product probability measures, the circle,
 the complete graph, the Bernoulli-Laplace model. Our general results  also apply to a large class of graphs which are not specifically studied in this paper.  
 
 As opposed to Erbar-Maas  results on graphs  \cite{Maa11,EM12,EM14},  entropic curvature results of this paper   imply  new Pr\'ekopa-Leindler type of inequalities on discrete spaces, and new transport-entropy inequalities related to  refined concentration properties for the graphs mentioned above. For example on the discrete hypercube $\{0,1\}^n$ and for the Bernoulli Laplace model, a new $W_2-W_1$ transport-entropy inequality is reached, that can not be derived by usual induction arguments over the dimension $n$.  As a surprising fact, our method also gives  improvements of weak transport-entropy inequalities (see \cite{GRST17}) associated to  the so-called convex-hull method by Talagrand \cite{Tal95}. 
 \end{abstract}


\maketitle

The paper starts with a brief overview  about known results  concerning entropic curvature  on discrete graphs. Then we introduce  a  specific  entropic curvature property on graphs  (see Definition \ref{defcourb}), derived from  C. L\'eonard approach \cite{Leo13, Leo17,Leo16}, and  dealing with Schr\"odinger bridges at zero temperature.   

The main curvature results  are given in section 2, with their connections to
new transport-entropy inequalities.
  The concentration properties following from such transport-entropy inequalities are not developed in the present paper. For that purpose, we refer to \cite{Sam17b} and  \cite{GRST17} by Gozlan \& al, where the link between transport-entropy inequalities and concentration properties are widely investigated.

 The strategy of proof, presented in section 3, uses the so called {\it slowing-down procedure} for Schr\"odin-ger bridges associated to jump processes on discrete spaces pushed forward by C. L\'eonard.  The key theorem of the present paper,   Theorem \ref{limdevseconde} (with Lemma \ref{lemconvex}), is derived  from  this procedure, which consists of decreasing a temperature parameter $\gamma$ to 0 in order to construct  $W_1$-Wasserstein  geodesics on the set of  probability measures on the graph. All the curvature results  of this paper are derived  from  Theorem \ref{limdevseconde}.  Our strategy  also applies for many other graph spaces which are not considered in this paper. The main goal of this work is  to push forward Leonard's  slowing-down procedure  to reach entropic curvature on graphs through  few significant new results. In a forthcoming paper, one will give sufficient geometric conditions to reach entropic curvature property on non-specific graphs from Theorem \ref{limdevseconde}.

\tableofcontents

\vspace{-1,5 cm}

\section{Introduction : Schr\"odinger bridges for entropic curvature }

For any measurable space $\Y$, we note $\M(\Y)$ the set of all non-negative $\sigma$-finite measures on $\Y$ and  $ {\mathcal P}(\Y)$   the set of all probability measures on $\Y$.

Let $(\X,d)$ be a geodesic space equipped with a reference measure $m\in \M(\X)$.
According to  Lott-Sturm-Villani theory of curvature  on geodesic spaces \cite{LV09,Stu06a,Stu06b,Vil09}, a lower bound $K\in \R$ on the entropic curvature of the space $(\X,d,m)$  is  characterized by a {\it $K$-convexity property} of the {\it relative entropy} along {\it constant speed geodesics}  of the { \it Wasserstein  space} $(\Pc_2(\X),W_2)$. 
Let us precise this property for the non specialist reader. 
By definition,  the  {\it relative entropy} of a probability measure $q$ on a measurable space $\Y$ with respect to a probability measure   $r\in {\mathcal P}(\Y)$, also called Kullback-Leibler distance between $q$ and $r$,  is  given by 
 \[H(q|r):= \int_{\Y} \log(dq/dr) \,dq\qquad \in [0, \infty], \]         
if $q$ is absolutely  continuous with respect to $r$ and    $H(q|r):=+\infty$ otherwise. 
As explained in \cite{Leo14}, this definition extends to unbounded measures $r\in\M(\Y)$
as follows. Since $r$ is a  $\sigma$-finite measure, there exists some measurable function $w:\Y\to [0,+\infty)$ such that 
\[z_w:=\int e^{-w} dr<\infty.\]
Define the probability measure $r_w= \frac{e^{-w}}{z_w} \,r$. Then the definition of $H(q|r)$ is given for all $q\in  {\mathcal P}(\Y)$ such that 
$\int w \,dq<+\infty$ by
 \[H(q|r)= H(q|r_w)-\int w\,dq -\log z_w\quad \in(-\infty,+\infty].\]
According to \cite{Leo14}, this definition makes sense since the right-hand side does not depends on the function $w$ satisfying $z_w<\infty$ and $\int w \,dq<+\infty$.
We refer to \cite{Leo14} for more details and properties about this definition of relative entropy with unbounded $\sigma$-finite measures.
Let 
$\Pc_2(\X)$ denote the space of probability measures with second moment, and let $W_2$ be the {\it Wasserstein distance of order 2} on $\Pc_2(\X)$: namely, for any $\nu_0,\nu_1\in \Pc_2(\X)$,
\begin{equation}\label{defW2}
W_2(\nu_0,\nu_1):=\left(\inf_{\pi\in \Pi(\nu_0,\nu_1)}\iint d(x,y)^2 d\pi(x,y)\right)^{1/2},
\end{equation}  
where $\Pi(\nu_0,\nu_1)$ is the set of all probability measures on the product space $\X\times \X$ with first marginal $\nu_0$ and second marginal $\nu_1$ (also called {\it transference plans} from $\nu_0$ to $\nu_1$).  
A path $(\nu_t)_{t\in [0,1]}$ in $\Pc_2(\X)$ is a {\it  constant speed $W_2$-geodesic} from $\nu_0$ to $\nu_1 $ if for all $0\leq s<t\leq 1$, $W_2(\nu_s,\nu_t)=(t-s)W_2(\nu_0,\nu_1)$. 
The {\it $K$-convexity property} of the relative entropy $H(\cdot|m)$ is expressed as follows: for any $\nu_0,\nu_1\in \Pc_2(\X)$ whose supports are included in the support of $m$, there exists a constant speed $W_2$-geodesic $(\nu_t)_{t\in [0,1]}$ from $\nu_0$ to $\nu_1$ such that  for all $t\in[0,1]$,
\begin{eqnarray}\label{deplace}
H(\nu_t|m)\leq (1-t)\, H(\nu_0|m)+t \,H(\nu_1|m)-\frac K2\, t(1-t)\,W_2^2(\nu_0,\nu_1).
\end{eqnarray}
If such a property holds, one says that the Lott-Sturm-Villani entropic curvature of the space $(\X,d,m)$ is bounded from below by $K$.

Property \eqref{deplace} with $K=0$ has been discovered by McCann on the Euclidean space  $(\X,d)=(\R^d,|\,\cdot\,|_2)$ endowed with the Lebesgue measure \cite{McC97}. More generally, as a remarkable fact, when  $\X$  is a Riemannian manifold equipped with its  geodesic distance $d$ and  a measure $m$  with density $e^{-V}$ with respect to the volume measure,  property \eqref{deplace} is equivalent to the so-called Bakry-Emery curvature condition $CD(K,\infty)$: $\rm{ Ricc} +{\rm Hess}(V)\geq K$ (see e.g. \cite{Bak94}). As a consequence, due to the wide range of implications of this notion of curvature, 
 property \eqref{deplace} has been  used as a guideline by Lott-Sturm-Villani  to define curvature on geodesic spaces (see also \cite{AGS08,AGS14}) and then by  different authors to  propose entropic definitions of curvature on discrete spaces : Bonciocat-Sturm \cite{BS09}, Ollivier-Villani on the discrete cube \cite{OV12}, Erbar-Maas \cite{Maa11, EM12,EM14}, Mielke \cite{Mie13}, L\'eonard \cite{Leo13, Leo17,Leo16}, Hillion \cite{Hil14,Hil17}   and Gozlan-Roberto-Samson-Tetali \cite{GRST14}. 
 
 This paper concerns L\'eonard entropic approach of curvature in discrete setting, from which we also recover results from \cite{GRST14} and \cite{Hil14}. In discrete spaces,  several other notions of curvature have already been studied which are not considered in this paper  : the coarse Ricci curvature  \cite{Oll09, Oll13},  
 the Bochner-Bakry-Emery approach with the (Bochner) curvature  \cite{CDP09, KKRT16} and the curvature dimension or exponential curvature dimension inequality \cite{BHLLMY15}.

 For  $m$ as  unique invariant probability measure of a Markov kernel on a  discrete space $\X$, a first global entropic approach has been proposed by   M. Erbar and J. Maas \cite{Maa11,EM12,EM14}. The core of their approach is the  construction of an abstract Wasserstein distance ${\mathcal W}_2$ on  $\Pc(\X)$, that replaces the Wasserstein distance $W_2$ in \eqref{deplace}.  This distance ${\mathcal W}_2$ is defined using a discrete analogue of the Benamou-Brenier formula for $W_2$, in order to  provide a Riemannian  structure for the probability space $\Pc(\X)$. Unfortunately,  there is no static definition of ${\mathcal W}_2^2$ as a minimum of a cost  among  transference plans $\pi$ as in the definition \eqref{defW2} of $W^2_2$. 
Erbar-Maas entropic Ricci curvature  definition  satisfies a tensorization property for product of graphs that allows to consider high dimensional spaces \cite{EM12}. This definition has been used to get lower bounds on curvature  for several models of graphs : the discrete circle, the complete graph, the discrete hypercube \cite{Maa11,EM12}, the Bernoulli-Laplace model, the random transposition model \cite{EMT15,FM16}, birth and death processes, zero-range processes \cite{FM16}, Cayley graphs of non-abelian groups, weakly interacting Markov chains such as the Ising model \cite{EHMT17}.  The main strategy of all this papers is to prove an equivalent criterion of Erbar-Maas entropic  curvature given in \cite{EM12},  by identifying some discrete analogue of the Bochner identity in continuous setting.

 Finding a minimizer in the definition of $W_2(\nu_0,\nu_1)$ is known as the quadratic Monge-Kantorovich problem. By the so-called {\it slowing down procedure}, T. Mikami \cite{Mik04} and then C. L\'eonard \cite{Leo12,  Leo13, Leo16,Leo17} show that  the quadratic Monge-Kantorovich problem in continuous, but also the $W_1$-Monge-Kantorovich problem in discrete, can be understood as the limit  of a sequence of entropy minimization problems, the so-called {\it Schr\"odinger problems}. 
  
 In this paper, the {\it slowing down procedure}, described further,  is used to prove entropic curvature properties of type \eqref{deplace} as $\X$ is a   graph, endowed with its natural graph distance $d=d_\sim$,  and with a measure $m$, reversible with respect to some generator $L$.  More precisely, in  property \eqref{deplace},   constant speed $W_2$-geodesics $(\nu_t)_{t\in [0,1]}$  are replaced by  constant speed $W_1$-geodesics where  $W_1$ is {\it the Wasserstein distance of order 1}  given by
 \[W_1(\nu_0,\nu_1):=\inf_{\pi\in \Pi(\nu_0,\nu_1)}\iint d(x,y)\,d\pi(x,y), \qquad \nu_0,\nu_1\in \Pc(\X).\]     
  As explained below, each of these constant speed $W_1$-geodesics  is the    limit path of a sequence  of Schr\"odinger briges  $(\widehat Q^\gamma_t)_{t\in [0,1]}$ indexed by a temperature parameter $\gamma>0$, as $\gamma$ goes to zero. Given two probability measures $\nu_0$ and $\nu_1$, this  constant speed $W_1$-geodesic selected from this cooling down process (or slowing down procedure) is unique. 
  According to its construction, we call it  {\it Schr\"odinger brige at zero temperature} and we denote it  $(\widehat Q_t)_{t\in [0,1]}$ throughout this paper ($\widehat Q_0=\nu_0$ and $\widehat Q_1=\nu_1$).  For $x,y\in \X$, one denotes $({Q_t}\!^{ x,y})_{t\in [0,1]}$ the  
  Schr\"odinger brige at zero temperature  from the Dirac measure $\delta_x={Q_0}\!^{ x,y}$ to the Dirac measure $\delta_y={Q_1}\!^{ x,y}$. Actually the bridge $(\widehat Q_t)_{t\in [0,1]}$ is a mixture of Schr\"odinger briges from Dirac measures on the support of $\nu_0$  to Dirac measures  on the support of $\nu_1$, according to a  selected transference plan denoted $\widehat \pi\in \Pi(\nu_0,\nu_1)$,  that achieves $W_1(\nu_0,\nu_1)$. Namely, one has for any $z\in \X$
  \begin{eqnarray}\label{limitbridge} 
 \widehat Q_t(z)= \iint {Q_t}\!^{ x,y}(z)\, d\widehat \pi(x,y), \qquad \mbox{with}\qquad \iint d(x,y)\,d\widehat \pi(x,y)=W_1(\nu_0,\nu_1).
 \end{eqnarray}
Observe that the set of minimizers of $W_1(\nu_0,\nu_1)$, also called $W_1$-optimal couplings of $\nu_0$ and $\nu_1$, is a convex set that is not necessarily reduced to a singleton. However, according to Leonard's paper \cite{Leo16}, we know that given $\nu_0,\nu_1$,  $\widehat \pi$ is uniquely determined, as a minimizer of a strictly convex optimization problem (see \cite[Result 0.2]{Leo16}).
  In our setting of property \eqref{deplace} on graphs, the  curvature term $ \,W_2^2(\nu_0,\nu_1)$ is also replaced by some transport cost $C_t(\widehat \pi)$ that depends on the selected $W_1$-minimizer $\widehat \pi\in \Pi(\nu_0,\nu_1)$, and  may also  depend  on the parameter $t\in (0,1)$. Let $\Pc_b(\X)$ denotes the set of probability measures on $\X$ with finite support. The analogue of property \eqref{deplace} on discrete graphs at the focus of this work is the following.
 \begin{definition}\label{defcourb} On the discrete space $(\X,d,m,L)$, one says that  the 
  relative entropy is $C$-displacement convex where $C=(C_t)_{t\in[0,1]}$,  if for any probability measure $\nu_0,\nu_1\in \Pc_b(X)$, the Schr\"odinger bridge at zero temperature $(\widehat Q_t)_{t\in [0,1]}$ from $\nu_0$ to $\nu_1$, satisfies for any $t\in(0,1)$, 
 \begin{eqnarray}\label{deplacebis}
H(\widehat Q_t|m)\leq (1-t) H(\nu_0|m)+t \,H(\nu_1|m)- \frac{t(1-t)}2C_t(\widehat \pi).
\end{eqnarray}
\end{definition}
 For some of the graphs studied in this paper,  the cost  $C_t(\widehat \pi)$ is bigger than  $  K\Big(\iint d(x,y)\,d\widehat \pi(x,y) \Big)^2=K\,W_1(\nu_0,\nu_1)^2$ for any $t\in (0,1)$ with $K>0$. Such a property is also a consequence  of Erbar-Maas positive entropic  curvature since ${\mathcal W}_2^2\geq2 W_1^2$  (see \cite[Proposition 2.12]{EM12}). However, their convexity property of entropy deals with ${\mathcal W}_2$-geodesics on $\Pc(\X)$, whereas property \eqref{deplacebis} deals with  $W_1$-geodesics. 
 As a definition in this paper, the largest constant $K\in \R$ so that 
 \eqref{deplacebis} holds  with $C_t(\widehat \pi)=K \,W_1(\nu_0,\nu_1)^2$ for any $\nu_0,\nu_1\in \Pc_b(X)$ and any $t\in(0,1)$ is called, if it exists,  {\it the $W_1$-entropic curvature} of
the space $(\X,d,m,L)$.  

 Given a non negative cost function $c:\N\to \R^+$, let us denote 
 \[T_c(\widehat \pi):= \iint c(d(x,y))\,d\widehat \pi(x,y)\] and $T_2: =T_c$ for the   square function $c(d)=d^2$, $d\geq 0$. For some graphs in this paper, in order to compare our results with the $W_2^2$ cost that appears in \eqref{deplace} on geodesics spaces,  we prove that   $C_t(\widehat \pi)\geq {K} \,T_{c_2}(\widehat \pi)$ with $K\geq 0$, where one denotes by $c_2$ any universal cost function (independent of any characteristic of the graph) satisfying
  \[\frac{d(d-1)}2\leq c_2(d)\leq d^2\] and which is equivalent to  the square function at infinity
 \[\lim_{d\to \infty} \frac{c_2(d)}{d^2}=1.\]
 For such a cost function, one has  for any $\varepsilon \in (0,1)$ and  any $d\in \N$, 
 \[ c_2(d)\geq (1-\varepsilon) d(d-1) - \alpha(\varepsilon) d,\]
 where  $\alpha$ is the non-negative  function given by $\alpha(\varepsilon):=\sup_{k\in \N^*} \left\{(1-\varepsilon) (k-1)- \frac{c_2(k)}{k}\right\}$ ($\alpha(\varepsilon)=0$ for $\varepsilon \in (1/2,1)$).
 It follows that $T_{c_2}(\widehat \pi)$ is controlled by the Wasserstein distances as follows, for any $\varepsilon \in (0,1)$
 \[ T_{c_2}(\widehat \pi)
 \geq \sup_{\varepsilon\in(0,1)} \left\{(1-\varepsilon) W_2^2(\nu_0,\nu_1)- [(1-\varepsilon) +\alpha(\varepsilon)]  W_1(\nu_0,\nu_1)\right\}\geq \frac12  \left(W_2^2(\nu_0,\nu_1)-W_1(\nu_0,\nu_1)\right)\geq 0.\]
 Therefore,
the cost $T_{c_2}(\widehat \pi)$ can be interpreted as a discrete analogue of the cost $W_2^2(\nu_0,\nu_1)$ in the usual $K$-convexity property \eqref{deplace} on geodesic spaces. 
As a definition in this paper,  {\it the $T_2$-entropic curvature} of the space $(\X,d,m,L)$ is the largest constant $K\in \R$ so that there exists a cost $c_2$ satisfying the above conditions and such that 
 \eqref{deplacebis} holds  with $C_t(\widehat \pi)=K \, T_{c_2}(\widehat \pi)$ for any $\nu_0,\nu_1\in \Pc_b(X)$ and any $t\in(0,1)$. 
  
Due to the abstract definition of the cost ${\mathcal W}_2^2$ with a discrete
analogue of Benamou-Brenier formula, we don't know how to compare ${\mathcal W}_2^2$ with costs involving transference plans and the discrete structure of the graph such as $T_{c_2}$ or any other proposed costs $C_t$ of this paper, excepted with $W_1^2$ for which ${\mathcal W}_2^2\geq 2 W_1^2$. As a consequence, it is still a challenging problem to reach most of the results of the present paper from Erbar-Maas approach of entropic curvature on discrete spaces. 

According to the property of the function  $c_2$, the cost $c_2(d(x,y))=0$ if $x$ and $y$ are neighbours. Therefore the  transport-cost $T_{c_2}$ does not well measure the distance between probabilities with  close supports. Observe that  such type of costs also appear in the paper by Bonciocat-Sturm \cite{BS09} in their definition of  rough (approximate) lower curvature.

For the graph with positive $W_1$-entropic curvature studied in this paper (the complete graph, the discrete hypercube and  the Bernoulli-Laplace model), one may bound from below  the cost $C_t(\widehat \pi)$ by different  symmetric versions of  weak transport  cost denoted by $\widetilde T_t(\widehat \pi)$ and bounded from below by the so-called weak optimal transport costs introduced in the paper \cite{GRST17}. Weak transport-entropy inequalities where introduced by K. Marton in the seminal work \cite{Mar96b} in order to get refined concentration properties for product measure, related to concentration's results derived from the so-called ``Convex hull method'' by M. Talagrand \cite{Tal95}.  
It was pushed forward in \cite{GRST14} that these costs are related to  displacement convexity property of entropy along $W_1$-geodesics in the case of the complete graph and of  the discrete hypercube. From the present paper, we learn that  same observation extends to models without product structure with different kind of weak transport costs, like for the Bernoulli-Laplace model.  Actually, our approach seems  very efficient to prove (weak) transport-entropy inequalities  since we discover new ones and    get improved versions of the known ones.  

As a guideline for other graphs, we present in this paper for the discrete hypercube and the Bernoulli Laplace model how to easily reach modified logarithmic Sobolev inequalities from  the $C$- displacement convexity property \eqref{deplacebis}. The strategy is to analyse the $C$-displacement convexity property \eqref{deplacebis} as $t$ goes to zero when the cost  $C_t(\widehat \pi)$ is lower bounded by some  weak transport costs  $\widetilde T_t(\widehat \pi)$. It may provide  different kinds of modified logarithmic Sobolev inequalities, depending on the model and the  structure of  weak transport cost $\widetilde T_t(\widehat \pi)$. Contrarily to the Erbar-Maas entropic curvature approach, connections and comparisons with other known modified logarithmic Sobolev inequalities with optimal constants are not always easy to handle. It still remains a challenge to improve our strategy or find other ways to reach modified logarithmic Sobolev inequalities from the use of Schr\"odinger bridges at zero temperature in discrete spaces.      

Applying  usual duality arguments,  the $C$-displacement convexity property \eqref{deplacebis} also implies new kinds of curved Prékopa-Leindler inequalities, 
as opposed to Erbar-Maas entropic approach of curvature due to the abstract definition of ${\mathcal W_2}$  (see Theorem  \ref{prek}).

Following the line of the paper \cite{GRST14}, a tensorization property of the $C$-displacement convexity property holds involving  Knothe-Rosenblatt coupling (see Theorem \ref{tenso}).
 
{In the present paper, a $C$-displacement convexity property is proved for the following discrete spaces~: the lattice $\Z^n$ endowed with the counting measure (see Theorem \ref{thmZ}), the discrete hypercube endowed with product probability measures (see Theorem \ref{thmcube}), the discrete circle endowed with uniform measure (see Theorem \ref{thmcircle}), 
 the complete graph (see Theorem \ref{thmcomplete}), the Bernoulli-Laplace model (see Theorem \ref{thmslicecube}).  
 For all these graphs, one gets a non-negative lower bound for their $W_1$ or  $T_2 $-entropic curvature. 
 
 In a forthcoming paper, starting from the key Theorem \ref{limdevseconde}, one will give sufficient geometric conditions on balls of radius 2, available on any graph space $(\X,d,m,L)$, that give lower bounds on  $W_1$ or  $T_2 $-entropic curvature. Other  examples of graphs will be studied, like the random transposition model on the symmetric group $S_n$ (for which the $W_1$-entropic curvature is lower bounded by $4/n^2$) or the multinomial distribution on the set $\X:=\{(x_1,\ldots, x_d)\in \N^d\,|\, x_1+\cdots +x_d=N\}$. Examples of graphs with negative entropic curvature like trees and also measures with interaction potential will be also considered.} 
 
For more comprehension, let us briefly explain the  {\it slowing down procedure}   in its original continuous setting  before considering discrete spaces.  Let $R^\gamma$ be the law of a reversible Brownian motion with diffusion coefficient $\gamma >0$ on the set $\Omega$ of continuous paths from $[0,1]$ to $\X=\R^d$. The coefficient $\gamma$ can be also interpreted as a temperature parameter.  The measure $R^\gamma\in\M(\Omega)$ is a Markov measure with infinitesimal operator 
 $L^\gamma=\gamma\Delta$ (where $\Delta$ denotes the Laplacian), and initial reversible measure $dm=dx$, the Lebesgue measure on $\R^d$. 

In all the paper, we use the following notations. For any $t\in[0,1]$,   $X_t$ is  {\it the projection map} \[X_t: \omega\in \Omega\mapsto \omega_t \in \X.\]
Given $Q\in \M(\Omega)$, the measure $Q_t:=X_t\# Q$ on $\X$ denotes the push-forward of the  measure $Q$  by $X_t$, and  for any $0\leq t<s\leq 1$, the measure 
$Q_{s,t}:=(X_s,X_t)\# Q$ on $\X\times \X$ denotes  the push forward of the measure $Q$ by the projection map $(X_s,X_t)$. 
For any integrable function $F:\Omega \to \R$ with respect to $Q$, one notes
\[\E_Q[F]:=\int_\Omega F dQ.\] 
 
 The  result by T. Mikami \cite{Mik04} or C. L\'eonard \cite{Leo12} is the following: for any absolutely continuous measures  $\nu_0,\nu_1\in \Pc_2(\X)$,  for any sequences $(\gamma_\ell )_{\ell\in \N} $ of temperature parameters going to zero,
 \begin{align*}
 W_2^2(\nu_0,\nu_1)&= \inf_{Q\in \Pc(\Omega)} \left\{\E_Q[c]\,\Big| Q_0=\nu_0,Q_1=\nu_1\,\right\}\\
 &=\lim_{\gamma_\ell \to 0}\left[\gamma_\ell   \min_{Q\in \Pc(\Omega)}\Big\{  H(Q|R^{\gamma_\ell })\,\Big|\, Q_0=\nu_0,Q_1=\nu_1 \Big\}\right], 
 \end{align*}
where 
$c(\omega):=\int_0^1 |\!\stackrel{.}{\omega}_t\!|^2 dt,$
 if the path $\omega=(\omega_t)_{t\in[0,1]}$ is absolutely continuous ($\stackrel{.}{\omega}$ denotes its time derivative), and  $c(\omega):=+\infty$ otherwise.
The first equality is known as the Benamou-Brenier formula  (see \cite{BB00}).
The second equality therefore relates $W_2$ to the so-called  {\it dynamic Schr\"odinger minimization problems}. As a convex minimization problem, 
for any fixed $\gamma>0$, it admits a  single minimizer $\widehat Q^{\gamma}$, namely
  \begin{equation}\label{SP}
  \min_{Q\in \Pc(\Omega)}\Big\{ H(Q|R^{\gamma})\,\Big|\, Q_0=\nu_0,Q_1=\nu_1 \Big\}=H(\widehat Q^{\gamma}|R^{\gamma}).
  \end{equation}
As interpretation, the measure $\widehat Q^{\gamma}$ is the law of the process with configuration $\widehat Q^{\gamma}_0=\nu_0$ at time $t=0$ and $\widehat  Q^{\gamma}_1=\nu_1$ at time $t=1$, which is the closest in some entropic meaning, to a reversible Brownian motion with diffusion coefficient $\gamma$. 
  As a result (see \cite{Mik04,Leo12}), the sequence of minimizers $(\widehat Q^{\gamma_\ell })_{\ell\in \N}$ converges to a single measure $\widehat Q\in\Pc(\Omega)$. For any $t\in [0,1]$, let $Q_t^\gamma:=\widehat Q^\gamma_t$ and $\nu_t:= \widehat Q_t$. By definition, $(Q_t^\gamma)_{t\in [0,1]}$ is {\it a Schr\"odinger bridge from $\nu_0$ to $\nu_1$ at fixed temperature $\gamma$},  and  as a main result, as $\gamma_\ell $ goes to zero, the limit path $(\nu_t)_{t\in [0,1]}$, is  a $W_2$-geodesic from $\nu_0$ to $\nu_1$ (see \cite{Leo13}). Therefore, it is natural to consider a relaxation of the curvature definition \eqref{deplace} by replacing the geodesic $(\nu_t)_{t\in[0,1]} $ by the bridge $(Q_t^\gamma)_{t\in[0,1]} $ and by replacing  $W_2^2(\nu_0,\nu_1)$ by $\gamma H(\widehat Q^\gamma|R^\gamma)$. This idea has been explored in continuous setting   by G. Conforti in \cite{Con18}.
  
Let us  present  the discrete  analogue of this  approach due to C. L\'eonard \cite{Leo13, Leo17,Leo16}. From now on, the space $\X$ is  a countable set endowed with the $\sigma$-algebra generated by  singletons. The set $\Omega\subset \X^{[0,1]}$ denotes the space of all left-limited, right-continuous, piecewise constant paths $\omega= (\omega_t)_{t\in [0,1]}$ on $\X$, with finitely many jumps. The space $\Omega$ is endowed with the $\sigma$-algebra $\mathcal F$ generated by the cylindrical sets. In all the paper, by convention, a sum indexed by an empty set is equal to zero.

According to C. L\'eonard's paper \cite{Leo16}, the discrete space $\X$ is equipped with a metric distance $d$. 
This distance is assumed to be {\it positively lower bounded}: for all $x\neq y$ in $\X$, $d(x,y)\geq 1$.  
The space $\X$ is also  the set of vertices of a connected graph $G=(\X,E)$ where $E\subset \X\times \X$ denotes the set of directed edges of the graph. $G$ is supposed to be an undirected graph so that for all $(x,y)\in E$, one has $(y,x)\in E$. Two vertices $x$ and $y$ are {\it neighbours} and we write $x\sim y$  if $(x,y)\in E$. We assume that any vertex $x\in \X$ has a finite number of neighbours $d_x$ and that $\sup_{x\in \X} d_x=d_{\max}<\infty$.  We note $V(x)$ the set of neighbours of $x$. 
 The  length $\ell(\omega)$ of a  piecewise constant  path $\omega=(\omega_t)_{t\in [0,1]}\in \Omega$ is given by
\[\ell(\omega):= \sum_{0<t<1} d(\omega_{t^-},\omega_t).\]
In  C. L\'eonard's paper, the distance is  assumed to be {\it intrinsic in the discrete sense} (see \cite[Hypothesis 2.1]{Leo16}), this means that 
for  any $x,y\in \X$, 
\[d(x,y):=\inf\Big\{\ell(\omega)\,\big|\, \omega\in \Omega, \omega_0=x,\omega_1=y\Big\}.\]
In this paper, we  only consider the simple case where $d=d_\sim$ is  the graph distance for which the above assumptions are fulfilled: $d_\sim(x,y)=1$ if and only if $x\sim y$. 
  
  A {\it discrete path $\alpha$ of length $\ell\in \N$} joining two  vertices $x$ and $y$ is a sequence of $\ell+1$ neighbours $\alpha=(z_0,\ldots,z_\ell)$   so that $z_0=x$ and $z_\ell=y$. In the sequel, we note $z\in\alpha$ if there exists $i\in\{0,\ldots, \ell\}$ such that $z=z_i$, and we note $(z,z')\in \alpha$ if there exists $0\leq i<j\leq \ell$ such that $z=z_i$ and $z'=z_j$.  The  distance $d(x,y)$  is also the minimal length of a path joining $x$ and $y$.  A {\it discrete geodesic path} joining $x$ to  $y$ is a path of length $d(x,y)$ from $x$ to $y$.  We note $G(x,y)$  the set of all geodesic paths joining $x$ to  $y$, and we note $[x,y]$ the set of all points that belongs to a geodesic from $x$ to $y$,
\[[x,y]=[y,x]=\Big\{z\in \X\,\big|\, z\in \alpha,\alpha\in G(x,y)\Big\}.\]   

At fixed temperature $\gamma>0$, as reference measure on $\Omega$, we consider a Markov path measure $R^\gamma$  with generator $L^\gamma$ defined by  
\begin{eqnarray*}
\left\{  \begin{array}{ll}L^\gamma(x,y):= \gamma^{d(x,y)} L(x,y)\quad \mbox{ for }  x\neq y,
\\  L^\gamma(x,x):=- \sum_{y\in \X, y\neq x} L^\gamma(x,y), 
\end{array}\right.
\end{eqnarray*}
and initial {\it reversible invariante} measure  $R_0^\gamma=m$. More precisely, we assume that $m$ is reversible with respect to $L$, which means that  for any $x,y\in \X$
\[ m(x)L(x,y)=m(y) L(y,x). \]
It implies that $m$ is reversible with respect to $L^\gamma$ for any $\gamma>0$, and therefore $R_t^\gamma=m$ for all $t\in [0,1]$.
We also assume that the Markov process is irreducible so that $m(x)>0$ for all $x\in \X$.
Recall that from the definition of a generator,  for any $t\geq 0$ and any  $x,y\in \X$, one has
\[R^\gamma_{t, t+h}(x,y)=R^\gamma_t(x) (\delta_x(y)+L^\gamma(x,y)h +o(h)),\]
where $\delta_x$ is the Dirac measure at point $x$.
We note  $\P_t, t\geq 0,$  the Markov  semi-group associated to $L$,  and $\P_t^\gamma, t\geq 0,$ the Markov semi-group associated to $L^\gamma, \gamma>0$. By reversibility, one has  for any $x,y\in \X$
 \[R_{0,t}^\gamma(x,y)=m(x) \P_t^\gamma (x,y)=m(y)\P_t^\gamma(y,x),\]
 and since the process is irreducible, $P_t^\gamma(x,y)>0$ for all $t>0$ and all $x,y\in \X$.
For any  integrable function $f:\X\to\R$ with respect to $\P_t^\gamma(x,\cdot)$,
 we set 
 \[\P_t^\gamma f(x):=\sum_{y\in \X} f(y) \,\P_t^\gamma(x,y).\]
In this paper we only consider  generator $L$ satisfying :
 \begin{equation}\label{hypL}
 L(x,y)>0 \quad\mbox{ if and only if }\quad  x\sim y,
 \end{equation}
 so that $P_t^\gamma=P_{\gamma t}$ for all $\gamma,t>0$, but also for any $x\neq y$,
 \begin{eqnarray*}
d(x,y)=\min\left \{k\in \N\,\big| \, L^k(x,y)>0 \right\}.
\end{eqnarray*}

Let  $\nu_0,\nu_1\in \Pc(\X)$  with  respective densities $h_0$ and $h_1$ according  to $m$.
In L\'eonard's paper \cite{Leo16}, Theorem 2.1  ensures that under some assumptions (see \cite[Hypothesis 2.1]{Leo16}), at fixed temperature $\gamma>0$,
the minimum value of the dynamic Schr\"odinger problem \eqref{SP} is reached for a single  probability measure $\widehat Q^\gamma$  which is Markov.  This Markov property implies that the measure $\widehat Q^\gamma$ has density $f^\gamma(X_0)g^\gamma(X_1)$ with respect to $R^\gamma$, where $f^\gamma$ and $g^\gamma$ are non-negative  functions
 on 
$\X$ satisfying the following so-called {\it Schr\"odinger system}
\begin{eqnarray} \label{SS}
 \left\{\begin{array}{ll} f^\gamma(x)\,\P_1^\gamma g^\gamma(x)&= h_0(x),
 \vspace{0,1 cm}\\
g^\gamma(y)\,\P_1^\gamma f^\gamma(y)&= h_1(y),\end{array}\right.\qquad \forall x,y\in\X.
\end{eqnarray}
Since $f^\gamma$ is non-negative and $f^\gamma\neq 0$,  by irreducibility   one has $\P_t^\gamma f^\gamma>0$ for all $t>0 $, and  for the same reason, $\P_t^\gamma g^\gamma>0$ for all $t>0 $.
As a consequence, if $\nu_0$ and $\nu_1$ have finite support, then the Schr\"odinger system \eqref{SS}  
implies that  $f^\gamma$ and $g^\gamma$ have also finite support.

According to \cite[Theorem 6.1.4.]{Leo17}, from the Markov property,  the law at time $t$ of the Schr\"odinger bridge at fixed temperature $\gamma$, $\widehat Q^\gamma_t$, is given by: for any $z\in\X$,
\begin{equation}\label{Qtstructure}
\widehat Q_t^\gamma(z)= \P^\gamma_t f^\gamma(z) \P^\gamma_{1-t}g^\gamma(z) m(z)=\sum_{x,y\in \X} m(z)\P_t^\gamma(z,x) \P_{1-t}^\gamma(z,y)f^\gamma(x) g^\gamma(y).
\end{equation}
Let us present  another expression for $\widehat Q_t^\gamma$. First, by reversibility, one has 
\[\sum_{z\in \X}  m(z)\P^\gamma_t(z,x) \P^\gamma_{1-t}(z,y)
= m(x) \P^\gamma_1(x,y)=R^\gamma_{0,1}(x,y).\]
Therefore,  setting 
\begin{equation}\label{definut}
{Q_t^\gamma}^{ x,y}(z):= \frac{m(z)\P^\gamma_t(z,x) \P^\gamma_{1-t}(z,y)}{m(x)\P_1^\gamma(x,y)}= \frac{\P^\gamma_t(x,z) \P^\gamma_{1-t}(z,y)}{\P^\gamma_1(x,y)}= \frac{\P^\gamma_{1-t}(y,z) \P^\gamma_{t}(z,x)}{\P^\gamma_1(y,x)},
\end{equation}
and 
\[\widehat\pi^\gamma(x,y):=\widehat Q^\gamma_{0,1}(x,y)=   R^\gamma_{0,1}(x,y) f^\gamma(x) g^\gamma(y),\]
we get for any $z\in\X$,
\begin{equation}\label{expressQ_tgamma}
\widehat Q_t^\gamma(z)= \iint {Q_t^\gamma}^{x,y}(z)\,\, d\widehat \pi^\gamma(x,y).
\end{equation}
Actually, for any $x,y\in \X$,  $({Q_t^\gamma}^{x,y})_{t\in[0,1]}$ is the Schr\"odinger bridge joining the Dirac measures $\delta_x$ and $\delta_y$. The path  $(\widehat Q^\gamma_t)_{[0,1]}$ is therefore a mixing of these Schr\"odinger bridges, according to the coupling measure $\widehat \pi^\gamma\in \Pi(\nu_0,\nu_1)$.
 
  Using the Schr\"odinger system \eqref{SS},  the measure $\widehat \pi^\gamma$ can be rewritten as follows,
 \[\widehat \pi^\gamma(x,y)= \nu_0(x) \,\frac{g^\gamma(y) \P_1^\gamma(x,y)}{\P_1^\gamma g^\gamma(x)}= \nu_1(y) \, \frac{f^\gamma(x) \P_1^\gamma(y,x)}{\P_1^\gamma f^\gamma(y)}.\]
 For any $\nu\in\Pc(\X)$, let $\supp(\nu)$ denote the support of the measure $\nu$, $\supp(\nu):=\{x\in \X\,|\, \nu(x)>0\}$.
The measure  $\widehat \pi^\gamma$ admits the  following decomposition, 
 \[ \widehat\pi^\gamma(x,y)=\nu_0(x)\,\widehat{\pi}^\gamma_{_\rightarrow}(y|x)= \nu_1(y)\,\widehat{\pi}^\gamma_{_\leftarrow}(x|y),\]
 where $\widehat{\pi}^\gamma_{_\rightarrow}$ and $\widehat{\pi}^\gamma_{_\leftarrow}$ are the Markov kernel defined by,  for any $x\in \supp(\nu_0)$, 
 \begin{eqnarray*}
 \widehat{\pi}^\gamma_{_\rightarrow}(y|x):=\frac{g^\gamma(y) \P_1^\gamma(x,y)}{\P_1^\gamma g(x)},
 \end{eqnarray*}
 and for any   $y\in \supp(\nu_1)$,  
 \begin{eqnarray}\label{defnoyaubis}
 \widehat{\pi}^\gamma_{_\leftarrow}(x|y):=\frac{f^\gamma(x) \P_1^\gamma(y,x)}{\P_1^\gamma f^\gamma(y)}.
 \end{eqnarray}
 
In order to fulfill this presentation, recall that the {\it static Schr\"odinger minimization problem} associated to $R_{0,1}^\gamma$ is to  find the minimum value of 
$H(\pi |R^\gamma_{0,1})$ over all  $\pi\in \Pi(\nu_0,\nu_1)$.
Theorem 2.1. by C. L\'eonard \cite{Leo16} ensures that under Hypothesis 2.1 of its  paper,  this minimum value is the same as the one of the dynamic Schr\"odinger minimization problem. Moreover it  is reached for $\widehat \pi^\gamma=\widehat Q^\gamma_{0,1}\in \Pc(\X\times\X)$  and therefore
\[
\inf_{\pi\in\Pi(\nu_0,\nu_1)}H(\pi |R^\gamma_{0,1})=H(\widehat \pi^\gamma|R^\gamma_{0,1})
=H(\widehat Q^\gamma|R^\gamma).\]

The main goal of this paper is to prove a convexity property 
for the function $t\in[0,1]\mapsto H(\widehat Q_t|m)$ by applying the  {\it slowing down procedure}. Our  strategy  is first to differentiate twice  at positive temperature $\gamma>0$ the function  $t\in[0,1]\mapsto H(\widehat Q^\gamma_t|m)$
using backward equations for the Markov process. Then as a main contribution of this paper, we analyse the behavior of the second derivative of this functions as the temperature $\gamma$ goes to zero (see Theorem 3.5). Considering different examples of graphs, any lower bound of this limit second derivative gives a convexity property of type \eqref{deplacebis}.

We want this strategy to hold for a large class of graphs $(\X, d,m, L)$, with  possibly  infinite set of vertices $\X$. Mainly in order to justify the lower bounds  on the second derivative as $\gamma$ goes to zero, we make the following assumptions.  
\begin{itemize}
\item  The measure $m$ is bounded,
\begin{eqnarray}\label{unifbounded0}
\sup_{x\in \X} m(x)<\infty, \qquad \mbox{and} \qquad \inf_{x\in \X} m(x)>0.
\end{eqnarray}
\item  The generator $L$ is uniformly bounded : there exists $S\geq 1$ such that 
\begin{eqnarray}\label{unifbounded1}
\sup_{x\in \X} |L(x,x)|\leq S, 
\end{eqnarray}
and there exists $I\in (0,1]$ such that
\begin{eqnarray}\label{unifbounded2}
\inf_{x,y\in \X, x\sim y} L(x,y)\geq I.
\end{eqnarray}
\item 
For any $x\in \X$, there exists 
$\gamma_o\in(0,1]$ such that 
\begin{eqnarray}\label{convhyp}
\sum_{y\in \X} \gamma_o^{d(x,y)}<\infty.
\end{eqnarray}
\end{itemize}

All these assumptions are obviously satisfy if $\X$ is finite. One may also consider any infinite graph $\X$ with bounded degree $d_{\rm max}$ endowed with the counting measure $m_0$, which is reversible with respect to the generator $L_0$ given by $L_0(x,y)=1$ for $x\sim y$, $L_0(x,x)=-d_x$.
On such graphs $(\X,d,m_0,L_0)$,  a condition dealing with the geometry of balls of radius 2 will be given in a forthcoming paper to get lower bounds on the  $T_2$-entropic curvature.
 
Unfortunately, the above assumptions are not fulfilled by example for  the $M/M/\infty$ process on $\N$ with Poisson stationary measure. For such processes,  the same strategy is expected to provide lower bounds on entropic curvature adapting proofs by the known specific expression of the Markov semi-group. A next challenge is  to weak the assumptions of this paper for other specific classes of processes. 

One of the main assets of Hypothesis  \eqref{unifbounded1} is to provide a simple expression for  the semi-group  $(\P^\gamma_t)_{t\geq 0}$, namely 
\begin{eqnarray}\label{formePt}
\P^\gamma_t:=e^{t\gamma L}= \sum_{k\in\N} \frac{(t\gamma)^k}{k!} L^k.
\end{eqnarray}
From this  expression, on may simply derive a rather expression of Schr\"odinger bridges at zero temperature between Dirac measures. 
Namely, given $x,y\in \X$, as condition \eqref{unifbounded1} holds, Lemma \ref{lemmetech} \ref{item4} gives  the limit of the path $({Q_t^\gamma}^{x,y})_{t\in [0,1]}$  defined by    \eqref{definut}, namely   for any $z\in \X$,
\begin{equation}\label{defnut0}
\lim_{\gamma\to 0} {Q_t^\gamma}^{x,y}(z)=  {Q_t}\!^{x,y}(z):=\1_{[x,y]}(z)\, r(x,z,z,y)\,\B_t^{d(x,y)}(d(x,z)),
\end{equation}
where for any $x,z,v,y\in\X$,
\begin{equation}\label{defr}
 r(x,z,v,y)= \frac{L^{d(x,z)}(x,z) L^{d(v,y)}(v,y)}{L^{d(x,y)}(x,y)},
 \end{equation}
and $\B_t^d$ denotes the binomial law with parameters $t\in [0,1]$ and $d\in \N$ :
\[ \B_t^d(k):= \binom{d}{k}\, t^k(1-t)^{d-k},\quad k\in\{0,\ldots,d\},\]
 with the binomial coefficient $\binom{d}{k}:=\frac{d!}{k!(d-k)!}$.
Obviously one has $\displaystyle {Q_0}\!^{x,y}=\delta_x$ and $\displaystyle {Q_1}\!^{ x,y}=\delta_y$ . Moreover, observe that  for any $t\in(0,1)$, the support of ${Q_t}\!^{ x,y}$ is $[x,y]$, the set of points  on  discrete geodesics from $x$ to $y$. Observe that this limit Schr\"odinger bridge $({Q_t}\!^{ x,y})_{t\in [0,1]}$ is  consistent with  the metric graph structure. This is not surprising. As the temperature $\gamma$ decreases to zero,  the jumps of the Markov process are less frequent, and the reference  process   is therefore  a {\it  lazy random walk } according to C. L\'eonard's terminology. Roughly speaking, ${Q_t}\!^{ x,y}$ can be interpreted as  the law of a  process which is forced to go from $x$ at time 0 to $y$ at time 1 and  that does not want to move or to jump too much between time 0 and   1. Therefore this process  follows the geodesics of the graph from $x$ to $y$.  

For a better understanding, the law ${Q_t}\!^{ x,y}$  on $[x,y]$ can be described as follows. Let $N_t$ denote a binomial random variable with parameters $t\in[0,1]$ and $d=d(x,y)\in \N$, and let $\Gamma$ be a random discrete geodesic in $G(x,y)$ whose law is given by 
\[\Pb(\Gamma=\alpha)=\frac{L(\alpha_0,\alpha_1)\cdots L(\alpha_{d-1},\alpha_d)}{L^{d(x,y)}(x,y)}, \qquad \mbox{for all } \alpha=(\alpha_0,\alpha_1, \ldots,\alpha_d)\in G(x,y). \]
  If $N_t$ and $\Gamma=(\Gamma_0, \ldots,\Gamma_d)$ are independent then $ {Q_t}\!^{ x,y}$ is the law of $\Gamma_{N_t}$.
  
    Let us come back to the behavior of the Schr\"odinger bridges $(\widehat Q^\gamma_t)_{t\in[0,1]}$ as $\gamma$ goes to zero. 
Assume  $\nu_0$ and $\nu_1$ have finite support. C. L\'eonard \cite[Theorem 2.1]{Leo16} proves  that given a positive sequence $(\gamma_\ell )_{\ell\in \N}$  with $\lim_{\ell \to \infty } \gamma_\ell  =0$, the sequence of optimal Schr\"odinger minimizers $(\widehat Q^{\gamma_\ell })_{\ell\gamma_\ell \in \N}$  converges to a single probability measure $\widehat Q\in {\mathcal P}(\Omega)$  for the narrow convergence, provided Hypothesis 2.1 holds. In this paper, the measure $\widehat Q$ is named as {\it the limit  Schr\"odinger problem optimizer at zero temperature, between $\nu_0$ and $\nu_1$}.
     In the framework  of this work, choosing two  probability measures  $\nu_0$ and $\nu_1$ with finite  supports, Hypothesis 2.1 in \cite{Leo16} is reduced to the following assumption (see condition $(\mu)$ in  Hypothesis 2.1): for any $x,y \in \X$  and for any  $\gamma>0$  
 \begin{equation*}
 \E_{R^\gamma}\left[\ell \,|\,X_0=x,X_1=y\right]<\infty .
 \end{equation*}
 According to Lemma \ref{lemmetech} (vi),  this assumption is fulfilled thanks to \eqref{unifbounded1}  since $\P_1^\gamma(x,y)>0$ for any $x,y\in \X$ and $\gamma>0$.

  As a main result of \cite[Theorem 2.1]{Leo16}, the measure $\widehat Q$ is also a solution of the following {\it dynamic Monge-Kantorovich problem} :
  \[ \inf\Big\{\E_{ Q}[\ell]\,\big|\, Q\in {\mathcal P}(\Omega), Q_0=\mu_0, Q_1=\mu_1\Big\}=\E_{\widehat Q}[\ell].\]
 The sequence of coupling measures $(\widehat\pi^{\gamma_\ell })_{\ell\in\N}$ also weakly converges to \begin{equation*}
 \widehat\pi:= \widehat Q_{0,1},
 \end{equation*}
 and similarly to the continuous case,   $\widehat \pi$ is a {\it $W_1$-optimal coupling} of $\nu_0$ and $\nu_1$.
  
 The weak convergence of  $(\widehat Q^{\gamma_\ell })_{\ell\in\N}$ to $\widehat Q$ also provides the convergence of $(\widehat Q^{\gamma_\ell }_t)_{\ell\in \N}$ to $\widehat Q_t$, and \eqref{expressQ_tgamma} implies \eqref{limitbridge}. 
According to its construction, this bridge is called {\it Schr\"odinger bridge at zero temperature} from $\nu_0$ to $\nu_1$. Observe that for any $t\in(0,1)$, the support of $\widehat Q_t$  only depends on the support of the optimal coupling $\widehat \pi$ of $\nu_0$ and $\nu_1$,
 \begin{equation}\label{suppQt}
 \supp(\widehat Q_t)=\bigcup_{(x,y)\in \supp(\widehat  \pi)} [x,y].
 \end{equation}
 As a main result,  C. Leonard proves that with hypothesis \eqref{hypL}, the path $(\widehat Q_t)_{t\in[0,1]}$ is a constant speed $W_1$-geodesic  (see \cite[Theorem 3.15]{Leo16}): for any $0\leq s\leq t\leq 1$,
 \[W_1\big(\widehat Q_t,\widehat Q_s\big)=(t-s) W_1(\nu_0,\nu_1).\]
 Actually, from the above interpretation of the measure ${Q_t}\!^{ x,y}$ as the law of $\Gamma_{N_t}$ where $\Gamma$ is a random geodesic from $x$ to $y$, independent of a binomial random  variable  $N_t$ with parameters $t\in[0,1]$ and $d(x,y)$, one proves that any bridge $(\widehat Q_t)_{t\in[0,1]}$ defined by \eqref{limitbridge} is a $W_1$-geodesic, as soon as $\widehat  \pi$ is a $W_1$-optimal coupling of $\nu_0$ and $\nu_1$. The proof of this result is the same as  the one of  \cite[Proposition 2.2]{GRST14}.

\section{Main results : examples of entropic curvature bounds along Schr\"odinger bridges on graphs}
The main purpose of  this section is to present    $W_1$ or $T_2 $-entropic curvature  bounds  for several  discrete graph spaces $(\X, d,m,L)$ in the framework of the first  section. As explained  before, these bounds  follows from $C$-displacement convexity  properties \eqref{deplacebis} of the relative entropy along  {\it Schr\"odinger bridges at zero temperature} $(\widehat Q_t)_{t\in[0,1]}$,  derived from the slowing down procedure. 

As in the paper \cite{GRST14},  $C$-displacement convexity  properties  imply a wide range of functional inequalities for the measure $m$ on $\X$, such as   Pr\'ekopa-Leindler type of inequalities, transport-entropy inequalities, and also discrete Poincar\'e or  modified log-Sobolev inequalities. 

As mentioned before, our  approach is efficient to reach new transport-entropy inequalities, transport cost well suited to get new concentration properties, using known connections between transport-entropy inequalities and  concentration properties  pushed forward in \cite{GRST17}.  Although  Erbar-Maas method  does not allow  to recover such concentration results  on graphs, both approaches  imply bounds on the so-called subgaussian constant $\sigma^2(\X)$ of the graph (see \cite{BHT06}), namely $\sigma^2(\X)\leq 1/K$ if the $W_1$-entropic curvature is bounded from below by $K>0$.     
 
As a guideline for other graphs, connexions between $C$-displacement convexity  properties along Schr\"odinger bridges at zero temperature and modified log-Sobolev inequalities are explained only  in the case of the discrete hypercube or the Bernoulli-Laplace Model (see comments $(d)$ after Theorem \ref{thmcube} and after Theorem \ref{thmslicecube}). 
Even if this global  strategy does not allow to recover exactly some known modified log-Sobolev inequality for the Bernoulli-Laplace model, preliminary computations look  promising to apply it  for measures on graphs with interaction potentials. A challenge is to improve it for that purpose. 

New Pr\'ekopa-Leindler type of inequalities are also a straightforward dual consequence of  the $C$-displacement convexity properties \eqref{deplacebis}. Here is
a general statement that applies for each of the discrete spaces $(\X,d,m,L)$ studied in this paper and presented next.

\begin{theorem}\label{prek} On a discrete space  $(\X,d,m,L)$, 
assume that the relative entropy  satisfies a $C$-displacement convexity property (see Definition \ref{defcourb}) with  $C=(C_t)_{t\in (0,1)}$ given by : for any $\nu_0,\nu_1\in {\Pc_b}(\X)$ 
 \[C_t(\widehat \pi) = \iint c_t(x,y) \,d\widehat \pi(x,y),\]
where $ \widehat\pi=\widehat Q_{01}$, and  $\widehat Q$ is the limit  Schr\"odinger problem optimizer between $\nu_0$ and $\nu_1$. Then, the next property holds for all $t\in(0,1)$. 
If  $f,g,h$ are functions on $\X$ satisfying 
\[(1-t) f(x)+t g(y)\leq \int h \,d{Q_t}\!^{ x,y} +\frac{t(1-t)}2 \,c_t(x,y),\qquad \forall x,y\in \X,\]
then 
\[\left( \int e^f\, dm\right)^{1-t}\left( \int e^g\, dm\right)^{t}\leq \int e^h\, dm.\]
\end{theorem}
The proof of this result is an easy adaptation of the one of Theorem 6.3 in \cite{GRST14}. It is left to the reader. 

Following  the paper \cite[section 3.2]{GRST17}, a  tensorization property holds for the $C$-displacement property by using  Knothe-Rosenblatt couplings. Let $(\X_i,d_i,m_i,L_i)$, $i\in[n]:=\{1,\ldots,n\}$, be $n$  graphs satisfying the assumptions of the paper \eqref{unifbounded0}-\eqref{convhyp}.  Let $(\X,d,m,L)$ be the product graph space defined by
$\X:=\X_1\times\cdots\times \X_n$, $m:=m_i\otimes\cdots\otimes m_n$, and  for all $x=(x_1,\ldots, x_n)\in \X$, $y=(y_1,\ldots, y_n)\in \X$,
\[d(x,y):=\sum_{i=1}^n d_i(x_i,y_i).\]
If each measure $m_i$ is reversible with respect to $L_i$, then the product measure $m$ is reversible with respect to the generator \[L:=L_1\oplus\cdots\oplus L_n.\] Namely $L$ is defined by  $L(x,y)=0$ if $d(x,y)\geq 2$, $L(x,x)=-\sum_{y\in\X, y\neq x} L(x,y)$, and for $d(x,y)=1$,  if   $i\in [n]$ is the index for which $d_i(x_i,y_i)=1$ (and $x_j=y_j$ for all $j\neq i$), then
\[ L(x,y)=L_i(x_i,y_i).\]
The Markov semi-group $(P_t)_{t\geq 0}$ associated to $L$ has a product structure, for any $x,y\in \X$, for any $t\geq 0$,
\[P_t(x,y)=P_{1,t}(x_1,y_1)\cdots P_{n,t}(x_n,y_n),\]
where $(P_{i,t})_{t\geq 0}$ denotes the semi-group associated to the generator $L_i$ on $\X_i$, $i\in\{1,\ldots,n\}$.
By construction, it follows that the Schr\"odinger bridge at zero temperature between the Dirac measures $\delta_x$ and $\delta_y$ is a  product of Schr\"odinger bridges at zero temperature between the Dirac measures $\delta_{x_i}$ and $\delta_{y_i}$ on $\X_i$, namely for any $z=(z_1,\ldots,z_n)\in \X$
\begin{equation}\label{prodschro}
{Q_t}\!^{ x,y}(z)={Q_t}\!^{x_1,y_1}(z_1)\cdots{Q_t}\!^{x_n,y_n}(z_n).
\end{equation}
This can be also derived from the geometric structure of the graph.  
 Since any discrete geodesic from $x$ to $y$ is made of $d_i(x_i,y_i)$ jumps for the $i$'s coordinates picked from a  discrete geodesic from $x_i$ to $y_i$ on $\X_i$, one has for $x\neq y$, 
 \[ L^{d(x,y)}(x,y)=\binom{d(x,y)}{d_1(x_1,y_1),\ldots,d_n(x_n,y_n)}\, L_1^{d(x_1,y_1)}(x_1,y_1)\cdots L_n^{d(x_n,y_n)}(x_n,y_n),\] 
 where for any integers $d, k_1,\ldots ,k_n$ such that $d=k_1+\cdot +k_n$, $\binom{d }{k_1,\ldots,k_n}:=\frac{d!}{k_1!\cdots k_n!}$ is the multinomial coefficient. 
The identity \eqref{prodschro} then easily follows.

Using the notations of the paper \cite{GRST17}, any measures $\nu_0,\nu_1 \in \Pc(\X)$
admit the following disintegration formulas: for all $x = (x_1, . . . , x_n ), y = (y_1, . . . , y_n ) \in \X$,
\[\nu_0(x) = \nu_0^1(x_1) \,\nu_0^2(x_2|x_1) \,\nu_0^3(x_3|x_1,x_2)\cdots  \nu_0^n(x_n|x_1,...,x_{n-1}), \]
\[\nu_1(y) = \nu_1^1(y_1) \,\nu_1^2(y_2|y_1)\, \nu_1^3(y_3|y_1,y_2)\cdots  \nu_0^n(y_n|y_1,...,y_{n-1}),\]
with $\nu_0^1, \nu_1^1\in\Pc(\X_1)$ and for any $i\in\{2,\ldots,n\}$, $\nu_0^i(\,\cdot\,|x_1,...,x_{i-1}),\nu_1^i(\,\cdot\,|y_1,...,y_{i-1})\in \Pc(\X_i)$.
For $i\in[n]$, let  $\pi^i(\,\cdot\,|x_1,...,x_{i-1},y_1,...,y_{i-1})\in\Pc(\X_i^2)$ be 
 a coupling of $\nu_0^i(\,\cdot\,|x_1,...,x_{i-1})$ and $\nu_1^i(\,\cdot\,|y_1,...,y_{i-1})$. Then, the Knothe-Rosenblatt coupling $\pi^{(n)}$ of $\nu_0$ and $\nu_1$ associated to the collection of couplings $\pi_i$'s is defined by 
 \[ \pi^{(n)} ( x , y ) := \pi_ 1 ( x_1 , y_1 ) \,\pi_ 2 ( x_2 , y_2 | x_1 , y_1 ) \cdots \pi_n ( x_n , y_n | x_1 , . . . , x_{n-1} , y_1 , . . . , y_{n_1} ).\]
 One notices  $(Q^{(n)}_t)_{t\in[0,1]}$  the bridge in $\Pc(\X)$ from $Q^{(n)}_0=\nu_0$ to $Q^{(n)}_1=\nu_1$, associated to the coupling $\pi^{(n)}$, defined by
\[Q^{(n)}_t(z)=\iint {Q_t}\!^{ x,y}(z) \,d \pi^{(n)}(x,y),\qquad t\in[0,1].\]  
 
 \begin{theorem}\label{tenso}   Let $(\X_i,d_i,m_i,L_i)$, $i\in[n]$, be a collection of graph spaces. Assume that each space  $(\X_i,d_i,m_i,L_i)$ satisfies a $C_i$-displacement convexity property  with  $C_i=(C_{i,t})_{t\in (0,1)}$. Let 
 $(\X,d,m,L)$ be the product space defined as above. Given $\nu_0,\nu_1\in\Pc_b(\X)$ with their disintegration formulas mentioned above, let $ \pi^{(n)}$ be the Knothe-Rosenblatt coupling  of $\nu_0$ and $\nu_1$, associated the collection of couplings $\widehat \pi_i$'s constructed as follows: $\widehat \pi_1:=\widehat Q^1_{0,1}$ is the projection at time 0 and 1 of $\widehat Q^1$, the limit Schr\"odinger problem optimizer at zero temperature between $\nu_0^1$ and $\nu_1^1$, and for $i\in\{2,\ldots,n\}$ and 
$x_1,...,x_{i-1}, y_1,...,y_{i-1}\in \X$,
\[\widehat \pi_i(\,\cdot\,|x_1,...,x_{i-1}, y_1,...,y_{i-1})=\widehat Q^i_{0,1}(\,\cdot\,|x_1,...,x_{i-1}, y_1,...,y_{i-1})\]
is the projection at time 0 and 1 of $\widehat Q^i(\,\cdot\,|x_1,...,x_{i-1}, y_1,...,y_{i-1})$, the limit Schr\"odinger problem optimizer at zero temperature between  $\nu_0^i(\,\cdot\,|x_1,...,x_{i-1})$ and $\nu_1^i(\,\cdot\,|y_1,...,y_{i-1})$. Then, the product space  $(\X,d,m,L)$ satisfies the following convexity property, for any $\nu_0,\nu_1\in\Pc_b(\X^2)$ and  any $t\in(0,1)$, 
 \begin{eqnarray}\label{deplacebisbis}
H( Q^{(n)}_t|m)\leq (1-t) H(\nu_0|m)+t \,H(\nu_1|m)- \frac{t(1-t)}2C_t(\pi^{(n)}),
\end{eqnarray}
where $(Q^{(n)}_t)_{t\in[0,1]}$ is  the bridge from $\nu_0$ to $\nu_1$ associated to the coupling $\pi^{(n)}$ and 
\[C_t(\pi^{(n)}):= \sum_{i=1}^n \iint C_{i,t}\left(\widehat \pi_i(\,\cdot\,|x_1,...,x_{i-1}, y_1,...,y_{i-1})\right) \, d\pi^{(n)}(x,y).\]
\end{theorem}
The proof of this result is a simple adjustment   of the proof of Theorem 1.1 in \cite{GRST14}, which is left to the reader. 
\begin{remarks}
\begin{itemize}
\item Even if  the $\widehat \pi_i$'s are $W_1$-optimal couplings in $\Pc(\X_i^2)$, there is no reason for $\pi^{(n)}$ to be a $W_1$-optimal coupling of $\nu_0$ and  $\nu_1$ in $\Pc(\X^2)$, and therefore for $(Q^{(n)}_t)_{t\in[0,1]}$ to be a   $W_1$-geodesic.  Therefore, the convexity property \eqref{deplacebisbis} on the product space $(\X,d,m,L)$ slightly differs from the convexity property given by Definition \ref{defcourb}. 
\item One will see on the discrete hypercube $\X=\{0,1\}^n$, that working  directly on the product space provides  convexity  properties that can not be derived from the tensorization property of Theorem \ref{tenso}.
\end{itemize}
\end{remarks}

 Let us now present results for specific discrete spaces $(\X,d,m,L)$. For each of these  spaces, we  describe the Schr\"odinger path at zero temperature and, as a main result, we give a $C$-displacement convexity property \eqref{deplacebis} satisfied by the reversible  measure $m$ by specifying the family of costs  $C=(C_t)_{t\in(0,1)}$. 
The  strategy of  proof of  these  results   is explained  in section \ref{proof}. 

\subsection{The lattice $\Z^n$ endowed with the counting measure}
Let $m$ denote the counting measure on $\X=\Z^n$. The  graph structure on $\Z^n$ is  given by the set of edges
\[E:=\Big\{(z,z+e_i),( z,z-e_i)\,\big|\, z\in \Z^n, i\in [n]\Big\},\]
where $(e_1,\ldots,e_n)$ is the canonical base of $\R^n$. The graph distance is  given by 
\[d(x,y):=\sum_{i=1}^n |y_i-x_i|,\qquad x,y\in \Z^n.\]
The  measure $m$ is reversible with respect to the generator $L$ defined by, 
for any $z\in \Z^n$, for any $i\in[n]$, 
\begin{eqnarray*}
L(z,z+e_i)=L( z,z-e_i)=1,\qquad  L(z,z)=-2n.
\end{eqnarray*}
For any integers $d, k_1,\ldots ,k_n$ such that $d=k_1+\cdot +k_n$, $\binom{d }{k_1,\ldots,k_n}=\frac{d!}{k_1!\cdots k_n!}$ denotes the multinomial coefficient.
Since 
\[
L^{d(x,y)}(x,y)= \# G(x,y)= \binom{d(x,y)}{|y_1-x_1|,\ldots,|y_n-x_n|},
\]
 the Schr\"odinger bridge at zero temperature $(\widehat Q_t)_{t\in[0,1]}$ joining two measures $\nu_0,\nu_1\in \Pc_b(\X)$ is given by \eqref{limitbridge} 
 with, according to \eqref{defnut0},
  \begin{align*}
  {Q_t}\!^{ x,y}(z)&= \1_{[x,y]}(z) \,\frac{\binom{d(x,z)}{|z_1-x_1|,\ldots,|z_n-x_n|}\binom{d(z,y)}{|y_1-z_1|,\ldots,|y_n-z_n|}}{ \binom{d(x,y)}{|y_1-x_1|,\ldots,|y_n-x_n|} } \B_t^{d(x,y)}(d(x,z)) \\
  &= \1_{[x,y]}(z)\,\binom{|y_1-x_1|}{|z_1-x_1|} \cdots \binom{|y_n-x_n|}{|z_n-x_n|}\;t^{d(x,z)} (1-t)^{d(z,y)} ,\qquad z\in \Z^n.
  \end{align*}
  Observe that $({Q_t}\!^{ x,y})_{t\in[0,1]}$ is a binomial interpolation path as in the paper by E. Hillion \cite{Hil14}.

\begin{theorem}\label{thmZ} On the space $(\Z^n,m,d,L)$, the relative entropy $H(\cdot|m)$ satisfies the 0-displacement convexity property \eqref{deplacebis}. In other words, for any  Schr\"odinger bridge at zero temperature $(\widehat Q_t)_{t\in[0,1]}$ joining any two measures $\nu_0,\nu_1\in\Pc_b(\Z^n)$, the map  $t\mapsto H(\widehat Q_t|m)$ is convex.
\end{theorem}

Therefore   the space $(\Z^n,d,m,L)$ has non-negative $W_1$ or $T_2 $-entropic curvature. Actually,  it can not be positive and one may say that $(\Z^n,d,m,L)$ is a flat space.  Indeed, if property \eqref{deplacebis} holds with $C_t(\widehat\pi ^0)=KW^2_1(\nu_0,\nu_1)$, $K>0$, then choosing $\nu_0=\delta_x$ and $\nu_1=\delta_y$ for $x,y\in\X$, one gets for $t=1/2$
\[-\log |[x,y]| = -\log | \supp({Q_{_{1/2}}}\!\!\!\!^{x,y})| =H\big({Q_{_{1/2}}}\!\!\!\!^{x,y}|m\big)\leq -\frac K8 d^2(x,y),\]
where for a finite set $A$, $|A|$ denotes its cardinality. Since $|[x,y]|=\prod_{i=1}^n (|y_i-x_i|+1)$, the last inequality implies for any $x,y\in \Z^n$,
\[\left(\sum_{i=1}^n |y_i-x_i|\right)^2\leq \frac8K \sum_{i=1}^n \log (|y_i-x_i|+1),\]
which is impossible for large values of $|y_i-x_i|$. A similar proof holds replacing $W^2_1(\nu_0,\nu_1)$ by $T_{c_2}(\widehat \pi)$.

The convexity property along  binomial interpolation paths given by Theorem \ref{thmZ}   has been first obtained by E. Hillion \cite{Hil14}. To compare with Hillion's method, the main interest of our approach is its simplicity.
As  explained in the next section, we first work  at positive temperature $\gamma>0$ so that the second derivative of the function $t\mapsto H(\widehat Q^\gamma_t|m)$ can be  easily  computed using  $\Gamma_2$ calculus. Then we analyse the behavior of the second derivative of this function as temperature goes to 0, and   get a nonnegative lower bound at zero temperature on $\Z^n$. This provides the convexity property of $t\mapsto H(\widehat Q_t|m)$.   In Hillion's paper,   one may say that computations are done  directly at zero temperature. It leads to harder computations and  the construction of the optimal coupling,  related to a cyclic monotonicity property,  is  rather difficult to handle. 

In the paper \cite{GRST19} by Gozlan \& al.,  another kind of convexity property of entropy has been proposed that generalizes a new Prekopa-Leindler inequality on $\Z$ by Klartag-Lehec \cite{KL19} (see also the more recent paper \cite{HKS19} by Halikias-Klartag-Slomka).  Their convexity property is of different nature, it is  only valid for $t=1/2$.  More precisely, given $\nu_0,\nu_1\in {\Pc_b}(\Z)$ they define two midpoint measures
\[
\nu_- = {m_-} \# \pi \qquad \text{and}\qquad  \nu_+ = {m_+} \# \pi,
\]
where $\pi$ is the monotone coupling between $\nu_0$ and $\nu_1$ (which is a $W_1$-optimizer), and for all $x,y\in \Z$,
\[m_-(x,y):=\left\lfloor \frac{x+y}{2}\right \rfloor,\qquad m_+(x,y):=\left\lceil \frac{x+y}{2} \right\rceil.\]
 Gozlan \& al. result \cite[Theorem 8]{GRST19} states that 
\begin{equation*}
\frac12 H(\nu_-| m) + \frac12 H(\nu_+| m) \leq \frac12 H(\nu_{0}| m) + \frac12 H(\nu_{1}| m) .
\end{equation*}
As a main  difference, the measures $\nu_+$ and $\nu_-$ are only concentrated on the midpoints $m_-(x,y)$, $m_+(x,y)$, for $x\in \supp(\nu_0)$ and $y\in \supp(\nu_1)$. Since $\nu_+$ and $\nu_-$ are much more concentrated than $\widehat Q_{1/2}$, their result directly implies a Brunn-Minkovsky type of inequality. Unfortunately it seems that their approach do not extend to other values of $t\in(0,1)$.

\subsection{The complete graph} Let $\X$ be a finite set and $\mu$  be any probability measure on  $\X$. The set of edges of the complete graph $G=(\X,E)$ is $E:=\X\times\X\setminus\{(x,x)\,|\, x\in \X\}$ and the 
graph distance is the Hamming distance $d(x,y):= \1_{x\neq y} $ for any   $x,y\in \X$.
The measure $\mu$  is reversible with respect to the generator $L$ given by : for any $z,z'\in \X$ with $z\neq z'$,
\[L(z,z'):=\mu(z'), \qquad L(z,z):=-(1-\mu(z)).\]
The Schr\"odinger bridge at zero temperature $(\widehat Q_t)_{t\in[0,1]}$ given by \eqref{limitbridge}, is the same as the bridge used in \cite{GRST14} for  the complete graph (see section 2.1.1): for any $x,y\in \X$ one has 
\begin{eqnarray}\label{expnu0complete}
  {Q_t}\!^{ x,y}(z)= (1-t)\,\delta_x(z) +t\, \delta_y(z)
 ,\qquad z\in \X,
 \end{eqnarray}
 and therefore $\widehat Q_t=(1-t)\nu_0+t\nu_1$.
 \begin{theorem}\label{thmcomplete} On the finite space $(\X,\mu,d,L)$, the relative entropy $H(\cdot|\mu)$ satisfies the $C$-displacement convexity property \eqref{deplacebis}, with $C=(C_t)_{t\in(0,1)}$ given by:  for any $\nu_0,\nu_1\in{\mathcal P}(\X)$ with associated  limit  Schr\"odinger problem optimizer 
 $\widehat Q\in \Pc(\Omega)$,
 \[
C_t(\widehat \pi):=\int h_t \left(\int \1_{w\neq x}\, d \widehat{\pi}_{_\rightarrow}(w|x)\right)d\nu_0(x)+ \int h_{1-t}\left(\int \1_{w\neq y} \,d \widehat{\pi}_{_\leftarrow}(w|y)\right) d\nu_1(y), 
\]
where   $ \widehat\pi=\widehat Q_{0,1}$, and for any $t\in (0,1)$, $u\geq 0$, 
\[
h_t(u):=\frac{ th(u)-h(tu)}{t(1-t)},\qquad\mbox{with} \qquad h(u)=
\left\{\begin{array}{ll}
2\left[(1-u)\log(1-u)+u\right] &\mbox{ for } \;0\leq u\leq 1,
\\
+\infty &\mbox{ for } \;u>1.
\end{array}\right.
\]
The cost $C_t(\widehat \pi)$ can be compared  with a function of the  total variation distance 
\begin{equation}\label{totvar}
 \|\nu_0-\nu_1\|_{TV}:=2\sup_{A\subset X} |\nu_0(A)-\nu_1(A)|= 2 \inf_{\pi\in \Pi(\nu_0,\nu_1)} \int \1_{x\neq y} d\pi(x,y)= 2 W_1(\nu_0,\nu_1).
 \end{equation}
Namely, one has 
\begin{equation}\label{zut}
C_t(\widehat \pi)\geq (1+W_1(\nu_0,\nu_1)) \, k_t\left(\frac{W_1(\nu_0,\nu_1)}{1+W_1(\nu_0,\nu_1)}\right),
\end{equation}
where for all $v\in [0,1/2]$, 
\begin{equation}\label{med}
k_t(v):=\inf_{\alpha,\beta, 0<\alpha+\beta\leq 1}\left\{ \alpha h_t\left(\frac v\alpha\right)+\beta h_{1-t}\left(\frac v\beta\right)\right\}\geq  \frac{4v^2}{1-v}.
 \end{equation}
\end{theorem}
\begin{comments} 
\begin{enumerate}[label*=(\alph*)]
\item This result is  an improved version of the convexity
 properties of the relative entropy  obtained  by Gozlan \& al. 
 \cite[Proposition 4.1]{GRST14}. Indeed, from  the estimate \eqref{zut} and the inequality \eqref{med} (whose proofs are given at the end of the proof of Theorem \ref{thmcomplete}), one gets 
\begin{equation}\label{lbpins}
C_t(\widehat \pi)\geq 4 W_1(\nu_0,\nu_1)^2=\|\nu_0-\nu_1\|_{TV}^2,
\end{equation}
and from
 the inequality  $h_t(u)\geq  u^2$, for all $u\in[0,1]$, $t\in (0,1)$, it follows that 
\[C_t(\widehat \pi)\geq \widetilde T_2(\nu_0,\nu_1),\]
with 
\begin{equation*}
\qquad\widetilde T_2(\nu_0,\nu_1):= \inf_{\pi\in \Pi(\nu_0,\nu_1)}  \Big[\int  \left(\int \1_{w\neq x}\, d {\pi}_{_\rightarrow}(w|x)\right)^2d\nu_0(x)+ \int \left(\int \1_{w\neq y} \,d {\pi}_{_\leftarrow}(w|y)\right)^2 d\nu_1(y)\Big].
\end{equation*}
These lower bounds on $C_t(\widehat \pi)$ exactly provide the convexity
 properties of Proposition 4.1 \cite{GRST14}. 
 \item \label{commentcomplet} Since $\mu$ is a probability measure, by Jensen's inequality $H(\widehat Q_t|\mu)\geq 0$. Therefore, the displacement convexity property \eqref{deplacebis} together with the bound  \eqref{lbpins} imply the
 well-known Csiszar-Kullback-Pinsker inequality  by optimizing over all $t\in(0,1)$ (see \cite[Remark 4.2]{GRST14}), namely  
 \[  \frac12\|\nu_0-\nu_1\|_{TV}^2\leq \left(\sqrt{H(\nu_0|\mu)}+\sqrt{H(\nu_1|\mu)}\right)^2, \qquad \forall  \nu_0,\nu_1\in {\mathcal P}(\X) .\]
The optimality of the constant $1/2$ on the left-hand side of this inequality gives the optimality of the constant 4 in \eqref{lbpins}. Therefore the $W_1$-entropic curvature of the complete graph is $4$. 

Observe that \eqref{deplacebis} actually provides an improved version of the  Csiszar-Kullback-Pinsker inequality, namely for any $t\in(0,1)$, 
\[  \frac12 (1+W_1(\nu_0,\nu_1)) \, k_t\left(\frac{W_1(\nu_0,\nu_1)}{1+W_1(\nu_0,\nu_1)}\right)  \leq \frac1t \,H(\nu_0|\mu)+\frac1{1-t} \,H(\nu_1|\mu),\qquad \forall  \nu_0,\nu_1\in {\mathcal P}(\X) .\]
 \end{enumerate}
 \end{comments}

\subsection{Product measures on the discrete hypercube}
In this section, the reference space is the discrete hypercube $\X=\{0,1\}^n$ equipped with a product of Bernoulli measures \[\mu=\mu_1\otimes\cdots\otimes \mu_n,\]
with  for any $i\in [n]$, $\mu_i(1)=1-\mu_i(0):=\alpha_i$, $\alpha_i\in(0,1)$.

 For any $z=(z_1,\ldots,z_n)\in\{0,1\}^n$ and any $i\in [n]$ let $\sigma_i(z)$ denotes the neighbour of $z$ according to the $i$'s coordinate defined by  
 \[\sigma_i(z):=(z_1,\ldots , z_{i-1},\overline z_i,   z_{i+1}, \ldots , z_n),\]
 where $\overline{ z_i}:=1-z_i$.
 The  set of edges on $\{0,1\}^n$ is 
 \[E:=\Big \{(z, \sigma_i(z))\,\big|\, z\in \{0,1\}^n, i\in [n]\Big\},\]
 and the graph distance  is the {\it Hamming distance} :
\[d(x,y):=\sum_{i=1}^n \1_{x_i\neq y_i}, \qquad x,y\in \{0,1\}^n.\]

The measure $\mu$ is reversible with respect to the generator $L$ given by: for all $z\in \{0,1\}^n$,
\[L(z, \sigma_i(z)) :=(1-\alpha_i)\,z_i+ \alpha_i\overline{  z_i},\quad \forall i\in[n],\]
and 
$L(z,z):=-\sum_{i=1} ^n L(z, \sigma_i(z)).$ 
Observe that setting 
\[L_i(z_i,\overline{z_i}):=(1-\alpha_i)\,z_i+ \alpha_i\overline{  z_i},\quad z_i\in\{0,1\},\] and 
$L_i(z_i,z_i)=-L_i(z_i,\overline{z_i})$,  
the Bernoulli measure $\mu_i$ is reversible with respect to $L_i$ and one has 
\[L:=L_1\oplus\cdots\oplus L_n.\]
Easy computations give, for any $x,y\in \{0,1\}^n$,
\begin{eqnarray}\label{Ldcube}
L^{d(x,y)}(x,y)= d(x,y)! \prod_{i=1}^n(1-\alpha_i)^{[x_i-y_i]_+}\alpha_i^{[y_i-x_i]_+},
\end{eqnarray} 
and it follows that the Schr\"odinger bridge at zero temperature $(\widehat Q^0_t)_{t\in[0,1]}$ joining two probability measures $\nu_0$ and $\nu_1$ is given by \eqref{limitbridge}, with according to \eqref{defnut0}
\begin{equation}\label{pontxycube}
 {Q_t}\!^{ x,y}(z)=  \1_{[x,y]}(z)\;t^{d(x,z)} (1-t)^{d(z,y)} ,\qquad z\in\{0,1\}^n.\\
 \end{equation}
 This path has exactly  the same structure as the one used in \cite{GRST14} to establish entropic curvature bounds on the product space $(\{0,1\}^n, \mu)$ (see section 2.1.2). 

\begin{theorem}\label{thmcube} Let $\mu=\mu_1\otimes\cdots \otimes \mu_n$ be a product probability measure on  the discrete hypercube $\X=\{0,1\}^n$.
On the space $(\{0,1\}^n,\mu,d,L)$, the relative entropy $H(\cdot|\mu)$ satisfies  the $C$-displacement convexity property \eqref{deplacebis}, with $C=(C_t)_{t\in (0,1)}$ defined by: for any $\nu_0,\nu_1\in{\mathcal P}(\{0,1\}^n)$ with associated  limit  Schr\"odinger problem optimizer $\widehat Q\in \Pc(\Omega)$,
\begin{equation*} 
C_t(\widehat \pi):=\max \Big[ \frac4n W_1^2(\nu_0,\nu_1),\,\frac4n T_{c_2}(\widehat \pi) ,  \widetilde T_t(\widehat \pi)\Big],
\end{equation*}
where $ \widehat\pi=\widehat Q_{0,1}$,
the cost function $c_2$ of $T_{c_2}$  is defined by 
\[c_2(h):=\max\left\{ \frac {h(h-1)}2,
h^2-2h(1+\log h)\1_{h\neq 0}\right\},\quad h\in \N,\]
the cost $\widetilde T_t$ is defined by 
\[\widetilde T_t(\widehat \pi):=\int \sum_{i=1}^n h_t\left(\Pi^i_\rightarrow(x)\right) d\nu_0(x)+\int \sum_{i=1}^n h_{1-t}\left(\Pi^i_\leftarrow(y)\right) d\nu_1(y),\]
with the 
 definition of the   functions $h_t$, $t\in(0,1)$ given  in Theorem \ref{thmcomplete} and setting 
\[\Pi^i_\rightarrow(x):=\int \1_{w_i\neq x_i} d \widehat{\pi}_{_\rightarrow}(w|x),\quad \Pi^i_\leftarrow(y):=\int \1_{w_i\neq y_i} d \widehat{\pi}_{_\leftarrow}(w|y).\; \]
\end{theorem}
\begin{comments} 
\begin{enumerate}[label*=(\alph*)]
\item The first lower bound  $C_t(\widehat \pi) \geq \frac4n W_1^2(\nu_0, \nu_1)^2$    gives the $W_1$-entropic curvature of the discrete hypercube $\{0,1\}^n$ bigger and asymptotically equal to $4/n$ as $n$ goes to infinity. Indeed, 
as in the previous part to recover the Csiszar-Kullback-Pinsker inequality,
 the well-known $W_1$-optimal transport-entropy inequality  on the discrete hypercube for product probability measures is a consequence the displacement convexity property \eqref{deplacebis}, using 
$H(\widehat Q_t|\mu)\geq 0$ and optimizing over all $t\in(0,1)$. Namely, one has  
\[\frac2n W_1^2(\nu_0, \nu_1)\leq \left(\sqrt{H(\nu_0|\mu)}+\sqrt{H(\nu_1|\mu)}\right)^2,\qquad \forall  \nu_0,\nu_1\in {\mathcal P}(\{0,1\}^n).\]
From the central limit Theorem, the constant $2/n$ (related to the subgaussian constant of the space as mentioned before)  is known to be asymptotically optimal as $n$ goes to infinity.

\item 
 The second lower bound $C_t(\widehat \pi)\geq \frac4n \,T_{c_2}(\widehat{\pi})$ can not be derived from a tensorisation property such as in Theorem \ref{tenso}. Indeed, for $n=1$, on the two points space, one has $T_{c_2}(\widehat{\pi})=0$. Therefore,  the Schr\"odinger approach allows to capture a property of the hypercube that can not be derived from a  tensorisation property as it is often the case.  
 
 This second  lower bound also gives a new kind of curved Pr\'ekopa-Lindler inequality on the discrete hypercube by applying Theorem \ref{prek}. 
It also implies the following new transport-entropy inequality on the discrete hypercube,  for any $\nu_0,\nu_1\in \Pc(\{0,1\}^n)$,   
\begin{equation}\label{talcub} 
\frac2n T_{c_2}(\nu_0,\nu_1) \leq \left(\sqrt{H(\nu_0|\mu)}+\sqrt{H(\nu_1|\mu)}\right)^2.
\end{equation}
As opposed to Marton's transport inequality or to $W_2$-Talagrand's transport inequality on  Euclidean space,   inequality \eqref{talcub} on the  hypercube  does not tensorize.
Nevertheless, it can be interpreted as a discrete analogue on the hypercube of the $W_2$-Talagrand's transport inequality. 
Indeed, from \eqref{talcub}, applying the central limit theorem, one exactly  recovers the well-known $W_2$-transport entropy inequality for the standard Gaussian probability measure $\gamma$ on $\R$, due to Talagrand \cite{Tal96c}. Namely, one has for any absolutely continuous  probability measure $\nu\in\Pc_2(\R)$,
\begin{eqnarray}\label{tal}
W_2^2(\nu,\gamma)\leq 2 H(\nu|\gamma).
\end{eqnarray}
For a sake of completeness, the proof of this implication is given in Appendix A (see Lemma \ref{cubegauss}). 
As a byproduct of this observation, since the constant 2 is optimal in Talagrand's inequality \eqref{tal}, the constant $2/n$ in \eqref{talcub} and the constant $4/n$ in $C_t(\widehat \pi)\geq \frac4n \,T_{c_2}(\widehat{\pi})$ are also asymptotically optimal in $n$. Therefore the $T_2$-entropic curvature of the discrete hypercube is asymptotically equivalent to $4/n$ as $n$ goes to infinity. 

Actually, according to the proof of Theorem \ref{thmcube},  for each fixed $t\in(0,1)$,  the cost function $c_2$ can be improved, one has 
 \[C_t(\widehat \pi)\geq\frac4n\iint w_t(d(x,y))\, d\widehat \pi(x,y),\] 
where for any $d\in \N$
\[w_t(d):=\max\left\{\frac{d(d-1)}2,\int_0^1 v_s(d) \,K_t(s)\,ds\right\}\geq c_2(d),\]
with 
\begin{equation*}
  v_t(d)
:=\frac14\left(\sum_{k=0}^{d} \sqrt{k(k-1)}\, \Big(\frac{\B_t^{d} (k)}t+\frac{\B_{1-t}^{d} (k)}{1-t}\Big) \right)^2. 
\end{equation*}

\item  The inequality  $h_t(u)\geq  u^2$, for all $u\in[0,1]$, $t\in (0,1)$ gives 
$\widetilde T_t(\widehat \pi) \geq \widetilde T_2(\nu_0,\nu_1)$ with
\begin{equation}\label{t2tilde}
 \widetilde T_2(\nu_0,\nu_1):= \inf_{\pi\in \Pi(\nu_0,\nu_1)} \Big[\int \sum_{i=1}^n \left(\int \1_{w_i\neq x_i} d {\pi}_{_\rightarrow}(w|x)\right)^2 d\nu_0(x)
+ \int \sum_{i=1}^n \left(\int \1_{w_i\neq y_i} d {\pi}_{_\leftarrow}(w|y)\right)^2 d\nu_1(y)\Big],
\end{equation}
So, from the third lower bound $\widetilde T_t(\widehat \pi)$ of  $C_t(\widehat \pi)$, one recovers a similar convexity property as the one obtained for the discrete cube in \cite[Corollary 4.4]{GRST14}.
The only difference is  the expression \eqref{limitbridge} of the path $(\widehat Q_t)_{t\in[0,1]}$, the coupling measure $\widehat \pi$  is  replaced by an optimal Knothe-Rosenblatt coupling. 

The following symmetric version of  Marton's transport entropy inequality on the discrete hypercube is a consequence of the last lower bound on $C_t(\widehat \pi)$: 
for any $\nu_0,\nu_1\in \Pc(\{0,1\}^n)$,
\[\frac12 \widetilde T_2(\nu_0,\nu_1)\leq \left(\sqrt{H(\nu_0|\mu)}+\sqrt{H(\nu_1|\mu)}\right)^2.\]

\item \label{commentlogsob} The lower bound $\widetilde T_t(\widehat \pi)$ is also well adapted to recover  modified logarithmic Sobolev inequality on the discrete hypercube as $t$ goes to 0. Assume $\nu_0$ is a probability measure with positive density $f$. 
Observe first that 
 \[\lim_{t\to 0} \widetilde T_t(\widehat \pi)=\int \sum_{i=1}^n h\left(\Pi^i_\rightarrow(x)\right) d\nu_0(x)+\int \sum_{i=1}^n h_{1}\left(\Pi^i_\leftarrow(y)\right) d\nu_1(y),\]
where for $u\in[0,1)$, $h_1(u):=\lim_{t\to 1}  h_t(u)=uh'(u)-h(u)=2(-u-\log(1-u)) $.
For any real function $g$ on $\{0,1\}^n$, let us note \[ D_ig(x):=g(\sigma_i(x))-g(x), \qquad x\in \{0,1\}^n.\]
Applying Lemma \ref{logsobcube}, since $\Pi^i_\rightarrow(x)=\Pi^{\sigma_i(x)}_\rightarrow(x)$,  the convexity property  \eqref{deplacebis}  with $C_t=\widetilde T_t$ given by Theorem \ref{thmcube} implies as $t$ goes 0
\begin{multline*}
H(\nu_0|\mu)\leq H(\nu_1|\mu) +\sum_{x\in\X} \sum_{i=1}^n   -D_i( \log f)(x)\, \Pi^i_\rightarrow(x) \,\nu_0(x)\\
-\frac12\int \sum_{i=1}^n h\left(\Pi^i_\rightarrow(x)\right) d\nu_0(x)- \frac12\int \sum_{i=1}^n h_{1}\left(\Pi^i_\leftarrow(y)\right) d\nu_1(y).
\end{multline*}
Choosing then $\nu_1=\mu$ it follows that 
\begin{multline}\label{grrr}
H(\nu_0|\mu)\leq  \sum_{x\in\X} \sum_{i=1}^n   -D_i( \log f)(x)\, \Pi^i_\rightarrow(x) \,\nu_0(x)\\
-\frac12\int \sum_{i=1}^n h\left(\Pi^i_\rightarrow(x)\right) d\nu_0(x)- \frac12\int \sum_{i=1}^n h_{1}\left(\Pi^i_\leftarrow(y)\right) d\mu(y).
\end{multline}

One may check that this inequality is optimal since for the two points space ($n=1$) this is an equality. The proof of this equality is left to the reader. It lies on the fact that since  $\widehat \pi$ is a $W_1$ optimizer, one has  $\pi(x,x)=\min(\nu_0(x),\mu(x))$ for $x=0$ and $x=1$.
From this remark,  starting from the tensorisation form of the one dimensional  convexity property with $C_t=\widetilde T_t$  given by Theorem \ref{tenso} with the  $ \pi^{(n)}$ be the Knothe-Rosenblatt coupling  of $\nu_0\in \Pb(\{0,1\}^n)$ and  $\nu_1=\mu$, one easily check that the same strategy as $t$ goes to 0 implies 
\[H(\nu_0|\mu)\leq  H\big(\nu_0^1\big|\mu_1\big)+ \int \sum_{i=2}^n H\big(\nu_0^i(\cdot|x_1,\ldots x_{i-1}\big|\mu_i\big) \nu_0(x),\]
which is still an equality due to the tensorisation property of entropy.
However, without using the tensorisation argument, we don't know if  \eqref{grrr} is an equality for dimension $n$ bigger than 2.

From  \eqref{grrr} in dimension $n$, using the identity 
\begin{eqnarray}\label{legendre}
\sup_{p\in[0,1)} \left\{  -D p -\frac12 \,h(p)\right\}=\frac12\,h^*(2D_-)= e^{-D_-}+D_--1,
\end{eqnarray}
and since 
\begin{align*}
&\frac12 \, h^*\left(2 [D_i( \log f)(x)]_-\right) f(x)\\
&= \big([D_if(x)]_- +f(\sigma_i(x))\big)[D_i(\log f)(x)]_--[D_if(x)]_-
\\& \leq  [D_i(\log f)(x)]_-[D_if(x)]_-,
\end{align*}
one gets the following modified logarithmic Sobolev inequality,
\begin{align}\label{logsob}
H(f\mu|\mu)&\leq \int \sum_{i=1}^n \frac12 \, h^*\left(2 [D_i( \log f)]_-\right) \nu_0(x) -\frac12 \sum_{i=1}^n \int  h_{1}\left(\Pi^i_\leftarrow(y)\right) d\mu(y)
\nonumber\\
&\leq  \sum_{i=1}^n \int \frac12 \, h^*\left(2 [D_i( \log f)]_-\right) d\nu_0\\
&\leq  \int \sum_{i=1}^n  [D_i(\log f)]_-[D_if]_- d\mu.\nonumber
\end{align}
Since $ h^*\left(2 [D_i( \log f)]_-\right)\leq   [D_i( \log f)]_-^2$, one recovers the reinforced modified  logarithmic Sobolev inequality of Corollary 5.5 in \cite{GRST14}. By means of the Central Limit Theorem, this reinforced modified log-Sobolev inequality actually leads to the usual logarithmic Sobolev inequality of Gross [18] for the standard Gaussian, with the optimal constant (see \cite[Corollary 5.5]{GRST14}).  

A simple way to   improve the modified inequality \eqref{logsob} is to take into account  the extra term involving $h_{1}\left(\Pi^i_\leftarrow(y)\right)$ in \eqref{grrr}. Given $x_i\in\{0,1\}$ and for $j\in[n]\setminus \{i\}$ given $z_j\in \{0,1\}$, let us introduce the notations
\begin{equation*}
\qquad z_{\bar \imath} x_i:= (z_1,\ldots, z_{i-1},x_i,z_{i+1},\ldots z_n)\in \{0,1\}^n, \quad \mbox{
and}\quad 
z_{\bar{\imath}}:= (z_1,\ldots, z_{i-1},z_{i+1},\ldots z_n)\in \{0,1\}^{n-1}.
\end{equation*}
Applying Jensen's inequality, the convexity property of the function $h_1$ provides 
\[\int  h_{1}\left(\Pi^i_\leftarrow(y)\right) d\mu(y)\geq 
\int h_1\left(1-\frac{\sum_{z_{\bar\imath}\in \{0,1\}^{n-1} }  \sum_{w_{\bar \imath}\in \{0,1\}^{n-1} } \widehat \pi(z_{\bar\imath} {y}_i,w_{\bar \imath}y_i)}{\mu_i(y_i)}\right) d\mu_i(y_i),\]
By setting $\mu_{\bar \imath}=\otimes_{j\in[n]\setminus\{i\}} \mu_i$ and since 
\[\sum_{z_{\bar\imath}\in \{0,1\}^{n-1} }\sum_{w_{\bar \imath}\in \{0,1\}^{n-1} } \widehat \pi(z_{\bar\imath}{y}_i,w_{\bar \imath}y_i)\leq \sum_{w\in \{0,1\}^{n} } \widehat \pi(z_{\bar\imath}{y}_i,w)=\sum_{z_{\bar\imath}\in \{0,1\}^{n-1} } f(z_{\bar\imath}{y}_i)   \mu_{\bar \imath}(z_{\bar\imath})\mu_i({y}_i),\]
and 
\[\sum_{z_{\bar\imath}\in \{0,1\}^{n-1} }\sum_{w_{\bar \imath}\in \{0,1\}^{n-1} } \widehat \pi(z_{\bar\imath}{y}_i,w_{\bar \imath}y_i)\leq \sum_{z\in \{0,1\}^{n} }\sum_{w_{\bar \imath}\in \{0,1\}^{n-1} } \widehat \pi(z,w_{\bar \imath}y_i)=\mu_i({y}_i),\]
 it follows that 
\begin{equation*}
 \int  h_{1}\left(\Pi^i_\leftarrow(y)\right) d\mu(y)
\geq
\int h_1\left(1-\min\Big\{1, {\sum_{z_{\bar\imath}\in \{0,1\}^{n-1} } f(z_{\bar\imath}{y}_i)   \mu_{\bar \imath}(z_{\bar\imath})}\Big\}\right) \,d\mu_i(y_i).
\end{equation*}
For any fixed $y_i\in\{0,1\}$, one has
\begin{equation*}1-\min\Big\{1, {\sum_{z_{\bar\imath}\in \{0,1\}^{n-1} } f(z_{\bar\imath}{y}_i) \mu_{\bar \imath}(z_{\bar\imath})}\Big\}=\left[\int f(z) -f(z_{\bar\imath}{y}_i) \,d \mu(z)\right]_+=\left[\int D_i f(y) \,d \mu_{\bar \imath}(y_{\bar\imath})\right]_+\mu_i(\overline y_i).
\end{equation*}
As a consequence \eqref{logsob} provides the following  new modified logarithmic  Sobolev inequality on the discrete hypercube,  
\begin{equation*}
H(f\mu|\mu)\leq  \sum_{i=1}^n \int \frac12 \, h^*\left(2 [D_i( \log f)]_-\right) d\nu_0 - \sum_{i=1}^n \frac12 \int h_1\left(\left[\int D_i f(y) \,d \mu_{\bar \imath}(y_{\bar\imath})\right]_+\mu_i(\overline y_i)\right) \,d\mu_i(y_i).
\end{equation*}
 
 As we will show in a forthcoming paper, the last strategy also simply provides modified logarithmic  Sobolev inequalities for probability measures $\mu_V=e^{-V}\mu$ on $\{0,1\}^n$  with interaction potentials  $V:\{0,1\}^n\to \R$, that can not be easily derived from tensorization property arguments.  \end{enumerate}
\end{comments}

 \subsection{The circle $\Z/N\Z$ endowed with a uniform measure}
Let $N\in \N$ and $\X$ be the space $\Z/N\Z$, endowed with the uniform probability   measure $\mu$, $\mu(x)=1/N$. The measure $\mu$ is reversible with respect to the generator $L$ given by , 
\[L(z,z+1)=L(z,z-1)=1, \qquad L(z,z)=-2,\]
for any $z\in \Z/N\Z$.
One always have $d(x,y)\leq \lfloor N/2\rfloor=n$ where $ \lfloor \cdot \rfloor$ denotes the floor function.  

If $N$ is odd then for any $x,y\in \Z/N\Z$, $L^{d(x,y)}(x,y)=1$ and therefore the Schr\"odinger bridge at zero temperature $(\widehat Q_t)_{t\in[0,1]}$ joining two probability measures $\nu_0$ and $\nu_1$ on $\Z/N\Z$  is given by \eqref{limitbridge}, with according to \eqref{defnut0} 
\[{Q_t}\!^{ x,y}(z)=
\1_{z\in[x,y]}\,\B_t^{d(x,y)}\big(d(x,z)\big).\]

If $N$ is even  then for any $x,y\in \Z/N\Z$ such that $d(x,y)<N/2$, $L^{d(x,y)}(x,y)=1$  and $L^{d(x,x+n)}(x,x+n)=2$. The Schr\"odinger bridge at zero temperature $(\widehat Q_t)_{t\in[0,1]}$   is given by \eqref{limitbridge}, with according to \eqref{defnut0} :  if $d(x,y)< N/2$ then 
\[{Q_t}\!^{ x,y}(z)=
\1_{z\in[x,y]}\,\B_t^{d(x,y)}\big(d(x,z)\big),\]
and if $d(x,y)= N/2$ ($y=x+n$),  for any  $z\in \Z/N\Z\setminus\{x,x+n\}$,
\[{Q_t}\!^{x,x+n}(z)=
\frac12\1_{z\in[x,x+n]}\,\B_t^{d(x,x+n)}\big(d(x,z)\big),\]
and ${Q_t}\!^{x,x+n}(x)=\,(1-t)^{d(x,x+n)}$, ${Q_t}\!^{x,x+n}(x+n)=\,t^{d(x,x+n)}$.
\begin{theorem}\label{thmcircle} On the space $(\Z/N\Z,\mu,d,L)$, the relative entropy $H(\cdot|\mu)$ satisfies the 0-displacement convexity \eqref{deplacebis}.
\end{theorem}
Therefore the space $(\Z/N\Z,d,\mu,L)$ has non-negative $W_1$ or $T_2 $ entropic curvature. 

\subsection{The Bernoulli-Laplace model} Let $\X=\X_\kappa$ denotes the slice of the discrete hypercube $\{0,1\}^n$ of order $k\in[n-1]$, endowed with the uniform probability   measure $\mu$, namely
\[ \X_\kappa:=\left\{x=(x_1,\ldots,x_n)\in \{0,1\}\,\big| \,x_1+\ldots +x_n=\kappa\right\}.\]
For $z\in \X_\kappa$, let $J_0(z):=\{i\in [n]\,|\, z_i=0\}$ and $J_1(z):=\{i\in [n]\,|\, z_i=1\}$.
For  any $i\in J_0(z)$ and $j\in J_1(z)$,  one denotes  $\sigma_{ij}(z)$  the neighbour of $z$ in $\X_\kappa$  defined by
\[\left(\sigma_{ij}(z)\right)_i=1,\quad\left(\sigma_{ij}(z)\right)_j=0,\]
and for any $\ell\in[n]\setminus\{i,j\}$, $\left(\sigma_{ij}(z)\right)_\ell=z_\ell$.
The set of edges of the graph is 
\[E:=\Big\{(z,\sigma_{ij}(z))\,\big|\, z\in \X_\kappa, \{i,j\}\subset[n], z_i=0, z_j=1\Big\},\]
and the  
graph distance  is given by
\[d(x,y):=\frac12\sum_{i=1}^n \1_{x_i\neq y_i},\qquad x,y\in \X_\kappa.\] 

The measure $\mu$ is reversible with respect to the generator  $L$ given by  
$L(z,\sigma_{ij}(z)):=1$   for any $i,j$  such that $z_i=0$  and  $z_j=1$, 
and $L(z,z):=-\kappa(n-\kappa)$.

Since $L^{d(x,y)}(x,y)= (d(x,y)!)^2$, the Schr\"odinger bridge at zero temperature $(\widehat Q_t)_{t\in[0,1]}$ is given by  \eqref{limitbridge}, with according to \eqref{defnut0},
\begin{eqnarray}\label{nutsectioncube}
  {Q_t}\!^{ x,y}(z)=  \1_{[x,y]}(z) \,\binom{d(x,y)}{d(x,z)} ^{-1}\;t^{d(x,z)} (1-t)^{d(z,y)},\qquad z\in\X_\kappa.
 \end{eqnarray}
\begin{theorem}\label{thmslicecube}   
On the space $(\X_\kappa,\mu,d,L)$, the relative entropy $H(\cdot|\mu)$ satisfies the $C$-displacement convexity property \eqref{deplacebis}, with $C=(C_t)_{t\in (0,1)}$ defined by: for any $\nu_0,\nu_1\in{\mathcal P}(\X_\kappa)$ with associated  limit  Schr\"odinger problem optimizer $\widehat Q\in \Pc(\Omega)$,
\begin{equation*} 
C_t(\widehat \pi):=\max \Big[ \frac4{\min[\kappa,n-\kappa]} W_1^2(\nu_0,\nu_1),\,\frac4{\min[\kappa,n-\kappa]}  T_{c_2}(\widehat \pi) ,  \widetilde T_t(\widehat \pi)\Big],
\end{equation*}
where $ \widehat\pi=\widehat Q_{0,1}$,
the cost function $c_2$ of $T_{c_2}$  is the same as in Theorem \ref{thmcube}, and  
the cost $\widetilde T_t$ is defined by 
\begin{multline*}
\widetilde T_t(\widehat \pi):=\int \max\Big[\sum_{i\in J_0(x)}  \,h_t\Big(  \Pi^i_\rightarrow(x)\Big),\sum_{j\in J_1(x)} \,h_t\Big(\Pi^j_\rightarrow(x)\Big) \Big]\,d\nu_0(x)\\
+\int  \max\Big[\sum_{i\in J_0(y)} h_{1-t}\Big(\Pi^i_\leftarrow(y) \Big), \sum_{j\in J_1(y)}   h_{1-t}\Big(\Pi^j_\leftarrow(y) \Big)\Big] \,d\nu_1(y),
\end{multline*}
with the same 
 definitions for  the   functions $h_t$, $t\in(0,1)$ and the quantities $\Pi^i_\rightarrow(x)$ and $\Pi^i_\leftarrow(y)$ as  in Theorem \ref{thmcube}. 
 \end{theorem}

\begin{comments}
\begin{enumerate}[label*=(\alph*)]
\item Since   
$C_t(\widehat \pi)\geq \frac{4}{\min(\kappa,n-\kappa)} \,W_1^2(\nu_0,\nu_1)$,  the $W_1$-entropic curvature of the space $(\X_\kappa,d,L)$ is bounded from below by $\frac4{\min(\kappa,n-\kappa)}$. Observe that this constant is optimal  for $\kappa=1$ or $\kappa=n-1$, since $\X_\kappa$ is the complete graph and one recovers its optimal lower curvature bound $4$ (see Comment \ref{commentcomplet} of Theorem \ref{thmcomplete}). 

In the paper \cite[Theorem 1.1]{EMT15} the Erbar-Maas entropic curvature of the Bernoulli Laplace model along $\mathcal{ W}_2$-geodesics is bounded from below  by $\frac{n+2}{2\kappa(n-\kappa)}$, therefore their curvature term is of order 
\[\frac{n+2}{2\kappa(n-\kappa)} \mathcal{ W}_2^2(\nu_0,\nu_1)\geq \frac{n+2}{\kappa(n-\kappa)} { W}_1^2(\nu_0,\nu_1).\] 
Theorem \ref{thmslicecube} a slightly better constant as regards to the $W_1$-curvature term since $ \frac{n+2}{\kappa(n-\kappa)}\leq \frac4{\min[\kappa,n-\kappa]}$ with equality for $(\kappa,n)=(1,2)$.

\item Since  $C_t(\widehat \pi)\geq \frac{4}{\min(\kappa,n-\kappa)} \,T_{c_2}\big(\widehat \pi\big)$,  the $T_2 $-entropic curvature of the space $(\X_\kappa,d,L)$ is bounded from below by $\frac4{\min(\kappa,n-\kappa)}$.
Moreover, applying  
Theorem \ref{prek}, this lower bound  provides a new type of curved Pr\'ekopa-Leindler inequality on the slices of the discrete hypercube.

\item According to  the definition \eqref{t2tilde} of $\widetilde T_2(\nu_0,\nu_1)$, as in the case of  the hypercube, one has 
$C_t(\widehat \pi)\geq \widetilde T_t(\widehat \pi)\geq \frac 12   \widetilde T_2(\nu_0,\nu_1)$.
As a consequence, since $H(\widehat Q_t|\mu)\geq 0$, optimizing over all $t\in(0,1)$, Theorem \ref{thmslicecube} implies the following weak transport-entropy inequality, for any  $\nu_0,\nu_1\in \Pc(\X_\kappa)$,
\[\frac14 \,\widetilde T_2(\nu_0,\nu_1)\leq \left(\sqrt{ H(\nu_0|\mu)} +\sqrt{ H(\nu_1|\mu)}\right)^2.\]
This inequality is a reinforced symmetric version of a transport entropy inequality  given  in \cite[Theorem 1.8 (b)]{Sam17} with the worse constant $1/8$ instead of $1/4$. It was surprisingly obtained by projection of a  transport-entropy inequality for the uniform measure on the symmetric group. The   approach of the present paper  is much more natural to reach such a result. 

\item From the lower bound $C_t(\widehat \pi)\geq \widetilde T_t(\widehat \pi)$, Theorem \ref{thmslicecube} also yields a  modified logarithmic Sobolev.  
For any real function $g$ on $\X_\kappa$, let us note \[ D_{ij}g(x):=g(\sigma_{ij}(x))-g(x), \qquad x\in \X_\kappa, \quad (i,j)\in J_0(x)\times J_1(x) .\]
Assume $\nu_0$ has positive density $f$ with respect to $\mu$ and let us choose $\nu_1=\mu$. According to Lemma \ref{logsobcube}, setting $\Pi^{ij}_\rightarrow(x)=\Pi^{\sigma_{ij}(x)}_\rightarrow(x)$,  the convexity property  \eqref{deplacebis}  with $C_t=\widetilde T_t$ given by Theorem \ref{thmslicecube} implies as $t$ goes to 0
\begin{multline*}
H(\nu_0|\mu)\leq \sum_{x\in\X} \sum_{(i,j)\in J_0(x)\times J_1(x)}   -D_{ij}( \log f)(x)\, \Pi^{ij}_\rightarrow(x) \,\nu_0(x)\\
-\frac12\int \max\Big[\sum_{i\in J_0(x)}  \,h\Big(  \Pi^i_\rightarrow(x)\Big),\sum_{j\in J_1(x)} \,h\Big(\Pi^j_\rightarrow(x)\Big) \Big]\,d\nu_0(x).
\end{multline*}
Now, let us observe that for $i\in J_0(x)$, one has
\begin{align*}
\sum_{j\in J_1(x)} \Pi^{ij}_\rightarrow(x) &=\int \sum_{j\in J_1(x)}\1_{\sigma_{ij}(x) \in [x,y]} d(x,y) \,r(x,\sigma_{ij}(x) ,\sigma_{ij}(x) ,y) \, d\widehat \pi_\rightarrow(y|x)\\
&= \int \sum_{j\in J_1(x)\cap J_1(y)} \1_{x_i\neq y_i} d(x,y) \,r(x,\sigma_{ij}(x) ,\sigma_{ij}(x) ,y) \, d\widehat \pi_\rightarrow(y|x)\\
&= \int \1_{x_i\neq y_i} \frac{d^2(x,y) \big((d(x,y)-1)!\big)^2}{ \big(d(x,y)!\big)^2 }  \, d\widehat \pi_\rightarrow(y|x)= \Pi^i_\rightarrow(x), 
\end{align*}
and similarly for $j\in J_1(x)$, one has $\sum_{i\in J_0(x)} \Pi^{ij}_\rightarrow(x) 
= \Pi^j_\rightarrow(x)$. 
It follows that 
\begin{align*}
H(f\mu|\mu)&\leq \int \min\Big[\sum_{i\in J_0(x)}   \left( \max_{j\in J_1(x)} [D_{ij}( \log f)(x)]_-\, \Pi^{i}_\rightarrow(x)-\frac12h\left(\Pi^j_\rightarrow(x) \right)\right),\\
 &\qquad\qquad\sum_{j\in J_1(x)}   \left( \max_{i\in J_0(x)} [D_{ij}( \log f)(x)]_-\, \Pi^{j}_\rightarrow(x)
-\frac12h\left(\Pi^j_\rightarrow(x) \right)\right)\Big] \,d\nu_0(x)
\end{align*}
Finally the identity \eqref{legendre} gives the following modified logarithmic inequality
\begin{align}\label{logsobslice}
H(f\mu|\mu)&\leq \int \min\Big[\sum_{i\in J_0}  \frac12 h^*\left(2\max_{j\in J_1} [D_{ij}( \log f)]_-\right),
\sum_{j\in J_1}  \frac12h^*\left(2\max_{i\in J_0} [D_{ij}( \log f)]_-\right)\Big] \,d\nu_0\nonumber\\
&\leq \int \min\Big[\sum_{i\in J_0} \max_{j\in J_1} \big([D_{ij}(\log f)]_-[D_{ij}f]_-\big), \sum_{j\in J_1} \max_{i\in J_0} \big([D_{ij}(\log f)]_-[D_{ij}f]_-\big)\Big] \,d\mu
\end{align}
From the lower bound $\frac{n+2}{2\kappa(n-\kappa)}$ of Erbar entropic curvature given in \cite[Theorem 1.1]{EMT15}, we know from \cite[Theorem 7.4]{EM12} that the following modified logarithmic Sobolev inequality holds 
\begin{equation}\label{EMT}
H(f\mu|\mu)\leq c_n\int \sum_{(i,j)\in J_0\times J_1} D_{ij}(\log f) D_{ij}f \,d\mu=2c_n \int \sum_{(i,j)\in J_0\times J_1} [D_{ij}(\log f)]_- [D_{ij}f]_-\,d\mu, 
\end{equation}
with $c_n={1}/{2(n+2)}$, and the best constant $c_n$ in this inequality  is known to be greater than $1/4n$ (see comments after \cite[Theorem 1.1]{EMT15}).
This inequality is stronger than \eqref{logsobslice}. Indeed, one has 
\begin{align*}
 &\frac{1}{(n+2)} \int \sum_{(i,j)\in J_0\times J_1} [D_{ij}(\log f)]_- [D_{ij}f]_-\,d\mu\\&\leq \frac{1}{(n+2)}\int \min\Big[\kappa \sum_{i\in J_0} \max_{j\in J_1} \big([D_{ij}(\log f)]_-[D_{ij}f]_-\big), (n-\kappa)\sum_{j\in J_1} \max_{i\in J_0} \big([D_{ij}(\log f)]_-[D_{ij}f]_-\big)\Big]\,d\mu\\
 &\leq \int \min\Big[ \sum_{i\in J_0} \max_{j\in J_1} \big([D_{ij}(\log f)]_-[D_{ij}f]_-\big), \sum_{j\in J_1} \max_{i\in J_0} \big([D_{ij}(\log f)]_-[D_{ij}f]_-\big)\Big]\,d\mu.
 \end{align*}
 Choosing  the function $f$ defined by $f(x):=\alpha(x_1+\beta)$, $x\in \X_\kappa$, where $\beta>0$ and $\alpha$ is a renormalisation constant, one may  check that the right-hand side and the left-hand side of this inequality are asymptotically equivalent as $n$ goes to infinity. However it remains a challenge to improve our strategy in order to recover \eqref{EMT}.   
\end{enumerate}
\end{comments}


\section{Proof of the main results}\label{proof}
This section is divided into two parts. We first present general statements  to prove displacement convexity property \eqref{deplacebis} along Schr\"odinger bridges at zero temperature. 
Then we show how it applies    for each involved discrete space of the last part. 

\subsection{Strategy  of  proof,  general statements to get entropic curvature results}
In order to prove property \eqref{deplacebis}, we fix two probability measures $\nu_0$ and $\nu_1$ in $\Pc_b(\X)$ in this part. 
As in the paper  by G. Conforti \cite{Con18} in continuous setting,
the first step is to decompose the relative-entropy using the product structure given by \eqref{Qtstructure}: for any  $t\in [0,1]$, 
\[H(\widehat Q^\gamma_t|m)=\varphi_\gamma(t)+\psi_\gamma(t),\]
where
\[\varphi_\gamma(t):=\int \log (P^\gamma_t f^\gamma) P^\gamma_t f^\gamma \, P^\gamma_{1-t}g^\gamma dm \quad\mbox{and}\quad \psi_\gamma(t):=\int \log (P^\gamma_{1-t}g^\gamma) P^\gamma_{1-t}g^\gamma  \,P^\gamma_t f^\gamma dm.\]
As recalled below, it is known that the function $\varphi_\gamma$  is non-increasing  and the function  $\psi_\gamma$ is non-decreasing (see \cite[Theorem 6.4.2]{Leo17}). 

Then, the strategy is to analyse  the behaviour of the second order derivative $\varphi_\gamma''$ and $\psi_\gamma''$ as $\gamma$ goes to 0, in order to apply the next Lemma. For any $t\in (0,1)$ let $K_t:[0,1]\to \R_+$, be defined by 
\begin{equation}\label{noyau}
 K_t(u)= \frac{2u}t\1_{u\leq t}+ \frac{2(1-u)}{1-t}\1_{u\geq t}, \qquad u\in[0,1].
 \end{equation}
$K_t$ is a kernel function since $\int_0^1K_t(u)\, du=1$. 
\begin{lemma}\label{lemconvex} 
Assume that   hypothesis  \eqref{unifbounded0}, \eqref{unifbounded1}, \eqref{unifbounded2} and \eqref{convhyp} hold. 
Let $(\gamma_\ell )_{\ell\in \N} $ be a sequence of positive numbers that converges to 0. If for any $t\in (0,1)$ 
\begin{eqnarray}\label{condlemconvex}
\liminf_{\gamma_\ell \to 0} \varphi_{\gamma_\ell }''(t)+\liminf_{\gamma_\ell \to 0} \psi_{\gamma_\ell }''(t)\geq \xi''(t),
\end{eqnarray}
where $\xi$ is a  continuous functions on [0,1], twice differentiable on $(0,1)$, depending on the coupling $\widehat \pi$,
then the displacement convexity property   \eqref{deplacebis} holds with 
\begin{equation*}
C_t(\widehat \pi):=\int_0^1\xi''(u) K_t(u)\, du
=\frac2{t(1-t)}\Big[
(1-t)\xi(0)+t\xi(1)-\xi(t)\Big].
\end{equation*}
\end{lemma}
Observe that if $\xi''=K$ is a constant function, then 
$C_t(\widehat \pi)=K.$
The proof of this lemma is postponed in Appendix B. 

In order to apply  Lemma \ref{lemconvex}, we need first to compute  $\varphi'_\gamma,\psi'_\gamma$ and $\varphi''_\gamma,\psi''_\gamma$ in a suitable form so as to get 
\eqref{condlemconvex}. 
For any real function $u$ on $\X$, we note
\[\nabla u (z,w)=u(w)-u(z), \quad z,w\in \X,\] 
and 
\[Lu(z):= \sum_{w\in\X}u(w)\,L(z,w)=\sum_{w, w\sim z}\nabla u (z,w)\, L(z,w).\]

The expressions of $\varphi_\gamma',\psi_\gamma'$ and $\varphi_\gamma'',\psi_\gamma''$ are given by the next lemmas. These expressions   can be found in L\'eonard's paper \cite[section 6.4]{Leo17} in a more general framework (for stationary non-reversible Markov processes). For 
completeness,  the proof of the next result is recalled in Appendix B.
\begin{lemma}\label{deriv1phi} For any  $t\in(0,1)$, one has
\[\varphi_\gamma'(t)=- \int \sum_{z',z'\sim z}\zeta(e^{\nabla F_t^\gamma(z,z')}) \,L^\gamma(z,z') \,d\widehat Q^\gamma_t(z),\]
and 
\[\psi_\gamma'(t)=\int \sum_{z',z'\sim z}\zeta(e^{\nabla G_t^\gamma(z,z')})\, L^\gamma(z,z') \,d\widehat Q^\gamma_t(z),\]
where 
$\zeta(s):=s\log s -s+1,  s> 0,$
and $G^\gamma_t$ and $F^\gamma_t$ are the so-called {\it Schr\"odinger potentials} according to 
L\'eonard's paper terminology \cite{Leo17},
\[ G_t^\gamma:=\log P_{1-t}^\gamma g^\gamma, \qquad \mbox{and} \qquad F_t^\gamma:=\log P_t^\gamma f^\gamma.\]
\end{lemma}

Since $\zeta\geq 0$, the function $\varphi_\gamma$ is non-increasing  and the function $\psi_\gamma$ is non-decreasing.

\begin{lemma}\label{deriv2phibis} For any $a>0,b>0$, let 
\begin{equation*}
\rho(a,b):=\left(\log b-2\log a-1\right) b,
\end{equation*}
and let $\rho(a,b)=0$ if either $a=0$ or $b=0$.
 For any $t\in(0,1)$, one has
\begin{align*}
\varphi''_\gamma(t)&=  \int  \Big[\Big(  \sum_{\tz, \tz\sim z} e^{\nabla F_t^\gamma(z, \tz)} L^\gamma(z, \tz)\Big)^2
+\sum_{\tz, \tz\sim z} \left(1+\nabla F_t^\gamma(z, \tz)\right) \,e^{\nabla F_t^\gamma(z, \tz)}\Big(L^\gamma(z,z)-L^\gamma(\tz,\tz) \Big) \,L^\gamma(z, \tz) \\
&\quad \qquad + \sum_{ \tz, \ttz, z\sim \tz \sim \ttz } \rho\left(e^{\nabla F_t^\gamma(z, \tz)},e^{\nabla F_t^\gamma(z, \ttz)}\right) 
L^\gamma(z, \tz )L^\gamma(\tz,\ttz )\Big] \,d\widehat Q_t^\gamma(z), \\
\psi''_\gamma(t)&=  \int  \Big[\Big(  \sum_{\tz, \tz\sim z} e^{\nabla G_t^\gamma(z, \tz)} L^\gamma(z, \tz)\Big)^2
+ \sum_{\tz, \tz\sim z} \left(1+\nabla G_t^\gamma(z, \tz)\right) \,e^{\nabla G_t^\gamma(z, \tz)}\Big(L^\gamma(z,z)-L^\gamma(\tz,\tz) \Big) \,L^\gamma(z, \tz) \\
&\quad \qquad+ \sum_{ \tz, \ttz, z\sim \tz \sim \ttz } \rho\left(e^{\nabla G_t^\gamma(z, \tz)},e^{\nabla G_t^\gamma(z, \ttz)}\right) 
L^\gamma(z, \tz )L^\gamma(\tz,\ttz )\Big] \,d\widehat Q_t^\gamma(z).
\end{align*}
\end{lemma}

Let us now analyse  the behavior of $\varphi_\gamma''(t)$, $\psi_\gamma''(t)$ as  temperature $\gamma$ goes to zero.  
Recall first that for $t\in(0,1)$, the support of the Schr\"odinger bridge at zero temperature $\widehat  Q_t$ given by \eqref{suppQt} is independent of $t$.
For sake of simplicity, one denotes  \[\widehat Z:= \supp(\widehat  Q_t),\qquad t\in(0,1).\] 
  As a consequence, one  expects that 
the limit behavior of $\varphi_\gamma''(t)$, $\psi_\gamma''(t)$ is expressed in term of sums restricted to points of  $\widehat Z$. Let us define, for any $z\in \widehat Z$, \[V_{_\rightarrow}(z):=\Big\{\tz\in V(z)\,\Big|\, (z,\tz)\in C_\rightarrow\Big\} \quad\mbox{ and }\quad V_{_\leftarrow}(z):=\Big\{\tz\in V(z)\,\Big|\, (z,\tz)\in C_\leftarrow\Big\},\]
where 
\[C_{_\rightarrow}:=\Big\{(z,w)\in\X\times \X\,\Big|\,z\neq w, \exists (x,y)\in \supp(\widehat{\pi}), (z,w)\in [x,y]\Big\},\]
and 
\[C_{_\leftarrow}:=\Big\{(z,w)\in\X\times \X\,\Big|\, (w,z) \in C_{_\rightarrow} \Big\}.\] 
Similarly, one also defines 
\[\V_{_\rightarrow}(z):=\Big\{\ttz,\in \V(z)\,\Big|\, (z,\ttz)\in C_{_\rightarrow}\Big\} \quad\mbox{ and }\quad\V_{_\leftarrow}(z):=\Big\{\ttz,\in \V(z)\,\Big|\, (z,\ttz)\in C_{_\leftarrow}\Big\},\] where for any $z\in \X$
 \[\V(z):=\Big\{\ttz\in \X\,\Big|\, d(z,\ttz)=2\Big\}.\]
 As a remarkable fact, according to Lemma \ref{opt} postponed in Appendix A, from the $d$-cyclically monotone property of the $W_1$-optimal coupling $\widehat{\pi}$, $C_{_\rightarrow}$ and $C_{_\leftarrow}$ are disjoint sets. This implies that $V_{_\rightarrow}(z)$
 and $V_{_\leftarrow}(z)$ are disjoint, and also  $\V_{_\rightarrow}(z)$ and $\V_{_\leftarrow}(z)$, for any $z\in \widehat Z$.

According to the expression of $\varphi_\gamma''(t)$, $\psi_\gamma''(t)$ given in Lemma \ref{deriv2phibis}, a first step is to give the behavior as $\gamma$ goes to zero of the quantities 
\[A^\gamma_t(z,u):=e^{\nabla F_t^\gamma (z, u)}=\frac{P_t^\gamma f^\gamma(u)}{P_t^\gamma f^\gamma(z)}\quad\mbox{ and }\quad
B^\gamma_t(z,u):= e^{\nabla G_t^\gamma (z, u)}=\frac{P_{1-t}^\gamma g^\gamma(u)}{P_{1-t}^\gamma g^\gamma(z)},\]
for $u=\tz$ or $u=\ttz$ with $z\sim\tz\sim\ttz$. This is a key result of this paper. Let us briefly give the intuition behind it. From the Markov property, the quantity  $A^\gamma_t(z,u)$ can be interpreted as the mean ratio  of transition probabilities under conditional law of the Schr\"odinger bridge, namely 
\begin{equation}\label{encore}
A^\gamma_t(z,u)=\sum_{w\in\X} \frac{P_t^\gamma(u,w)}{ P_t^\gamma(z,w)} \, \, \frac{f^\gamma(w)P_t^\gamma(z,w)}{P_tf^\gamma(z)}=\sum_{w\in\X} \frac{P_t^\gamma(u,w)}{ P_t^\gamma(z,w)} \, \widehat Q^\gamma(X_0=w|X_t=z),
\end{equation}
where $\widehat Q^\gamma(X_0=w|X_t=z)$ is the law of $X_0$ given $X_t=z$ under the law $\widehat Q^\gamma$. As $\gamma$ goes to 0, the law $\widehat Q^\gamma$ tends to $\widehat Q$, and the behavior of the ratio is given by the Taylor expansion of $P_t^\gamma$ as $\gamma$ goes to 0, namely according to Lemma \ref{lemmetech} $(iii)$,
  \[\frac{P_t^\gamma(u,w)}{ P_t^\gamma(z,w)}=(t\gamma)^{d(u,w)-d(z,w)} \left(\frac{L^{d(u,w)}(u,w)d(z,w)!}{L^{d(z,w)}(z,w)d(u,w)!}\,+o(1)\right). \]
  Therefore, if $\gamma$ goes to 0 then the main contribution in the sum given by \eqref{encore} is for points $w\in\X$ such that $d(u,w)-d(z,w)$ has minimum value. This means that $u\in[z,w]$, so that $d(u,w)-d(z,w)=-d(u,z)$. It follows that for $u=\tz$ with  $\tz\sim z$,
  \[A^\gamma_t(z,\tz)\underset{\gamma\to 0}{\sim} \frac1{\gamma t}\sum_{w\in\X, \tz\in[z,w]}  \frac{L^{d(\tz,w)}(\tz,w)d(z,w)!}{L^{d(z,w)}(z,w)d(\tz,w)!} \, \widehat Q(X_0=w|X_t=z),\]
  and for $u=\ttz$ with $d(z,\ttz)=2$,
  \[A^\gamma_t(z,\ttz)\underset{\gamma\to 0}{\sim}   \frac1{\gamma^2t^2}  \sum_{w\in\X, \ttz\in[z,w]} \frac{L^{d(\ttz,w)}(\ttz,w)d(z,w)!}{L^{d(z,w)}(z,w)d(\ttz,w)!} \, \widehat Q(X_0=w|X_t=z).\]
 The quantity  $ B^\gamma_t(z,u)$ can be similarly analysed as $\gamma$ goes to 0.
 
Let us now formulate precise statements. One needs to define several quantities. For any $z\in \X$, $x \in \supp(\nu_0)$, $y\in \supp(\nu_1)$ and any $t\in(0,1)$, let   
 \begin{equation}\label{a_t}
 a_t(z,y) :=\widehat Q(X_t=z|X_1=y)=\int {Q_t}\!^{ w,y}(z) \,d\widehat{\pi}_{_\leftarrow}(w|y),
 \end{equation}
 and 
 \[b_t(z,x):=\widehat Q(X_t=z|X_0=x)=\int {Q_t}\!^{ x,w}(z)\,d\widehat{\pi}_{_\rightarrow}(w|x).\]
 Observe that for $t\in(0,1)$, $a_t(z,y)>0$ if and only if $z\in \widehat Z$ and 
 $y\in \widehat Y_z$ with
 \[ \widehat Y_z:=\Big\{y\in\supp(\nu_1)\,\Big|\, \exists x\in \X,  (x,y)\in \widehat{\pi}, z\in[x,y]\Big\}.\]
 Identically $b_t(z,x)>0$ if and only if $z\in \widehat Z$ and 
 $x\in \widehat X_z$ with
 \[ \widehat X_z:=\Big\{x\in\supp(\nu_0)\,\Big|\, \exists x\in \X,  (x,y)\in \widehat{\pi}, z\in[x,y]\Big\}.\]
 For further use, for any $y\in \supp(\nu_1)$ and $x\in \supp(\nu_0)$, we also introduce the sets 
 \[\widehat Z^y:=\Big\{z\in \widehat Z\,\Big|\, y\in \widehat Y_z\Big\}\quad \mbox{ and } \quad\widehat Z_x:=\Big\{z\in \widehat Z\,\Big|\, x\in \widehat X_z\Big\},\]
 so that 
 \[\left(z\in \widehat Z, y\in  \widehat Y_z\right)\Leftrightarrow \left(y\in  \supp(\nu_1), z\in \widehat Z^y\right),\]
 and 
  \[\left(z\in \widehat Z, x\in  \widehat X_z\right)\Leftrightarrow \left(x\in  \supp(\nu_0), z\in \widehat Z_x\right).\]

 For any $z\in \widehat Z$,   $\tz\in V(z)$,  define 
\begin{equation}\label{a_t'}
{\mathrm a}_t(z,\tz,y):= \sum_{w\in \X, (z,\tz)\in[y,w]}  
r(y,z,\tz,w) \,d(y,w)\, \B_t^{d(y,w)-1}(d(z,w)-1)\,\widehat{\pi}_{_\leftarrow}(w|y),
\end{equation}
and 
\begin{equation*}
{\mathrm b}_t(z,\tz,x):=\sum_{w\in \X, (z,\tz)\in[x,w]}  
\,r(x,z,\tz,w) \,d(x,w)\, \B_t^{d(x,w)-1}(d(x,z))\,\widehat{\pi}_{_\rightarrow}(w|x),
\end{equation*}
where the function $r$ is given by \eqref{defr}.
One easily check that ${\mathrm a}_t(z,\tz,y)>0$ if and only if $\tz\in V_{_\leftarrow}(z)$ and 
$y\in \widehat Y_{(z,\tz)}$ with
 \[ \widehat Y_{(z,\tz)}=\Big\{y\in\supp(\nu_1)\,\Big|\, \exists x\in \X,  (x,y)\in \widehat{\pi}, (z,\tz)\in[y,x]\Big\}\subset \widehat Y_z\cap  \widehat Y_{\tz} ,\]
 and identically ${\mathrm b}_t(z,\tz,x)>0$ if and only if $\tz\in V_{_\rightarrow}(z)$ and 
$x\in \widehat X_{(z,\tz)}$ with
 \[ \widehat X_{(z,\tz)}=\Big\{x\in\supp(\nu_0)\,\Big|\, \exists y\in \X,  (x,y)\in \widehat{\pi}, (z,\tz)\in[x,y]\Big\}\subset \widehat X_z\cap  \widehat X_{\tz}.\]

For any $z\in \widehat Z$ and $\ttz\in \V(z)$, define also 
 \begin{equation}\label{a_t''}
{\mathbbm a}_t(z,\ttz,y):= \!\!\!\!\!\sum_{w\in \X, (z,\ttz)\in[y,w]}  
\!\!\!\!\!r(y,z,\ttz,w) \,d(y,w)(d(y,w)-1)\, \B_t^{d(y,w)-2}(d(z,w)-2)\,\widehat{\pi}_{_\leftarrow}(w|y),
\end{equation}
and 
\begin{eqnarray*}
 {\mathbbm b}_t(z,\ttz,x):=\!\!\!\!\!\sum_{w\in \X, (z,\ttz)\in[x,w]} \!\!\!\!\! 
r(x,z,\ttz,w) \,d(x,w)(d(x,w)-1)\, \B_t^{d(x,w)-2}(d(x,z))
\,\widehat{\pi}_{_\rightarrow}(w|x).
\end{eqnarray*}
We also have ${\mathbbm a}_t(z,\ttz,y)>0$ if and only if $\ttz\in \V_{_\leftarrow}(z)$ and 
$y\in \widehat Y_{(z,\ttz)}$, 
and ${\mathbbm b}_t(z,\ttz,x)>0$ if and only if $\ttz\in \V_{_\rightarrow}(z)$ and 
$x\in  \widehat X_{(z,\ttz)}$.
  
\begin{lemma}\label{lemintaibis} Assume that conditions \eqref{unifbounded1} and  \eqref{unifbounded2} 
are fulfilled. Let $(\gamma_\ell )_{\ell\in \N}$ be a sequence of positive numbers converging to 0, and let $\widehat Q_t$ denote the weak limit of the sequence of probability measures  $(\widehat Q_t^{\gamma_\ell })_{\ell\in \N}$.  Let $z\in \widehat Z$. 
\begin{itemize}
\item For any $\tz \in V(z)$, it holds
 \begin{eqnarray}\label{conv1}
 \lim_{\gamma_\ell\to 0}  \left(\gamma_\ell  A^{\gamma_\ell }_t(z,\tz)\right)={A_t(z,\tz)\geq 0}   \quad \mbox{and} \quad \lim_{\gamma_\ell\to 0}  \left(\gamma_\ell  B^{\gamma_\ell }_t(z,\tz)\right)={B_t(z,\tz)}\geq 0,
 \end{eqnarray}
 with $A_t(z,\tz)>0$ if and only if  $\tz\in V_{_\leftarrow}(z)$, and $B_t(z,\tz)>0$ if and only if  $\tz\in V_{_\rightarrow}(z)$. 
 Moreover, given $\tz\in V_{_\leftarrow}(z)$, for any $y\in  \widehat Y_z$  
 \[ \displaystyle A_t(z,\tz):=\frac{{\mathrm a}_t(z,\tz,y)}{a_t(z,y)},\]
 and given $\tz\in V_{_\rightarrow}(z)$, for any $x\in  \widehat X_z$  
\[B_t(z,\tz):=\frac{{\mathrm b}_t(z,\tz,x)}{b_t(z,x)}.\]
\item   For any $\ttz \in \V(z)$, it holds 
 \begin{equation}\label{conv2}
 \lim_{\gamma_\ell\to 0}  \left({\gamma_\ell }^2 A^{\gamma_\ell }_t(z,\ttz)\right)={{\mathbbm A}_t(z,\ttz)}\geq 0 \; \quad \mbox{and} \quad\; \lim_{\gamma_\ell\to 0}  \left({\gamma_\ell }^2 B^{\gamma_\ell }_t(z,\ttz)\right)={{\mathbbm B}_t(z,\ttz)}\geq 0,
 \end{equation}
with ${\mathbbm A}_t(z,\ttz)>0$ if and only if  $\ttz\in \V_{_\leftarrow}(z)$ and ${\mathbbm B}_t(z,\ttz)>0$ if and only if  $\ttz\in \V_{_\rightarrow}(z)$. 
Moreover, given $\ttz\in \V_{_\leftarrow}(z)$, for any $y\in  \widehat Y_z$
\[ \displaystyle {\mathbbm A}_t(z,\ttz):=\frac{{\mathbbm a}_t(z,\ttz,y)}{a_t(z,y)},\]
and given $\ttz\in \V_{_\rightarrow}(z)$, for any $x\in  \widehat X_z$  
\[{\mathbbm B}_t(z,\ttz):=\frac{{\mathbbm b}_t(z,\ttz,x)}{b_t(z,x)}.\]
\end{itemize}
\end{lemma}

Lemma \ref{lemintaibis}  provides the following  Taylor estimates for the functions $\varphi''_{\gamma_\ell }$ and $\psi''_{\gamma_\ell }$ as $\gamma_\ell $ goes to 0, which  are a key result of this paper.  

\begin{theorem}\label{limdevseconde} Assume that conditions \eqref{unifbounded1}, \eqref{unifbounded2} and \eqref{convhyp} are fulfilled. Let $(\gamma_\ell )_{\ell\in \N}$ be a sequence of positive numbers converging to 0 and $\widehat Q_t$ denotes the weak limit of the sequence of probability measures  $(\widehat Q_t^{\gamma_\ell })_{\ell\in \N}$. With  the notations of Lemma \ref{lemintaibis}, one has for any $t\in (0,1)$
\begin{align*}
&\liminf_{\gamma_\ell \to 0} \varphi''_{\gamma_\ell }(t)\\
&\geq \int  \Bigg[\Big(\sum_{\tz\in V_{_\leftarrow}(z)}
A_t(z,\tz) \,L(z,\tz)\Big)^2  + \sum_{\tz\in V_{_\leftarrow}(z),\, \ttz\in \V_{_\leftarrow}(z), \,\tz\sim \ttz} \rho \Big(A_t(z,\tz),{\mathbb A_t}(z,\ttz)\Big)  \, L(\tz, \ttz)L(z,\tz)\Bigg]  \,d\widehat Q_t(z)\\
&= \int  \Bigg[\Big(\sum_{\tz\in V(z)}
A_t(z,\tz) \,L(z,\tz)\Big)^2  + \sum_{\tz\in V(z),\, \ttz\in \V(z), \,\tz\sim \ttz} \rho \Big(A_t(z,\tz),{\mathbb A_t}(z,\ttz)\Big)  \, L(\tz, \ttz)L(z,\tz)\Bigg]  \,d\widehat Q_t(z),
\end{align*} 
and 
\begin{align*}
&\liminf_{\gamma_\ell \to 0} \psi''_{\gamma_\ell }(t)\\
&\geq \int \Bigg[ \Big(\sum_{\tz\in V_{_\rightarrow}(z)}
B_t(z,\tz) \,L(z,\tz)\Big)^2 +  \sum_{\tz\in V_{_\rightarrow}(z),\, \ttz\in \V_{_\rightarrow}(z),\, \tz\sim \ttz}  \rho\Big(B_t(z,\tz),\mathbbm{B}_t(z,\ttz)\Big)  \,
 L(\tz,\ttz)L(z,\tz)\Bigg] \,d\widehat Q_t(z)\\
 &=\int \Bigg[ \Big(\sum_{\tz\in V(z)}
B_t(z,\tz) \,L(z,\tz)\Big)^2 +  \sum_{\tz\in V(z),\, \ttz\in \V(z),\, \tz\sim \ttz}  \rho\Big(B_t(z,\tz),\mathbbm{B}_t(z,\ttz)\Big)  \,
 L(\tz,\ttz)L(z,\tz)\Bigg] \,d\widehat Q_t(z).
 \end{align*}
\end{theorem}

\begin{comments} Let us briefly explain  how to use this result.  First, adding the two above inequalities of this Theorem provides a  lower bound on the second derivative of the relative entropy along the Schr\"odinger path at zero temperature. Then, it remains to find good estimates of this lower bound to apply Lemma \ref{lemconvex} in order to get entropic curvature lower-bounds for the graph.  The  following equalities are a main guideline for this estimation, one has
\begin{align}\label{sumA}
\int \sum_{\tz\in V_{_\leftarrow}(z)} A_t(z,\tz) \,L(z,\tz)  \,d\widehat Q_t(z)&= 
\int \sum_{z\in\widehat Z} \sum_{\tz\in V_{_\leftarrow}(z)} {\mathrm a}_t(z,\tz,y) L(z,\tz) \,d\nu_1(y)\nonumber \\
&= \int \sum_{w\in \X} d(y,w) \,\widehat{\pi}_{_\leftarrow}(w|y) \,d\nu_1(y)\nonumber \\
&= W_1(\nu_0,\nu_1),
\end{align}
 and similarly
 \begin{equation}\label{sumB}
 \int \sum_{\tz\in V_{_\rightarrow}(z)} B_t(z,\tz) \,L(z,\tz)  \,d\widehat Q_t(z)= W_1(\nu_0,\nu_1),
 \end{equation}
 but also
\begin{align*}
 \int \sum_{ \ttz\in \V_{_\leftarrow}(z)} {\mathbb A_t}(z,\ttz) \,L^2(z,\ttz)  \,d\widehat Q_t(z)&= \int \sum_{w\in \X} d(y,w) (d(y,w)-1) \,\widehat{\pi}_{_\leftarrow}(w|y) \,d\nu_1(y)\\
 &=\iint d(x,y)(d(x,y)-1) \,d\widehat{\pi}(x,y),
 \end{align*}
 and
\[  \int \sum_{ \ttz\in \V_{_\rightarrow}(z)} {\mathbb B_t}(z,\ttz) \,L^2(z,\ttz)  \,d\widehat Q_t(z)=\iint d(x,y)(d(x,y)-1) \,d\widehat{\pi}(x,y).\] 
The easy proof of these equalities is left to the reader.
 \end{comments}

\begin{proof}[Proof of Theorem \ref{limdevseconde}.]
We only present the proof of the lower bound of $\liminf_{\gamma_\ell \to 0} \varphi''_{\gamma_\ell }(t)$ since by symmetry, identical  arguments provide the lower bound of  $\liminf_{\gamma_\ell \to 0} \psi''_{\gamma_\ell }(t)$. We start with the expression of $ \varphi_\gamma''(t)$  given by Lemma,  \ref{deriv2phibis}, for $t\in(0,1)$
 \begin{equation}\label{decomp}
   \varphi''_\gamma(t)=\int \left(M_t^\gamma +R_t^\gamma \right)\,d\widehat Q_{t}^\gamma,  
 \end{equation}
 with for any $z\in \X$,
 \begin{eqnarray*}
  M_t^\gamma(z):= \Big(  \sum_{\tz, \,\tz\sim z} e^{\nabla F_t^\gamma(z, \tz)} L^\gamma(z, \tz)\Big)^2 + \sum_{ \tz,\, \ttz,\, z\sim \tz \sim \ttz } \rho\left(e^{\nabla F_t^\gamma(z, \tz)},e^{\nabla F_t^\gamma(z, \ttz)}\right) 
L^\gamma(z, \tz )L^\gamma(\tz,\ttz ),
\end{eqnarray*}
and 
\[
R_t^\gamma(z):=\sum_{\tz,\, \tz\sim z} \left(1+\nabla F_t^\gamma(z, \tz)\right) \,e^{\nabla F_t^\gamma(z, \tz)}\left(L^\gamma(z,z)-L^\gamma(\tz,\tz) \right)\, L^\gamma(z, \tz).
\]

We will get the behaviour of $ \varphi_\gamma''(t)$ as $\gamma$ goes to zero by  applying Fatou's Lemma. For that purpose, we need first to bound from below the function  $\left(M_t^\gamma +R_t^\gamma \right)\widehat Q_t^\gamma$ uniformly in $\gamma$ by some integrable function with respect to the counting measure on $\X$.   Let us  first  lower bound $M_t^\gamma(z)$  and bound $|R_t^\gamma(z)|$ uniformly in $\gamma$, for $\gamma$ sufficiently small for any $z\in \X$. 

Recall that $\rho(a,b)=0$ as soon as $a=0$ or $b=0$, and $\rho(a,b)= (\log b-2\log a -1)b$. Therefore, easy computations give for any $a\geq 0$, 
\begin{equation}\label{infbG}
\inf_{b\geq 0} \rho(a,b)=-a^2,
\end{equation}
As a consequence, according to the definition of $A_t^\gamma$, one has 
  \[M_t^\gamma(z)\geq - \sum_{ \tz,\, \ttz,\, z\sim \tz \sim \ttz }   A_t^\gamma(z,z')^2 L^\gamma(z,\tz)L^\gamma(\tz,\ttz).\]
 From hypothesis \eqref{unifbounded1} and then applying inequality \eqref{encadrement}, it follows that  for any $z\in \X$ 
\begin{equation}\label{majM}
    M_t^\gamma(z)\geq -\gamma^2 S^2 d_{\max}^2\max_{\tz, \tz\sim z} A_t^\gamma(z,z')^2 \geq -\frac{\left(d^2(x_0,z)+1\right) \,K^{2 d(x_0,z)} \,O(1)}{ t^2}.
\end{equation}
where $x_0$ is a fixed point of $\X$, $K=2S/I $ and $O(1)$ denotes a positive constant that does not depend on $z,\gamma,t$.
Similarly, from \eqref{unifbounded1} and \eqref{encadrement}, one may show that 
\begin{equation}\label{majR}
 |R_t^\gamma(z)|\leq  \frac{\gamma }t \Big[ \log\left(\frac 1\gamma\right)+d(x_0,z)\Big]\,d(x_0,z)\, K^{d(x_0,z)}\,O(1)\leq \frac{|\gamma\log\gamma|}{t}\,d^2(x_0,z)\,K^{d(x_0,z)}\,O(1). 
\end{equation} 

Lemma \ref{lemmetech} \ref{item7} therefore implies for any $z \in \X$ and any $0\leq \gamma<\bar \gamma<1$,
\begin{eqnarray*}
(M_t^\gamma(z)+R_t^\gamma(z))\,\widehat Q_{t}^\gamma(z)
\geq -O(1)\left( \1_B(z) +\1_{\X\setminus B}(z)\, \bar\gamma\,\left(\bar\gamma K^2\right)^{[2d(x_0,z)-4D-1]_+}\right) \left(d^2(x_0,z)+1\right) \,K^{2 d(x_0,z)},
\end{eqnarray*}
where 
\[B:=\bigcup_{x\in\supp(\nu_0),y\in\supp(\nu_1)} [x,y]\quad \supset \widehat Z.\]
It remains to choose $\bar \gamma $ such that $(\bar \gamma K^3)^2<\gamma_o$ so that  hypothesis \eqref{convhyp} implies
\[\sum_{z\in \X} \left( \1_B(z) +\1_{\X\setminus B}(z) \,\bar \gamma\,\left(\bar \gamma K^2\right)^{[2d(x_0,z)-4D-1]_+}\right) \left(d^2(x_0,z)+1\right)\, K^{2 d(x_0,z)}<+\infty.\] 

Now,  conditions for Fatou's Lemma are fulfilled and one has 
\begin{eqnarray}\label{transit1}
\lim_{\gamma_\ell \to 0}  \varphi''_\gamma(t)\geq \sum_{z\in\X} \liminf_{\gamma_\ell \to 0} \left[\left(M_t^{\gamma_\ell }(z)+R_t^{\gamma_\ell }(z)\right)\widehat Q_{t}^{\gamma_\ell }(z)\right]>-\infty.
\end{eqnarray}
The weak convergence of $(\widehat Q^{\gamma_\ell })_\ell$ to $\widehat Q$ implies
$\lim_{\gamma_\ell \to 0} \widehat Q_{t}^{\gamma_\ell }(z)=\widehat Q_{t}(z)$,
and the inequality \eqref{majR} gives  
$\lim_{\gamma_\ell \to 0} R_t^{\gamma_\ell }(z)=0$  for any $z\in \X$.
As a consequence, 
\[\liminf_{\gamma_\ell \to 0} \left[\left(M_t^{\gamma_\ell }(z)+R_t^{\gamma_\ell }(z)\right)\widehat Q_{t}^{\gamma_\ell }(z)\right]=
\liminf_{\gamma_\ell \to 0}  \left[M_t^{\gamma_\ell }(z)\right] \,\widehat Q_{t}(z).\]
In order to complete the proof Proposition \ref{limdevseconde}, it remains to bound from below $\liminf_{\gamma_\ell \to 0}\left[M_t^{\gamma_\ell }(z)\right] $ for any $z \in \widehat Z$ since otherwise $\widehat Q_t(z)=0$. 
One has 
$M_t^{\gamma_\ell }= E_t^{\gamma_\ell }+F_t^{\gamma_\ell },$
where for any $z\in \widehat Z$, 
\[E_t^{\gamma_\ell }(z):=  \Big(  \sum_{\tz, \,\tz\sim z} \gamma_\ell A_t^{\gamma_\ell }(z,\tz) \,L(z, \tz)\Big)^2 -\sum_{ \tz, \,\ttz,\, z\sim \tz \sim \ttz } \!\!\!\!\!\!\gamma_\ell  ^2  A_t^{\gamma_\ell }(z,\ttz)^2 \,L(z, \tz )L(\tz,\ttz ), \]
and
\[F_t^{\gamma_\ell }(z)
= \sum_{ \tz, \,\ttz,\, z\sim \tz \sim \ttz } \gamma_\ell  ^2 \left[\rho\left(A_t^{\gamma_\ell }(z,\tz),A_t^{\gamma_\ell }(z,\ttz)\right) + A_t^{\gamma_\ell }(z,\ttz)^2 \right] 
L(z, \tz )L(\tz,\ttz ).\]

 Lemma \ref{lemintaibis} implies 
\begin{eqnarray}\label{transit2}
\lim_{\gamma_\ell \to 0} E_t^{\gamma_\ell }(z) =  \Big(  \sum_{\tz \in V_{_\leftarrow}(z)}A_t(z,\tz) \,L(z, \tz)\Big)^2- \sum_{ \tz\in V_{_\leftarrow}(z),\,\ttz\in \X,\, \ttz \sim \tz }   A_t(z,\tz)^2\, L(z, \tz )L(\tz,\ttz )  .
\end{eqnarray}
Assume that $\tz\in V_{_\leftarrow}(z)$, or equivalently  $\lim_{\gamma_\ell \to 0} \gamma_\ell  A_t^{\gamma_\ell }(z,\tz)\neq 0$. According to  
Lemma \ref{lemintaibis}, for any $\ttz\sim \tz$, one has  $\lim_{\gamma_\ell \to 0} \left(\gamma_\ell ^2
A_t^{\gamma_\ell }(z,\ttz)\right) =0$ if $d(z,\ttz)\leq1$ and $\lim_{\gamma_\ell \to 0} \left(\gamma_\ell ^2
A_t^{\gamma_\ell }(z,\ttz)\right) = {\mathbbm A}_t(z,\ttz)$ if $\ttz\in \V(z)$. As a consequence  
the continuity of the function $\rho$ on the set $(0,\infty)\times [0,\infty)$, 
implies
\begin{eqnarray*}
\lim_{\gamma_\ell \to 0}  \left[\rho\left( \gamma_\ell  A_t^{\gamma_\ell }(z,\tz), \gamma_\ell  ^2 A_t^{\gamma_\ell }(z,\ttz)\right) + \gamma_\ell  ^2A_t^{\gamma_\ell }(z,\tz)^2 \right]
  = \rho\Big( A_t(z,\tz),  {\mathbbm A}_t(z,\ttz)\Big)\1_{\ttz\in \V(z)} +  {A}_t(z,\tz)^2.
\end{eqnarray*}
If $\tz\in V(z)\setminus V_{_\leftarrow}(z)$, or equivalently $\lim_{\gamma_\ell \to 0} \gamma_\ell  A_t^{\gamma_\ell }(z,\tz)=A_t(z,\tz)= 0$, then  identity \eqref{infbG}  provides, according to the definition of the function $\rho$,
\begin{multline*}
\liminf_{\gamma_\ell \to 0}  \left[\rho\left( \gamma_\ell A_t^{\gamma_\ell }(z,\tz), \gamma_\ell  ^2 A_t^{\gamma_\ell }(z,\ttz)\right) + \gamma_\ell  ^2A_t^{\gamma_\ell }(z,\tz)^2 \right]\\
\geq 0=\rho(0,{\mathbbm A}_t(z,\ttz))= \rho( A_t(z,\tz),{\mathbbm A}_t(z,\ttz)) \1_{\ttz\in \V(z)}+   {A}_t(z,\tz)^2. 
\end{multline*}
As a consequence, one gets  
 \begin{multline*} 
 \liminf_{\gamma_\ell \to 0} F_t^{\gamma_\ell }(z) \geq \!\!\sum_{ \tz,\, \ttz, \,z\sim \tz \sim \ttz } \!\!  \left[\rho( A_t(z,\tz),{\mathbbm A}_t(z,\ttz))\1_{\ttz\in \V(z)}+  {A}_t(z,\tz)^2\right] \\
=  \!\!\sum_{ \tz\in V_{_\leftarrow}(z),\, \ttz\in \V_{_\leftarrow}(z) ,\,  \tz \sim \ttz } \!\!  \rho( A_t(z,\tz),{\mathbbm A}_t(z,\ttz)) + \sum_{ \tz\in V_{_\leftarrow}(z),\,\ttz\in \X,\, \ttz \sim \tz }   A_t(z,\tz)^2\, L(z, \tz )L(\tz,\ttz ).
\end{multline*}
This inequality together with \eqref{transit1} and \eqref{transit2} ends the proof of Theorem \ref{limdevseconde}.
\end{proof}

\newpage
\subsection{Application to specific examples of graphs}

\subsubsection{The lattice $\Z^n$ .}

\begin{proof}[Proof of Theorem \ref{thmZ}]
For any $z\in \Z^n$  and  any  $i\in[n]$, we note
$\sigma_{i+}(z)=z+e_i$ and $\sigma_{i-}(z)=z-e_i$. One has $\sigma_{i+}\sigma_{i-}=id$ and for $j\neq i$, $\sigma_{i+}\sigma_{j+}=\sigma_{j+}\sigma_{i+}$, $\sigma_{i+}\sigma_{j-}=\sigma_{j-}\sigma_{i+}$, $\sigma_{i-}\sigma_{j-}=\sigma_{j-}\sigma_{i-}$. We note 
\[A_{i+}(z):=A_t(z,\sigma_{i+}(z)), \quad A_{i+j+}(z):={\mathbbm A}_t(z,\sigma_{i+}\sigma_{j+}(z)),\qquad z\in \Z^n.\]
We define similarly $A_{i-},A_{i-j-},A_{i-j+}$.
Applying  Theorem  \ref{limdevseconde},   by symmetrisation one gets
\begin{align*}
\liminf_{\gamma_\ell \to 0} \varphi''_{\gamma_\ell }(t)
&\geq \int  \Big(\sum_{i=1}^n
(A_{i+}  + A_{i-})\Big)^2 \,d\widehat Q_{t}
+\int  \sum_{i=1}^n \Big(\rho\left(A_{i+},A_{i+i+}\right) + \rho\left(A_{i-},A_{i-i-}\right)\Big)\,d\widehat Q_{t}\\
&+\frac12\int  \sum_{i,j, i\neq j} \Big(\rho(A_{i+},A_{j+i+})+\rho(A_{j+},A_{j+i+})\Big) + \Big(\rho(A_{i-},A_{j-i-})+\rho(A_{j-},A_{j-i-})\Big)\\
&\qquad+\Big(\rho(A_{i+},A_{j-i+})+\rho(A_{j-},A_{j-i+})\Big) + \Big(\rho(A_{i-},A_{j+i-})+\rho(A_{j+},A_{j+i-})\Big) \, \,d\widehat Q_{t}.
\end{align*}
Identity \eqref{infbG} implies for any $a,a',b\in \R_+$, 
\begin{equation}\label{infbGG}
  \rho(a,b)+\rho(a',b)=2 \rho\left(\sqrt{aa'},b\right)\geq -2aa'.
   \end{equation}
   It follows that 
\begin{align*}
\liminf_{\gamma_\ell \to 0} \varphi''_{\gamma_\ell }(t)
  &\geq \int  \Big(\sum_{i=1}^n
(A_{i+}  + A_{i-})\Big)^2 \,d\widehat Q_t
-\int  \sum_{i=1}^n \Big(A_{i+}^2 +A_{i-}^2\Big) \,d\widehat Q_{t}\\
&\qquad-\int  \sum_{i,j, i\neq j} \Big(A_{i+}A_{j+}+A_{i-}A_{j-}+A_{i+}A_{j-}+A_{i-}A_{j+}\Big)\,d\widehat Q_t\\
&=2\int\sum_{i=1}^nA_{i+} A_{i-}\,d\widehat Q_t \geq 0.
\end{align*}
Identically one  proves  that 
$\displaystyle \liminf_{\gamma_\ell \to 0} \psi''_{\gamma_\ell }(t)
  \geq 0$.
Applying then Lemma \ref{lemconvex} ends the proof of Theorem~\ref{thmZ}.
\end{proof}

\subsubsection{The complete graph}

\begin{proof}[Proof of Theorem \ref{thmcomplete}]
Since for any $x,y\in X$, $d(x,y)=1$, Theorem \ref{limdevseconde} and Lemma \ref{lemintaibis} provide for any $t\in (0,1)$
\begin{align*}
\liminf_{\gamma_\ell \to 0} \varphi''_{\gamma_\ell }(t)
&\geq \int  \Big(\sum_{\tz\in V_{_\leftarrow}(z)}
A_t(z,\tz) \,L(z,\tz)\Big)^2 d\widehat Q_t(z)=  \iint  \Big(\sum_{\tz\in V_{_\leftarrow}(z)}
A_t(z,\tz) \,L(z,\tz)\Big)^2 d\widehat Q_{t,1}(z,y)\\
&= \int \sum_{z\in Z^y} \Bigg(\sum_{\tz\in V_{_\leftarrow}(z)}
\frac{{\mathrm a}_t(z, \tz,y)}{a_t(z,y)} \,L(z,\tz)\Bigg)^2 a_t(z,y) d\nu_1(y)
\end{align*}
 With the expression \eqref{expnu0complete} of ${Q_t}\!^{ x,y}$, one easily check that 
 for any $z\in \widehat Z,y\in \widehat Y_z$, or equivalently for any $y\in\supp(\nu_1), z\in \widehat Z^y$, 
 \[a_t(z,y)= (1-t) \,\widehat{\pi}_{_\leftarrow}(z|y)+t \,\delta_y(z),\]
 and with \eqref{a_t'}, for any $\tz\in V_{_\leftarrow}(z)$,
 \[{\mathrm a}_t(z, \tz,y)= \1_{z=y} \frac{ \widehat{\pi}_{_\leftarrow}(\tz|y)}{\mu(\tz)}.\]
As a consequence, one gets
\begin{align*}
& \int \sum_{z\in Z^y} \Bigg(\sum_{\tz\in V_{_\leftarrow}(z)}
\frac{{\mathrm a}_t(z, \tz,y)}{a_t(z,y)} \,L(z,\tz)\Bigg)^2 a_t(z,y) d\nu_1(y)
= \int \Big(\sum_{\tz\in V_{_\leftarrow}(y)} \frac{{\mathrm a}_t(y, \tz,y)}{a_t(y,y)} \mu(\tz)\Big)^2 \!\!a_t(y,y) \,d\nu_1(y)\\
&= \int \,\frac{\Big(1-\widehat{\pi}_{_\leftarrow}(y|y)\Big)^2}{1-(1-t)\left(1-\widehat{\pi}_{_\leftarrow}(y|y)\right)} \,d\nu_1(y)= \int \, \frac12 \Big(1-\widehat{\pi}_{_\leftarrow}(y|y)\Big)^2 h''\left((1-t)\left(1-\widehat{\pi}_{_\leftarrow}(y|y)\right)\right) d\nu_1(y)
= \xi_\leftarrow''(t),
\end{align*}
where  for any $t\in [0,1]$,
\[\xi_\leftarrow(t):=\frac12 \int h\left((1-t)(1-\widehat{\pi}_{_\leftarrow}(y|y)\right) \,  d\nu_1(y).\]
One similarly shows that for any $t\in (0,1)$,
\[\liminf_{\gamma_\ell \to 0} \psi''_{\gamma_\ell }(t)\geq \xi_\rightarrow''(t) ,\]
with 
$\xi_\rightarrow(t):=\frac12\int h\left(t(1-\widehat{\pi}_{_\rightarrow}(x|x)\right) \,  d\nu_0(x)$.
The proof of Theorem \ref{thmcomplete} ends  applying  Lemma \ref{lemconvex} and the two following identities
\[(1-t)\xi_\leftarrow(0)+t\xi_\leftarrow(1)-\xi_\leftarrow(t)=\frac{t(1-t)}2\int h_{1-t}\left(\int \1_{w\neq y} d \widehat{\pi}_{_\leftarrow}(w|y)\right) d\nu_1(y),\]
and
\[(1-t)\xi_\rightarrow(0)+t\xi_\rightarrow(1)-\xi_\rightarrow(t)=
\frac{t(1-t)}2\int h_t \left(\int \1_{w\neq x} d \widehat{\pi}_{_\rightarrow}(w|x)\right)d\nu_0(x). 
\]

Let us now compare $C_t(\widehat \pi)$ with a function of $W_1(\nu_0,\nu_1)$.
Observe that for any $y\in \supp(\nu_1)$, 
$ \int \1_{w\neq y} d \widehat{\pi}_{_\leftarrow}(w|y)
\neq 0$,
if and only if $y$ belongs to the set 
\[D_{_\leftarrow}:=\Big\{ w\in \supp(\nu_1)\,\Big|\, \exists x\in \X, w\neq x, (x,w)\in \supp(\widehat{\pi}) \Big\}.\] Since $h_{1-t}(0)=0$ and $h_{1-t}$ is convex, 
Jensen's inequality provides 
\[\int h_{1-t}\left(\int \1_{w\neq y} d \widehat{\pi}_{_\leftarrow}(w|y)\right) d\nu_1(y)\geq 
\nu_1(D_{_\leftarrow})\,h_{1-t} \left(\frac{ \iint \1_{w\neq y} d \widehat{\pi}_{_\leftarrow}(w|y) d\nu_1(y)  }{\nu_1(D_{_\leftarrow})} \right)=\nu_1(D_{_\leftarrow})\,h_{1-t} \left(\frac{W_1(\nu_0,\nu_1)}{\nu_1(D_{_\leftarrow})} \right).\]
Similarly one has 
\[\int h_{t}\left(\int \1_{w\neq x} d \widehat{\pi}_{_\rightarrow}(w|x)\right) d\nu_0(x)\geq 
\nu_0(D_{_\rightarrow})\,h_{t} \left(\frac{ \iint \1_{w\neq x} d \widehat{\pi}_{_\rightarrow}(w|x) d\nu_1(y)  }{\nu_1(D_{_\rightarrow})} \right)=\nu_0(D_{_\rightarrow})\,h_{t} \left(\frac{W_1(\nu_0,\nu_1)}{\nu_0(D_{_\rightarrow})} \right),\]
with 
\[D_{_\rightarrow}:=\Big\{ w\in \supp(\nu_0)\,\Big|\, \exists y\in \X, w\neq y, (w,y)\in \supp(\widehat{\pi}) \Big\}.\] 
According to  \eqref{totvar}, $W_1(\nu_0,\nu_1)\geq \nu_0(D_\rightarrow)-\nu_1(D_\rightarrow)$, and  
 we know  from Lemma \ref{opt} \ref{mon2} that  the sets $D_{_\leftarrow}$ and $D_{_\rightarrow}$ are disjoint. As a consequence,   
 \[\nu_0(D_\rightarrow)+\nu_1(D_\leftarrow)\leq 
 W_1(\nu_0,\nu_1)+\nu_1(D_\rightarrow)+\nu_1(D_\leftarrow)+ \leq W_1(\nu_0,\nu_1) +1.\]
This leads to the expected result \eqref{zut} :
\begin{align*}
C_t(\widehat \pi)&\geq  (1+W_1(\nu_0,\nu_1))\inf_{\alpha,\beta, 0<\alpha+\beta\leq 1}\left\{ \alpha h_t\left(\frac{W_1(\nu_0,\nu_1)}{\alpha(1+W_1(\nu_0,\nu_1))}\right)+\beta h_{1-t}\left(\frac{W_1(\nu_0,\nu_1)}{\beta(1+W_1(\nu_0,\nu_1))}\right)\right\},\\
&=(1+W_1(\nu_0,\nu_1)) \, k_t\left(\frac{W_1(\nu_0,\nu_1)}{1+W_1(\nu_0,\nu_1)}\right).
\end{align*}

In order to prove the estimate \eqref{med} of the function $k_t$, one  first observes that by construction, for any $t\in (0,1)$ and $v\in[0,1]$, 
\[h_t(v)=\frac12 \int_0^1 v^2 h''(uv)\, K_t(u) \, du= \int_0^1 \frac{v^2}{1-u v} K_t(u) \,du,\]
and since  $K_t(u)=K_{1-t}(1-u)$, 
 \[h_{1-t}(v)=  \int_0^1 \frac{v^2}{1-(1-u) v} K_t(u) \,du.\]
 Since $h_t(u)=+\infty$, for $u>1$, it follows that  for any $v\in [0,1/2]$,
 \begin{multline*}
k_t(v)=\inf_{\alpha,\beta, 0<\alpha+\beta\leq 1}\left\{ \alpha h_t\left(\frac v\alpha\right)+\beta h_{1-t}\left(\frac v\beta\right)\right\}\geq \inf_{\alpha,\beta,\alpha> v,\beta>v, \alpha+\beta\leq 1} \left\{ \alpha h_t\left(\frac v\alpha\right)+\beta h_{1-t}\left(\frac v\beta\right)\right\}\\ 
\geq  \int_0^1 v^2 \inf_{\alpha,\beta,\alpha>v,\beta>v, \alpha+\beta\leq 1}\left\{\frac{1}{\alpha-uv}+\frac{1}{\beta-(1-u)v}\right\} K_t(u) du .
\end{multline*}
 Easy computations give
 \begin{align*}
  &\inf_{\alpha,\beta,\alpha>v,\beta>v, \alpha+\beta\leq 1}\left\{\frac{1}{\alpha-uv}+\frac{1}{\beta-(1-u)v}\right\}=  \inf_{\alpha',\beta',\alpha'>(1-u)v,\beta'>uv, \alpha'+\beta'\leq 1-v}\left\{\frac1{\alpha'}+\frac1{\beta'}\right\}\\
  &\qquad\qquad\qquad\qquad\geq\inf_{\alpha',\beta', \alpha'>0,\beta'>0, \alpha'+\beta'\leq 1-v}\left\{\frac1{\alpha'}+\frac1{\beta'}\right\}=\frac{4}{1-v}.
  \end{align*}
  It provides the expected estimate \eqref{med}, namely $k_t(v)\geq \frac{4v^2}{1-v}$.
 \end{proof}

\subsubsection{Product probability measures on the discrete hypercube}\label{cube}
\begin{proof}[Proof of Theorem \ref{thmcube}]
The first step of the proof is to express the lower bounds on $\liminf_{\gamma_{\ell}\to 0} \varphi_{\gamma_\ell}''(t)$ and $\liminf_{\gamma_{\ell}\to 0} \psi_{\gamma_\ell}''(t)$ given by Theorem \ref{limdevseconde}
using the symmetries of the graph structure of the hypercube, and  keeping in mind the comments given next to Theorem \ref{limdevseconde}. This leads to the estimates  \eqref{aller1} and \eqref{aller2}. 
The second step is to prove that  each of the lower bound on $C_t(\widehat \pi)$ in  Theorem \ref{thmcube}  is a consequence of  these estimates. 

{\bf Step 1 :} Given $z\in \widehat Z$, let us define the sets 
\[I^{\leftarrow}(z):=\Big\{i\in [n]\,\Big|\, \sigma_i(z) \in V_{_\leftarrow}(z)\Big\}=\Big\{i\in [n]\,\Big|\, (z,\sigma_i(z))\in C_{_\leftarrow}\Big\},\]
\[I^{\rightarrow}(z):=\Big\{i\in [n]\,\Big|\, \sigma_i(z) \in V_{_\rightarrow}(z)\Big\}=\Big\{i\in [n]\,\Big|\, (z,\sigma_i(z))\in C_{_\rightarrow}\Big\},\]
\[{\mathbb I}^{\leftarrow}(z):=\Big\{(i,j)\in [n]\times[n]\,\Big|\, i\neq j,\sigma_i\sigma_j(z) \in {\mathbb V}_{_\leftarrow}(z)\Big\}=\Big\{(i,j)\in [n]\times[n]\,\Big|\, (z,\sigma_i\sigma_j(z))\in C_{_\leftarrow}\Big\},\]
\[{\mathbb I}^{\rightarrow}(z):=\Big\{(i,j)\in [n]\times[n]\,\Big|\, \sigma_i\sigma_j(z) \in {\mathbb V}_{_\rightarrow}(z)\Big\}=\Big\{(i,j)\in [n]\times[n]\,\Big|\, (z,\sigma_i\sigma_j(z))\in C_{_\rightarrow}\Big\}.\]
and 
$ {\mathbb I}_1^{\leftarrow}(z):=\big\{ i\in [n]\,|\, \exists j\in [n], (i,j)\in{\mathbb I}^{\leftarrow}(z)\big\}$, $ {\mathbb I}_1^{\rightarrow}(z):=\big\{ i\in [n]\,|\, \exists j\in [n], (i,j)\in{\mathbb I}^{\rightarrow}(z)\big\}$. Observe that if ${\mathbb I}^{\leftarrow}(z)\neq \emptyset$ then $|{\mathbb I}_1^{\leftarrow}(z)|\geq 2$.  Obviously  one has ${\mathbb I}_1^{\leftarrow}(z) \subset I^{\leftarrow}(z)$ and  since $\sigma_i\sigma_j=\sigma_j\sigma_i$, one has ${\mathbb I}^{\leftarrow}(z)=\{ (i,j)\, \,|\, i,j \in {\mathbb I}_1^{\leftarrow}(z), i\neq j\}$. 
Same remarks hold with the  sets  $ I^{\rightarrow}(z),{\mathbb I}_1^{\rightarrow}(z),{\mathbb I}^{\rightarrow}(z)$. 
The sets $C_{_\leftarrow}$ and $C_{_\rightarrow}$  are disjoints and therefore $I^{\leftarrow}(z)$ and $I^{\rightarrow}(z)$ are also disjoints. 
To simplify,  for $z\in \widehat Z$ and $i,j\in [n], i\neq j$ les us denotes 
\[A_i(z):= A_t(z,\sigma_i(z)),\quad {\mathbb A}_{ij}(z):={\mathbb A}_t(z,\sigma_i\sigma_j(z)), \quad L_i(z):=L(z,\sigma_i(z)).\]
Since for any $i\neq j $, $\sigma_i\sigma_j=\sigma_j\sigma_i$, one has ${\mathbb A}_{ij}={\mathbb A}_{ji}$, and observing that
 $L\left(\sigma_i(z), \sigma_j \sigma_i(z)\right) =L_j(z)$
Theorem \ref{limdevseconde} provides after symmetrization 
\begin{equation}\label{totocube}
\liminf_{\gamma_\ell \to 0} \varphi''_{\gamma_\ell }(t)\geq \int  \Big(\sum_{i\in I^{\leftarrow}}
A_i\,L_i\Big)^2 \,d\widehat Q_t+ \int\sum_{\{i,j\}\subset {\mathbb I}_1^{\leftarrow} } \Big[\rho\Big(A_i,{\mathbbm A}_{ij}\Big)+\rho\Big(A_j,{\mathbbm A}_{ij}\Big) \Big] L_iL_j\,d\widehat Q_t.
\end{equation}
Let ${\mathbb A}:=\sum_{\{i,j\}\subset {\mathbb I}_1^{\leftarrow}} 2 {\mathbb A}_{ij}L_{i}L_j$ and  
$\beta_{ij}:=\frac{2 {\mathbb A}_{ij}L_{i}L_j}{\mathbb A}$.
 According to the definition of the function $\rho$ given in Lemma \ref{deriv2phibis},  computations provide 
\begin{align}\label{ouf2}
&\sum_{\{i,j\}\subset {\mathbb I}_1^{\leftarrow}} \Big[\rho\Big(A_i,{\mathbbm A}_{ij}\Big) +\rho\Big(A_j,{\mathbbm A}_{ij}\Big) \Big] L_i L_j\nonumber\\
&={\mathbb A}\log {\mathbb A} -{\mathbb A} + {\mathbb A}\sum_{\{i,j\}\subset{\mathbb I}_1^{\leftarrow}} \beta_{ij}\log ( \beta_{ij})-{\mathbb A}\sum_{\{i,j\}\subset {\mathbb I}_1^{\leftarrow}} \log( 2A_iA_j) \,\beta_{ij}\nonumber \\
&\geq {\mathbb A}\log {\mathbb A} -{\mathbb A} + {\mathbb A} \log \sum_{\{i,j\}\subset {\mathbb I}_1^{\leftarrow}}  2A_iA_j,
\end{align}
 where the last inequality follows from the duality formula between the log-Laplace transform and the entropy, namely in this case 
 \[\sup_{\beta}\left\{ \sum_{\{i,j\}\subset {\mathbb I}_1^{\leftarrow}}  \log( 2A_iA_j)  \,\beta_{ij}- \sum_{\{i,j\}\subset{\mathbb I}_1^{\leftarrow}}  \beta_{ij}\log ( \beta_{ij}) \right\}= \log \sum_{\{i,j\}\subset {\mathbb I}_1^{\leftarrow}} 2A_iA_j , \]
 where the supremum runs over all probabilities $\beta$ on ${\mathbb I}_1^{\leftarrow}$. Note that $\mathbb A\neq 0$ if and only if $|{\mathbb I}_1^{\leftarrow}|\geq 2$ and therefore $\sum_{\{i,j\}\subset {\mathbb I}_1^{\leftarrow}} 2A_iA_j>0$. It follows that  all quantities above are well defined.  Setting $A:=\sum_{i\in I^{\leftarrow}} A_i L_i$ and  ${\widetilde A}^2:=\sum_{i\in I^{\leftarrow}} (A_i L_i)^2$, since ${\mathbb I}_1^{\leftarrow}\subset I^{\leftarrow}$,
\eqref{totocube} and  \eqref{ouf2} finally give the following lower-bound
 \begin{equation}\label{aller1}
 \liminf_{\gamma_\ell \to 0} \varphi''_{\gamma_\ell }(t)\geq \int \left[A^2 -{\mathbb A} + {\mathbb A}\left(\log {\mathbb A}-\log\left(A^2-\widetilde A^2\right)\right)\1_{|{ I}^{\leftarrow}|\geq 2} \right] d\widehat Q_t.
 \end{equation}
  From the lower-bound of $ \liminf_{\gamma_\ell \to 0} \psi''_{\gamma_\ell }(t)$
  given by Theorem \ref{thmcube}, following the same lines of proof one gets 
  \begin{equation}\label{aller2}
 \liminf_{\gamma_\ell \to 0} \psi''_{\gamma_\ell }(t)\geq \int \left[B^2 -{\mathbb B} + {\mathbb B}\left(\log {\mathbb B}-\log\left(B^2-\widetilde B^2\right)\right)\1_{|{ I}^{\rightarrow}|\geq 2}  \right] d\widehat Q_t,
 \end{equation}
 where we set for any $z\in \widehat Z$
 \[B(z):=\sum_{i\in I^\rightarrow} B_t(z,\sigma_i(z))\,  L_i(z),\quad \widetilde B^2(z):=\sum_{i\in I^\rightarrow} B^2_t(z,\sigma_i(z))\,  L_i^2(z),\quad {\mathbb B}(z):=\sum_{\{i,j\}\subset {\mathbb I}_1^\rightarrow} {\mathbb B}_t(z,\sigma_i\sigma_j(z))\, L_{i}(z) L_{j}(z). \]

{\bf Step 2 :} By the Cauchy-Schwarz inequality $\widetilde A^2\geq A^2/ |I^\leftarrow|$ 
and therefore  \eqref{aller1}  gives
\begin{align}\label{aller1'}
 \liminf_{\gamma_\ell \to 0} \varphi''_{\gamma_\ell }(t)&\geq \int \left[A^2 -{\mathbb A} + {\mathbb A}\left(\log {\mathbb A}-\log(A^2)-\log\left( 1-\frac{1}{|I^\leftarrow|}\right) \right)\1_{|{ I}^{\leftarrow}|\geq2}\right] d\widehat Q_t\nonumber \\
 & 
 \geq \int \max \left\{ \frac{A^2}{|I^\leftarrow|}\, , -\log\Big( 1-\frac{1}{|I^\leftarrow|}\Big) \, {\mathbb A} \right\} d\widehat Q_t\nonumber\\
 &\geq  \int  \frac{1}{|I^\leftarrow|}  \max \left\{ A^2,  {\mathbb A} \right\} d\widehat Q_t
 \end{align}
 where the last inequalities follows from the concavity property of the logarithmic function and since ${A^2}/{|I^\leftarrow|}=A^2\1_{|I^\leftarrow|=1}+{A^2}/{|I^\leftarrow|}\1_{|I^\leftarrow|\geq2}$. Identically \eqref{aller2} implies 
 \begin{equation}\label{aller2'}
 \liminf_{\gamma_\ell \to 0} \psi''_{\gamma_\ell }(t)\geq  \int  \frac{1}{|I^\rightarrow|}  \max \left\{ B^2,  {\mathbb B} \right\} d\widehat Q_t  
 \end{equation}

Keeping only the quantities involving $A$ and $B$ in \eqref{aller1'} and \eqref{aller2'}, and applying Cauchy-Schwarz inequality, the identities \eqref{sumA} and \eqref{sumB} yield 
 \[\liminf_{\gamma_\ell \to 0} \varphi''_{\gamma_\ell }(t)+ \liminf_{\gamma_\ell \to 0} \psi''_{\gamma_\ell }(t)\geq \, W_1^2(\nu_0,\nu_1) \left[\frac{1}{ \int |I^\leftarrow| \, d\widehat Q_t} +\frac{1}{ \int |I^\rightarrow| \, d\widehat Q_t} \right].\]
 Since, the sets $I^\leftarrow$ and $I^\rightarrow$ are disjoint $  \int |I^\leftarrow| \, d\widehat Q_t+ \int |I^\rightarrow| \, d\widehat Q_t \leq n$, and therefore  the identity  $\min_{\alpha,\beta>0, \alpha+\beta\leq 1}\left\{\frac{1}\alpha+\frac{1}\beta\right\}=4$ implies
 \[\liminf_{\gamma_\ell \to 0} \varphi''_{\gamma_\ell }(t)+ \liminf_{\gamma_\ell \to 0} \psi''_{\gamma_\ell }(t)\geq \frac{4}{n}\,  W_1^2(\nu_0,\nu_1).\]
Then applying Lemma \ref{lemconvex}, this estimate give the first lower bound of $C_t(\widehat \pi)$ in Theorem \ref{thmcube}.

Keeping only the quantities involving ${\mathbb A} $ and ${\mathbb B} $ in \eqref{aller1'} and \eqref{aller2'}, one gets  
 \begin{equation}\label{oouf}
 \liminf_{\gamma_\ell \to 0} \varphi''_{\gamma_\ell }(t)+ \liminf_{\gamma_\ell \to 0} \psi''_{\gamma_\ell }(t)\geq   \int  \frac{\mathbb A}{  |I^\leftarrow| } +\frac{{\mathbb B} }{  |I^\rightarrow| }\, d\widehat Q_t .
 \end{equation}
 According to   Lemma \ref{lemintaibis} and \eqref{a_t''},    for any $i,j\in[n]$  with $i\neq j$, for any $z\in\widehat Z$ and $y\in\widehat Y_z$,
 \[ 
 2\displaystyle {\mathbbm A}_t(z,\sigma_j\sigma_i(z)) L_i(z)L_j(z):=\frac{{\mathbbm a}_t(z,\sigma_j\sigma_i(z),y) 2L_i(z)L_j(z)}{a_t(z,y)}
 \]
with 
 \begin{equation*}
  {\mathbbm a}_t(z,\sigma_j\sigma_i(z),y):=\!\!\!\!\!\sum_{w, (z, \sigma_i\sigma_j(z))\in [y,w]}  
\!\!\!\!\!r(y,z,\sigma_i\sigma_j(z),w) \,d(y,w)(d(y,w)-1)\, \B_t^{d(y,w)-2}(d(z,w)-2)\,\widehat{\pi}_{_\leftarrow}(w|y).
\end{equation*}
From the identity
\[\sum_{\{i,j\}\subset{\mathbb I}_1^{\leftarrow}}  r(y,z,\sigma_i\sigma_j(z),w) L^2(z,\sigma_i\sigma_j(z)) = r(y,z,z,w),\]
and since $L^2(z,\sigma_i\sigma_j(z))=2L_i(z)L_j(z)$ one has for any $z\in\widehat Z$ and $y\in\widehat Y_z$,
\[{\mathbb A}(z) =\frac{1}{a_t(z,y)} \sum_{w\in \X} \sum_{z,z\in [y,w]} r(y,z,z,w) d(y,w)(d(y,w)-1)\, \B_t^{d(y,w)-2}(d(z,w)-2)\,\widehat{\pi}_{_\leftarrow}(w|y).\]
Working identically with ${\mathbb B}(z)$ one finally gets 
\begin{equation}\label{osteo2}
 \int  \frac{\mathbb A}{  |I^\leftarrow| } + \frac{{\mathbb B} }{  |I^\rightarrow| } \, d\widehat Q_t \geq \iint c_t(x,y) \,d\widehat{\pi}(x,y),
 \end{equation}
where 
\begin{multline*}
c_t(x,y)=\sum_{z\in[x,y]}  \frac1{|I^\leftarrow(z)|}\,r(x,z,z,y) d(x,y)(d(x,y)-1)\, \B_t^{d(x,y)-2}(d(x,z)-2) \\
+\sum_{z\in[x,y]}  \frac1{|I^\rightarrow(z)|} r(x,z,z,y) d(x,y)(d(x,y)-1)\, \B_t^{d(x,y)-2}(d(x,z)).
\end{multline*}
Since $ |I^\leftarrow| \leq n$ and $ |I^\rightarrow|\leq n$ and for any $k\in\{0,\ldots,d(x,y)\}$, $\sum_{z\in[x,y], d(x,z)=k}r(x,z,z,y) =1$, it follows that  
\begin{equation}\label{osteo}
c_t(x,y)\geq \frac 2n  d(x,y)(d(x,y)-1).
\end{equation}
For large values of $d(x,y)$, this lower bound can be improved using the fact that $I^\leftarrow$ and  $I^\rightarrow$ are disjoint and therefore $ |I^\leftarrow| + |I^\rightarrow| \leq n$. By first rewriting  $c_t(x,y)$,  applying Cauchy-Schwarz inequality, and then using the identity $\inf_{\alpha>0,\beta>0, \alpha+\beta \leq 1}\left\{ \frac{u^2}\alpha +\frac{v^2}{\beta}\right\}= (u+ v)^2, \, u,v\geq 0$, one gets 
\begin{align*}
c_t(x,y)&=\sum_{z,z\in [x,y]} \left( \frac{d(x,z)(d(x,z)-1)}{ |I^\leftarrow(z)| t^2}+\frac{d(z,y)(d(z,y)-1)}{|I^\rightarrow(z)| (1-t)^2}\right)\,  {Q_t}\!^{ x,y}(z)\\
&\geq \frac{\left(\int \sqrt{d(x,z)(d(x,z)-1)} d{Q_t}\!^{ x,y}(z)\right)^2}{t^2\int I^\leftarrow d{Q_t}\!^{ x,y}} + \frac{\left(\int \sqrt{d(z,y)(d(z,y)-1)} d{Q_t}(z)\!^{ x,y}(z)\right)^2}{(1-t)^2\int I^\rightarrow d{Q_t}\!^{ x,y}}
 \\&\geq\frac1n\left(\int \frac{\sqrt{d(x,z)(d(x,z)-1)}}{t} d{Q_t}\!^{ x,y}(z)+ \int \frac{\sqrt{d(z,y)(d(z,y)-1)}}{1-t} d{Q_t}\!^{ x,y}(z)\right)^2 \\
 &=\frac4n \,v_t(d(x,y)),
 \end{align*}
 with for any $d\in\N$, 
 \begin{equation*}
  v_t(d)
=\frac14\left(\sum_{k=0}^{d} \sqrt{k(k-1)}\, \Big(\frac{\B_t^{d} (k)}t+\frac{\B_{1-t}^{d} (k)}{1-t}\Big) \right)^2. 
\end{equation*}
Then applying Lemma \ref{lemconvex} together with \eqref{oouf}, \eqref{osteo2}, \eqref{osteo} provides the following  lower bound on the cost  $C_t(\widehat \pi)$, 
\begin{equation*}
C_t(\widehat \pi) 
\geq   \frac4n \iint  w_t(d(x,y))\,d\widehat{\pi}(x,y) 
\end{equation*}
with \[w_t(d):=\max\left\{\frac{d(d-1)}2,\int_0^1 v_s(d) \,K_t(s)\,ds\right\},\qquad d\in\N.\]
The proof of the second lower bound on $C_t(\widehat \pi)$ ends from the next estimate of the quantity $\int_0^1 v_s(d) \,K_t(s)\,ds$.
Since for any $s\in(0,1)$ and $d\in \N$, one has 
\begin{align*}
v_s(d)&=\frac14 \left(2d-\sum_{k=1}^{d} \frac{k}{k+\sqrt{k(k-1)}}\, \Big(\frac{\B_t^{d} (k)}t+\frac{\B_{1-t}^{d} (k)}{1-t}\Big) \right)^2 \\
&\geq d^2-d \sum_{k=1}^{d}  \Big(\frac{\B_t^{d} (k)}t+\frac{\B_{1-t}^{d} (k)}{1-t}\Big)\geq d^2-d  \Big(\frac{1-\B_t^{d} (0)}t+\frac{1-\B_{1-t}^{d} (0)}{1-t}\Big),
\end{align*}
it follows that  for any $t\in (0,1)$
\begin{equation*}
\int_0^1 v_s(d) \,K_t(s)\,ds\geq 
d^2-d \int_0^1 \Big(\frac{1-\B_s^{d} (0)}s+\frac{1-\B_{1-s}^{d} (0)}{1-s}\Big)\,K_t(s)\,ds,
\end{equation*}
with for $d\geq 1$,
\begin{align*}
&\int_0^1   \Big(\frac{1-\B_s^{d} (0)}s+\frac{1-\B_{1-s}^{d} (0)}{1-s}\Big)  \,K_t(s) \,ds
=  2\sum_{k=1}^{d}\left((1-t)^{k-1}+t^{k-1}\right)\left(\frac1{k}-\frac1{d+1}\right)\\
&\leq 
2\left(1+\sum_{k=2}^{d}\frac1k\right)\leq 2 +2 \log{d}.
\end{align*}

For the proof of third lower bound on $C_t(\widehat \pi)$, one uses again \eqref{aller1} and \eqref{aller2} with the concavity of the logarithmic function to obtain
\begin{equation}\label{oouf1}
\liminf_{\gamma_\ell \to 0} \varphi''_{\gamma_\ell }(t)+ \liminf_{\gamma_\ell \to 0} \psi''_{\gamma_\ell }(t) \geq \int \widetilde A^2 \, d\widehat Q_t  +  \int \widetilde B^2 \, d\widehat Q_t.
\end{equation} 
According to the definition of $\widetilde A^2$,
\[\int \widetilde A^2 \, d\widehat Q_t = \sum_{i=1}^n \int   \sum_{z\in E_i^{\leftarrow}(y)}\Big(A_i(z)\,L_i(z)\Big)^2  a_t(z,y) \,d\nu_1(y),
\]
where $E_i^{\leftarrow}(y):=\Big \{z\in\widehat Z^y\,\Big |\,  y\in\widehat Y_{(z,\sigma_i(z))} \Big\}$
for any $y\in \supp(\nu_1)$.
Easy computations give 
\[ \sum_{z\in E_i^{\leftarrow}(y)} A_i(z) \,L_i(z)\,a_t(z,y)=\sum_{z\in E_i^{\leftarrow}(y)}  \,{\mathrm a}_t(z,\sigma_i(z),y)\,L_i(z)
= \Pi^i_\leftarrow(y).\]
and therefore by the Cauchy-Schwarz inequality
\[\int \widetilde A^2 \, d\widehat Q_t  \geq \sum_{i=1}^n  \frac{\Pi^i_\leftarrow(y)^2}{\sum_{z\in E_i^{\leftarrow}(y)} a_t(z,y)}.\]
If $z\in E_i^{\leftarrow}(y)$ then $z_i=y_i$ and therefore 
\[\sum_{z\in E_i^{\leftarrow}(y)} a_t(z,y)\leq 1- \sum_{z\in\widehat Z^y}  \1_{z_i\neq y_i} a_t(z,y).\]
From the definition \eqref{a_t} of $a_t(z,y)$,   
and observing that if $z\in [y,w]$ and $z_i\neq y_i$ then necessarily $z_i=w_i$, one gets 
\[\sum_{z\in\widehat Z^y} \1_{z_i\neq y_i} a_t(z,y)=\sum_{w\in\{0,1\}^n}  \1_{w_i\neq y_i} \left(\sum_{z\in[y,w]} \1_{z_i= w_i} (1-t)^{d(y,z)} t^{d(z,w)}\right) \widehat\pi_{_\leftarrow}(w|y)= (1-t) \sum_{w\in\{0,1\}^n}  \1_{w_i\neq y_i}\widehat\pi_{_\leftarrow}(w|y),   \]
and therefore 
\[\int \widetilde A^2 \, d\widehat Q_t \geq \sum_{i=1}^n \sum_{y\in\{0,1\}^n}  \frac{\Pi^i_\leftarrow(y)^2\,{\nu_1(y)}}{1-(1-t)  \,\Pi^i_\leftarrow(y)}\,.\] 
This inequality   implies (as  in the proof of Theorem \ref{thmcomplete}) 
for any $t\in(0,1)$
 \begin{equation*}
\int \widetilde A^2 \, d\widehat Q_t \geq \xi_\leftarrow''(t),\quad\mbox{with}\quad \xi_\leftarrow(t):= \frac12\int \sum_{i=1}^n h\Big((1-t) \,\Pi^i_\leftarrow(y)\Big) \,d\nu_1(y).
 \end{equation*}
  Similar computations   with the quantity $\int \widetilde B^2 \, d\widehat Q_t $ and \eqref{oouf1}  finally provide 
   \begin{equation*}
  \liminf_{\gamma_\ell \to 0} \varphi''_{\gamma_\ell }(t)+ \liminf_{\gamma_\ell \to 0} \psi''_{\gamma_\ell }(t) \geq   \xi_\leftarrow''(t)+\xi_\rightarrow''(t),
   \end{equation*}
 with
 $\xi_\rightarrow(t):= \frac12\int \sum_{i=1}^n h\Big(t \,  \Pi^i_\rightarrow(x)\Big)\, d\nu_0(x).$
 Following the proof of Theorem \ref{thmcomplete}, the two above estimates  yield the third lower bound of $C_t(\widehat \pi)$.
\end{proof}

\subsubsection{The circle $\Z/N\Z$}

\begin{proof}[Proof of Theorem \ref{thmcircle}] Let us note $n'=\lceil N/2\rceil$ where $\lceil \cdot\rceil$ denotes  the ceiling function. Let $y\in\supp(\nu_1)\subset \Z/N\Z$,  and $z \in \widehat Z^y$.
We observe that if $\{w\in \Z/N\Z\,|\, (z,z-1)\in[y,w]\}\neq \emptyset$ then necessarily $(z-1,z)\in[y+n',y]$
and if  $\{w\in \Z/N\Z\,|\, (z,z+1)\in[y,w]\}\neq \emptyset$ then necessarily $(z,z+1)\in[y,y+n]$. 
As a consequence, since the sets $\{z\in \Z/N\Z\,|\, (z,z+1)\in[y,y+n]\}$ and $\{z\in \Z/N\Z\,|\, (z-1,z)\in[y+n',y]\}$ are disjoints, the sets $\{z\in \widehat Z^y\,| \,y\in \widehat Y_{(z,z+1)} \}$ and $\{z\in \widehat Z^y\,| \,y\in \widehat Y_{(z,z-1)} \}$
are also disjoints. It follows that 
\begin{eqnarray*}
\int \sum_{z\in \widehat Z^y}  \Big(
A_{t}(z,z+1)  + A_{t}(z,z-1)\Big)^2 \,a_t(z,y)\, d\nu_1(y)=\int \sum_{z\in \widehat Z^y} \Big(
A_{t}^2(z,z+1)  + A_{t}^2(z,z-1)\Big) \,a_t(z,y)\, d\nu_1(y).
\end{eqnarray*}
Therefore Theorem  \ref{limdevseconde} together with  \eqref{infbG} provide
\begin{align*}
\liminf_{\gamma_\ell \to 0} \varphi''_{\gamma_\ell }(t)
&\geq \int \sum_{z\in \widehat Z^y} \Big(
A_{t}^2(z,z+1)  + A_{t}^2(z,z-1)\Big) +  \rho\Big(A_t(z,z+1),{\mathbbm A}_t(z,z+2)\Big) \\
&\qquad\qquad\qquad\qquad\qquad\qquad\qquad+ \rho\Big(A_t(z,z-1),{\mathbbm A}_t(z,z-2)\Big)  \,a_t(z,y)\, d\nu_1(y)\\
&\geq 0
\end{align*}
Identically one proves that 
$\displaystyle \liminf_{\gamma_\ell \to 0} \psi''_{\gamma_\ell }(t)\geq 0.$
The proof of Theorem \ref{thmcircle} ends  applying Lemma \ref{lemconvex}.
\end{proof}

\subsubsection{The Bernoulli-Laplace model}

\begin{proof}[Proof of Theorem \ref{thmslicecube}]
One follows the same strategy as for the proof of Theorem \ref{thmcube}. As a first step, the geometric structure of the slices of the cube provides estimates of the lower lower bounds on $\liminf_{\gamma_{\ell}\to 0} \varphi_{\gamma_\ell}''(t)$ and $\liminf_{\gamma_{\ell}\to 0} \psi_{\gamma_\ell}''(t)$ given by Theorem \ref{limdevseconde}. In the second step, one explains how these estimates (namely \eqref{aller1''} and \eqref{aller2''})  imply each of the lower bound on $C_t(\widehat \pi)$ given by  Theorem \ref{thmslicecube}.  

{\bf Step 1 :}  For $z\in \widehat Z$, one defines  the sets 
\[I^{\leftarrow}(z):=\Big\{(i,j)\in J_0(z)\times J_1(z)\,\Big | \, (z,\sigma_{ij}(z))\in C_{_\leftarrow}\Big\},\]
\[I^{\rightarrow}(z):=\Big\{(i,j)\in J_0(z)\times J_1(z)\,\Big | \, (z,\sigma_{ij}(z))\in C_{_\rightarrow}\Big\},\]
\[{\mathbb I}^{\leftarrow}(z):=\Big\{((i,j),(k,l))\in (J_0(z)\times J_1(z))^2\,\Big|\, (i,j)\neq (k,l),\sigma_{kl}\sigma_{ij}(z) \in {\mathbb V}_{_\leftarrow}(z)\Big\},\]
\[{\mathbb I}^{\rightarrow}(z):=\Big\{((i,j),(k,l))\in (J_0(z)\times J_1(z))^2\,\Big|\, (i,j)\neq (k,l),\sigma_{kl}\sigma_{ij}(z) \in {\mathbb V}_{_\rightarrow}(z)\Big\},\]
\[{\mathbb I}_1^\leftarrow(z):=\Big\{(i,j)\in J_0(z)\times J_1(z)\,\Big |\, \exists (k,l)\in J_0(z)\times J_1(z),  ((i,j),(k,l))\in {\mathbb I}^\leftarrow(z)\Big\},\]
\[{\mathbb I}_1^\rightarrow(z):=\Big\{(i,j)\in J_0(z)\times J_1(z)\,\Big |\, \exists (k,l)\in J_0(z)\times J_1(z),  ((i,j),(k,l))\in {\mathbb I}^\rightarrow(z)\Big\},\]
 The sets $I^{\leftarrow}(z)$ and $I^{\rightarrow}(z)$ are  disjoints since $C_{_\leftarrow}\cap C_{_\rightarrow}=\emptyset$.
Obviously  one has ${\mathbb I}_1^{\leftarrow}(z) \subset I^{\leftarrow}(z)$. Observe  that $\sigma_{ki}\sigma_{ij}(z)=\sigma_{kj}(z)$  so that $d(z,\sigma_{ki}\sigma_{ij}(z))=1$ and similarly $d(z,\sigma_{jl}\sigma_{ij}(z))=1$. It follows that if $((i,j),(k,l))\in {\mathbb I}^{\leftarrow}(z)$ or $((i,j),(k,l))\in {\mathbb I}^{\leftarrow}(z)$, then the indices $i,j,k,l$ all differ and $\sigma_{kl}\sigma_{ij}(z)=\sigma_{ij}\sigma_{kl}(z)$.
 As a consequence  one has ${\mathbb I}^{\rightarrow}(z)=\{ ((i,j),(k,l))\, \,|\, \{(i,j),(k,l)\} \subset {\mathbb I}_1^{\leftarrow}(z),i\neq k, j\neq l\}$. 
Same remarks hold with the  sets  $ I^{\rightarrow}(z),{\mathbb I}_1^{\rightarrow}(z),{\mathbb I}^{\rightarrow}(z)$. 
To simplify, one denotes $A_{ij}(z):=A_t(z,\sigma_{ij}(z))$ and $A_{kl,ij}(z):={\mathbbm A}_t(z,\sigma_{kl}\sigma_{ij}(z))$.
After symmetrization, Theorem \ref{limdevseconde} provides
\begin{equation}\label{trop1}
\liminf_{\gamma_\ell \to 0} \varphi''_{\gamma_\ell }(t)\geq \int  \Big(\sum_{(i,j)\in I^{\leftarrow}}
A_{ij}\Big)^2 \,d\widehat Q_t+ \int\sum_{\{(i,j),(k,l)\}\subset {\mathbb I}_1^{\leftarrow}, i\neq k, j\neq l} \Big[\rho\Big(A_{ij},{\mathbbm A}_{kl,ij}\Big)+\rho\Big(A_{ij},{\mathbbm A}_{kl,ij}\Big) \Big] \,d\widehat Q_t.
\end{equation}
Setting ${\mathbb A}:=\sum_{\{(i,j),(k,l)\}\subset {\mathbb I}_1^{\leftarrow}, i\neq k, j\neq l} 2 {\mathbb A}_{kl,ij}$, 
$\beta_{kl,ij}:=\frac{2 {\mathbb A}_{kl,ij}}{\mathbb A}$, 
 according to the definition of the function $\rho$ given in Lemma \ref{deriv2phibis},  easy computations provides
\begin{align}\label{trop2}
\sum_{\{(i,j),(k,l)\}\subset {\mathbb I}_1^{\leftarrow}, i\neq k, j\neq l } &\Big[\rho\Big(A_{ij},{\mathbbm A}_{kl,ij}\Big)+\rho\Big(A_{ij},{\mathbbm A}_{kl,ij}\Big) \Big]
\nonumber\\&={\mathbb A}\log {\mathbb A} -{\mathbb A} + {\mathbb A}\sum_{\{(i,j),(k,l)\}\subset {\mathbb I}_1^{\leftarrow} , i\neq k, j\neq l} \beta_{kl,ij}\log ( \beta_{kl,ij})-{\mathbb A}\sum_{\{(i,j),(k,l)\}\subset {\mathbb I}_1^{\leftarrow}, i\neq k, j\neq l} \log( 2A_{ij}A_{kl}) \,\beta_{ij,kl}\nonumber \\
&\geq {\mathbb A}\log {\mathbb A} -{\mathbb A} - {\mathbb A} \log \sum_{\{(i,j),(k,l)\}\subset {\mathbb I}_1^{\leftarrow}, i\neq k, j\neq l }  2A_{ij}A_{jk},
\end{align}
 where the last inequality follows from the duality formula between the log-Laplace transform and the entropy. 
 For $z\in \widehat Z$, let 
 \[J_0^{\leftarrow}(z):=\{i\in J_0(z)\,|\, \exists j\in J_1(z), (i,j)\in I^{\leftarrow}\},\qquad J_1^{\leftarrow}(z):=\{j\in J_1(z)\,|\, \exists i\in J_1(z), (i,j)\in I^{\leftarrow}\}, \] 
 and let us define identically $J_0^{\rightarrow}(z)$ and $J_1^{\rightarrow}(z)$ by replacing the set $I^{\leftarrow}$ by the $I^{\rightarrow}$.
 If $i\in J_0^{\leftarrow}(z) \cap J_0^{\rightarrow}(z)$ then there exists $j$ and $l$ in $J_1(z)$  such that $(z,\sigma_{ij}(z))$ and $(\sigma_{il}(z), z)$ are points of $C_{_{\rightarrow}}$. According to Lemma \ref{opt} $i)$, this is impossible since $d(\sigma_{ij}(z), \sigma_{il}(z))\leq 1$. It follows that $J_0^{\leftarrow}(z)\cap J_0^{\rightarrow}(z)=\emptyset$ and identically one proves that $J_1^{\leftarrow}(z) \cap J_1^{\rightarrow}(z)=\emptyset$.
 Let $A:=\sum_{(i,j)\in I^{\leftarrow}} A_{ij}$. Since ${\mathbb I}_1^{\leftarrow}\subset { I}^{\leftarrow}$, one  checks that 
\begin{align*}
\sum_{\{(i,j),(k,l)\}\subset {\mathbb I}_1^{\leftarrow} , i\neq k, j\neq l}  2A_{ij}A_{jk}&\leq \sum_{((i,j), (k,l))\in{ I}^{\leftarrow} \times { I}^{\leftarrow} , i\neq k, j\neq l} A_{ij}A_{kl}\\
&=A^2 + \sum_{(i,j)\in I^{\leftarrow}} A_{ij}^2 -\sum_{i\in J_0^{\leftarrow}}\Big(\sum_{j\in J_1^{\leftarrow}} A_{ij}\Big)^2-\sum_{j\in J_1^{\leftarrow}}\Big(\sum_{i\in J_0^{\leftarrow}} A_{ij}\Big)^2.
\end{align*}
Therefore, setting 
\[ {\widetilde A}^2:= \sum_{i\in J_0^{\leftarrow}}\Big(\sum_{j\in J_1^{\leftarrow}} A_{ij}\Big)^2+\sum_{j\in J_1^{\leftarrow}}\Big(\sum_{i\in J_0^{\leftarrow}} A_{ij}\Big)^2-\sum_{(i,j)\in I^{\leftarrow}} A_{ij} ^2,\]
\eqref{trop1} and \eqref{trop2}
imply
 \begin{equation}\label{aller1''}
 \liminf_{\gamma_\ell \to 0} \varphi''_{\gamma_\ell }(t)\geq \int \left[A^2 -{\mathbb A} + {\mathbb A}\left(\log {\mathbb A}-\log\left(A^2-\widetilde A^2\right)\right) \right] d\widehat Q_t.
 \end{equation}

 Identically,  the lower-bound of $ \liminf_{\gamma_\ell \to 0} \psi''_{\gamma_\ell }(t)$
  given by Theorem \ref{limdevseconde} provides
  \begin{equation}\label{aller2''}
 \liminf_{\gamma_\ell \to 0} \psi''_{\gamma_\ell }(t)\geq \int \left[B^2 -{\mathbb B} + {\mathbb B}\left(\log {\mathbb B}-\log\left(B^2-\widetilde B^2\right)\right) \right] d\widehat Q_t,
 \end{equation}
 where we set for any $z\in \widehat Z$, $B(z):=\sum_{(i,j)\in I^{\rightarrow}} B_t(z,\sigma_{ij}(z))$, ${\mathbb B}(z):=\sum_{\{(i,j),(k,l)\}\subset {\mathbb I}_1^{\rightarrow}, i\neq k, j\neq l} 2 {\mathbb B}_t(z, \sigma_{kl}\sigma_{ij}(z))$ and 
 \[ {\widetilde B}^2(z):= \sum_{i\in J_0^{\rightarrow}}\Big(\sum_{j\in J_1^{\rightarrow}} B_t(z,\sigma_{ij}(z))\Big)^2+\sum_{j\in J_1^{\rightarrow}}\Big(\sum_{i\in J_0^{\rightarrow}} B_t(z,\sigma_{ij}(z))\Big)^2-\sum_{(i,j)\in I^{\rightarrow}} B_t(z,\sigma_{ij}(z)) ^2.\]
 
 {\bf Step 2 :} By the Cauchy-Schwarz inequality, one has 
 \[ {\widetilde A}^2 \geq \max\Big[\sum_{i\in J_0^{\leftarrow}}\Big(\sum_{j\in J_1^{\leftarrow}} A_{ij}\Big)^2, \sum_{j\in J_1^{\leftarrow}}\Big(\sum_{i\in J_0^{\leftarrow}} A_{ij}\Big)^2\Big] \geq  \max\Big[\frac1{|J_0^{\leftarrow}|}, \frac1{|J_1^{\leftarrow}|}\Big] \,A^2,\]
 and therefore, \eqref{aller1''} together with the concavity property of the logarithmic function yield
\begin{equation}\label{aller1'''}
 \liminf_{\gamma_\ell \to 0} \varphi''_{\gamma_\ell }(t)\geq  \int \max\Big[\frac1{|J_0^{\leftarrow}|}, \frac1{|J_1^{\leftarrow}|}\Big] \max \left\{ A^2,  {\mathbb A} \right\} d\widehat Q_t.
 \end{equation}
Identically \eqref{aller2''} gives 
 \begin{equation}\label{aller2'''}
 \liminf_{\gamma_\ell \to 0} \psi''_{\gamma_\ell }(t)\geq  \int  \max\Big[\frac1{|J_0^{\rightarrow}|}, \frac1{|J_1^{\rightarrow}|}\Big] \max \left\{ B^2,  {\mathbb B} \right\}d\widehat Q_t.  
 \end{equation}

Keeping  the quantities involving $A$ and $B$ in \eqref{aller1'''} and \eqref{aller2'''}, and applying Cauchy-Schwarz inequality, the identities \eqref{sumA} and \eqref{sumB} yield 
 \[\liminf_{\gamma_\ell \to 0} \varphi''_{\gamma_\ell }(t)+ \liminf_{\gamma_\ell \to 0} \psi''_{\gamma_\ell }(t)\geq \, W_1^2(\nu_0,\nu_1) \left[\frac{1}{ \int \min\Big[|J_0^{\leftarrow}|, |J_1^{\leftarrow}|\Big]  \, d\widehat Q_t} +\frac{1}{ \int \min\Big[|J_0^{\rightarrow}|, |J_1^{\rightarrow}|\Big] \, d\widehat Q_t} \right].\]
 Since $J_0^{\leftarrow}\cap J_0^{\rightarrow}=\emptyset$ and $J_1^{\leftarrow}\cap J_1^{\rightarrow}=\emptyset$, one has 
 \begin{equation}\label{2004}
 \min\Big[|J_0^{\leftarrow}|, |J_1^{\leftarrow}|\Big] +\min\Big[|J_0^{\rightarrow}|, |J_1^{\rightarrow}|\Big] \leq \min\Big[ |J_0^{\leftarrow}|+ |J_0^{\rightarrow}|,  |J_1^{\leftarrow}|+|J_1^{\rightarrow}|\Big] \leq \min[n-\kappa, \kappa].
 \end{equation}
and therefore  the identity  $\min_{\alpha,\beta>0, \alpha+\beta\leq 1}\left\{\frac{1}\alpha+\frac{1}\beta\right\}=4$ implies
 \[\liminf_{\gamma_\ell \to 0} \varphi''_{\gamma_\ell }(t)+ \liminf_{\gamma_\ell \to 0} \psi''_{\gamma_\ell }(t)\geq \frac{4}{\min[n-\kappa, \kappa]}\,  W_1^2(\nu_0,\nu_1).\]
The first lower bound of $C_t(\widehat \pi)$ in Theorem \ref{thmslicecube} then follows applying Lemma \ref{lemconvex}.

Keeping only the quantities involving ${\mathbb A} $ and ${\mathbb B} $ in \eqref{aller1'''} and \eqref{aller2'''} gives 
\begin{equation}\label{oouf'}
 \liminf_{\gamma_\ell \to 0} \varphi''_{\gamma_\ell }(t)+ \liminf_{\gamma_\ell \to 0} \psi''_{\gamma_\ell }(t)\geq   \int  \frac{\mathbb A}{\min\Big[|J_0^{\leftarrow}|, |J_1^{\leftarrow}|\Big]  } +\frac{{\mathbb B} }{ \min\Big[|J_0^{\rightarrow}|, |J_1^{\rightarrow}|\Big]}\, d\widehat Q_t .
 \end{equation}
 According to   Lemma \ref{lemintaibis} one has for any $z\in\widehat Z$,  $y\in\widehat Y_z$, and  for any $((i,j),(k,l))\in {\mathbb I}^\leftarrow(z)$, 
 \[ 
 2\displaystyle{\mathbb A}_{kl,ij}(z) :=\frac{{\mathbbm a}_t(z,\sigma_{kl}\sigma_{ij}(z),y) L(z,\sigma_{kl}\sigma_{ij}(z)) }{a_t(z,y)}.
 \]
Therefore the expression \eqref{a_t''} of  ${\mathbbm a}_t(z,\sigma_{kl}\sigma_{ij}(z),y)$ with  the identity 
\[\sum_{\{(i,j),(k,l)\}\}\subset{\mathbb I}_1^{\leftarrow}, i\neq k,j\neq l }  r(y,z,\sigma_{kl}\sigma_{ij}z),w) L^2(z,\sigma_{kl}\sigma_{ij}(z)) = r(y,z,z,w),\]
give
\begin{align*}
{\mathbb A}(z) &=\sum_{\{(i,j),(k,l)\}\subset {\mathbb I}_1^{\leftarrow}, i\neq k, j\neq l} 2 {\mathbb A}_{kl,ij}(z)\\
&=\frac{1}{a_t(z,y)} \sum_{w\in \X} \sum_{z,z\in [y,w]} r(y,z,z,w) d(y,w)(d(y,w)-1)\, \B_t^{d(y,w)-2}(d(z,w)-2)\,\widehat{\pi}_{_\leftarrow}(w|y).
\end{align*}
Working identically with ${\mathbb B}(z)$ we finally get  
\begin{equation*}
  \int  \frac{\mathbb A}{\min\Big[|J_0^{\leftarrow}|, |J_1^{\leftarrow}|\Big]  } +\frac{{\mathbb B} }{ \min\Big[|J_0^{\rightarrow}|, |J_1^{\rightarrow}|\Big]}\, d\widehat Q_t \geq \iint c_t(x,y) \,d\widehat{\pi}(x,y),
 \end{equation*}
where 
\begin{multline*}
c_t(x,y)=\sum_{z\in[x,y]}  \frac1{\min\Big[|J_0^{\leftarrow}(z)|, |J_1^{\leftarrow}(z)|\Big] }\,r(x,z,z,y) d(x,y)(d(x,y)-1)\, \B_t^{d(x,y)-2}(d(x,z)-2) \\
+\sum_{z\in[x,y]}  \frac1{\min\Big[|J_0^{\rightarrow}(z)|, |J_1^{\rightarrow}(z)|\Big]} r(x,z,z,y) d(x,y)(d(x,y)-1)\, \B_t^{d(x,y)-2}(d(x,z)).
\end{multline*}
Using the inequality \eqref{2004}, the end of the proof of the second lower bound of $C_t(\widehat \pi)$ involving $T_{c_2}(\widehat \pi)$ is exactly the same as in the proof Theorem \ref{thmcube}. It is left to the reader.

We now turn to the  proof of third lower bound on $C_t(\widehat \pi)$. Using again \eqref{aller1''} and \eqref{aller2''} and the concavity of the logarithmic function, one gets 
\begin{equation}\label{oouf1'}
\liminf_{\gamma_\ell \to 0} \varphi''_{\gamma_\ell }(t)+ \liminf_{\gamma_\ell \to 0} \psi''_{\gamma_\ell }(t) \geq \int \widetilde A^2 \, d\widehat Q_t  +  \int \widetilde B^2 \, d\widehat Q_t.
\end{equation} 
According to the definition of $\widetilde A^2$, one has
\begin{align}\label{IHP'}\int \widetilde A^2 \, d\widehat Q_t &\geq \int  \max\Big[  \sum_{i\in J_0^\leftarrow}\Big(\sum_{j\in J_1^\leftarrow} A_{ij}\Big)^2, \sum_{j\in J_1^\leftarrow}\Big(\sum_{i\in J_0^\leftarrow} A_{ij}\Big)^2\Big]\,d\widehat Q_t
\nonumber\\
&= \int \sum_{z\in \widehat Z^y} \max\Big[ \sum_{i\in [n]}\Big(
\sum_{j\in [n]}  A_{ij}(z)\1_{(i,j)\in I^\leftarrow(z)}\Big)^2, \sum_{j\in [n]}\Big(
\sum_{i\in [n]}  A_{ij}(z)\1_{(i,j)\in I^\leftarrow(z)}\Big )^2\Big]  a_t(z,y) \,d\nu_1(y)\nonumber \\
&\geq \int \max\Big[\sum_{z\in \widehat Z^y}  \sum_{i\in [n]}\Big(
\sum_{j\in [n]}  A_{ij}(z)\1_{(i,j)\in I^\leftarrow(z)}\Big)^2 a_t(z,y), \sum_{z\in \widehat Z^y} \sum_{j\in [n]}\Big(
\sum_{i\in [n]}  A_{ij}(z)\1_{(i,j)\in I^\leftarrow(z)}\Big )^2a_t(z,y)\Big]  \,d\nu_1(y)
\end{align}
For any $y\in \supp(\nu_1)$, and any $i\in J_0(y)$, $j\in J_1(y)$ we note 
\[E_{i,0}^{\leftarrow}(y):=\Big \{z\in \X_\kappa\Big |\, \exists l \in J_1(y),  y \in \widehat Y_{(z,\sigma_{il}(z))} \Big\},
\qquad E_{j,1}^{\leftarrow}(y):=\Big \{z\in \X_\kappa\Big |\, \exists k\in J_0(y),  \in \widehat Y_{(z,\sigma_{kj}(z))} \Big\}.
\]
Since $(i,j)\in I^{\leftarrow}(z)$ and $z\in\widehat Z^y$  imply $z\in E_{i,0}^{\leftarrow}(y)$ and $z\in E_{j,1}^{\leftarrow}(y)$, one has 
\[
\sum_{z\in \widehat Z^y} \sum_{i\in [n]}\Big(
\sum_{j\in [n]}  A_{ij}(z)\1_{(i,j)\in I^\leftarrow(z)}\Big)^2\, a_t(z,y)
 = \sum_{i\in J_0(y)}\sum_{z\in E_{i,0}^{\leftarrow}(y)}\Big(
\sum_{j\in  J_1(y)}  A_{ij}(z)\Big)^2 \,a_t(z,y) ,
\]
and therefore by Cauchy-Schwarz inequality,
\begin{align}\label{martinkine}
 \sum_{z\in \widehat Z^y} \sum_{i\in [n]}\Big(
\sum_{j\in [n]}  A_{ij}(z)\1_{(i,j)\in I^\leftarrow(z)}\Big)^2\, a_t(z,y) 
 \geq   \sum_{i\in J_0(y)} \frac {\left(
\sum_{j\in  J_1(y)} \sum_{z\in E_{i,0}^{\leftarrow}(y)} A_{ij}(z) a_t(z,y)\right)^2} {\sum_{z\in E_{i,0}^{\leftarrow}(y)} a_t(z,y)} .
\end{align}
For $(i,j)\in J_0(y)\times J_1(y)$, one may compute the quantity 
$\sum_{z\in E_{i,0}^{\leftarrow}(y)} A_{ij}(z) a_t(z,y)$
 using the two following observations. First  $(z,\sigma_{ij}(z))\in [y,w]$ holds  if and only if one has $y_i=z_i=w_j=0$, $y_j=z_j=w_i=1$ and $z\in [y, \sigma_{ij}(w)]$. Secondly, the generator $L$ is translation invariant which implies for any $(z,\sigma_{ij}(z))\in[y,w]$, \[r(y,z,\sigma_{ij}(z),w)=r(y,z,z,\sigma_{ij}(w))\,\frac{  L^{d(y,\sigma_{ij}(w))}(y,\sigma_{ij}(w))}{L^{d(y,w)}(y,w)}.\] Therefore, using \eqref{a_t'}, one gets for any $(i,j)\in J_0(y)\times J_1(y)$,
\begin{align*}
&\sum_{z\in E_{i,0}^{\leftarrow}(y)} A_{ij}(z)\,a_t(z,y)=\sum_{z\in \X_\kappa} {\mathrm a}_t(z,\sigma_{ij}(z),y)\\
&= \sum_{w\in \X_\kappa} \1_{y_i=w_j=0}  \1_{y_j=w_i=1}\!\!\!\!\!\!\! \sum_{s=0}^{d(y,\sigma_{ij}(w))} \sum_{z\in [y,\sigma_{ij}(w)],d(y,z)=s}  \!\!\!\!\!\!\! 
r(y,z,z,\sigma_{ij}(w))\frac{  L^{d(y,\sigma_{ij}(w))}(y,\sigma_{ij}(w))}{L^{d(y,w)}(y,w)}\,\\
&\qquad\qquad\qquad\qquad\qquad\qquad\quad\qquad\qquad  \qquad d(y,w) \B_t^{d(y,w)-1}(d(y,w)-1-s)
\,\widehat{\pi}_{_\leftarrow}(w|y)
\\&=\sum_{w\in \X_\kappa} \1_{y_i=w_j=0}  \1_{y_j=w_i=1}\frac {L^{d(y,\sigma_{ij}(w))}(y,\sigma_{ij}(w))}{L^{d(y,w)}(y,w)} \,d(y,w)\,\widehat{\pi}_{_\leftarrow}(w|y)
\\&= \sum_{w\in \X_\kappa} \frac{\1_{y_i=w_j=0}  \1_{y_j=w_i=1}}{d(y,w)}\,\widehat{\pi}_{_\leftarrow}(w|y),
\end{align*}
where the last equality holds since $L^{d(x,y)}(x,y)=(d(x,y)!)^2$ for any $x,y\in \X_\kappa$.
Since for $i\in J_0(y)$, $\displaystyle \sum_{j\in J_1(y)} \1_{y_i=w_j=0}  \1_{y_j=w_i=1}= d(y,w)\, \1_{w_i\neq y_i}$, it follows that
\begin{equation}\label{IHP}
\sum_{j\in J_1(y)}\sum_{z\in E_{i,0}^{\leftarrow}(y)} A_{ij}(z)\,a_t(z,y)
=  \sum_{w\in \X_\kappa}  \1_{w_i\neq y_i} \,\widehat{\pi}_{_\leftarrow}(w|y).
\end{equation}
Similar computations also provide, for any $i\in J_0(y)$,  
 \begin{align*}
& \sum_{z\in E_{i,0}^{\leftarrow}(y)} a_t(z,y)=\sum_{w\in \X_\kappa} \sum _{z\in[y,w]}  \1_{z\in E_{i,0}^{\leftarrow}(y)} {Q_t}^{w,y}(z) \,\widehat{\pi}_{_\leftarrow}(w|y)\\
& \leq \sum_{w\in \X_\kappa} \sum _{z\in[y,w]}  \1_{z_i=y_i=0} {Q_t}^{w,y}(z) \,\widehat{\pi}_{_\leftarrow}(w|y)\\
 &= \sum_{w\in \X_\kappa}  \1_{y_i=w_i=0} \,\widehat{\pi}_{_\leftarrow}(w|y) + \sum_{w\in \X_\kappa}  \1_{y_i\neq w_i} \Big(\sum _{z\in[y,w]}  \1_{z_i=y_i=0} {Q_t}^{w,y}(z)\Big)\,\widehat{\pi}_{_\leftarrow}(w|y).
 \end{align*}
 Moreover from the expression of ${Q_t}^{w,y}(z)$ given by \eqref{nutsectioncube}, one has  for  $y_i=0$ and $w_i=1$,  
 \begin{align*}
\sum _{z\in[y,w]}  \1_{z_i=y_i=0} {Q_t}^{w,y}(z)&=\sum_{k=0}^{d(y,w)-1} \Big(\sum_{z, z\in [y,w], z_i=0} \1_{d(y,z)=k}\Big) \,\frac{(1-t)^k t^{d(y,w)-k}}{\binom{d(y,w)}{k}}\\
&=\sum_{k=0}^{d(y,w)-1}  \binom{d(y,w)}{k}\binom{d(y,w)-1}{k} \,  \frac{(1-t)^k t^{d(y,w)-k}}{\binom{d(y,w)}{k}}\\
&=t.
\end{align*}
It follows that  for any $i\in J_0(y)$
\[ \sum_{z\in E_{i,0}^{\leftarrow}(y)} a_t(z,y)\leq 1-(1-t)\int \1_{y_i\neq w_i} d\widehat{\pi}_{_\leftarrow}(w|y).\]
As a consequence, since $\Pi_{\leftarrow}^i(y) :=\int  \1_{y_i\neq w_i} d\widehat{\pi}_{_\leftarrow}(w|y)$, \eqref{martinkine} and \eqref{IHP} implies
\begin{equation*}
 \sum_{z\in \widehat Z^y} \sum_{i\in [n]}\Big(
\sum_{j\in [n]}  A_{ij}(z)\1_{(i,j)\in I^\leftarrow(z)}\Big)^2\, a_t(z,y) 
\geq   \sum_{i\in J_0(y)} \frac {\Pi_{\leftarrow}^i(y) ^2} {1-(1-t) \Pi_{\leftarrow}^i(y)}.
\end{equation*}
By symmetry, the same inequality holds exchanging the role of $i $ and $j$, and therefore \eqref{IHP'} gives 
\[\int \widetilde A^2 \, d\widehat Q_t \geq \int \max\Big[   \sum_{i\in J_0(y)} \frac {\Pi_{\leftarrow}^i(y) ^2} {1-(1-t) \Pi_{\leftarrow}^i(y)},  \sum_{j\in J_1(y)} \frac {\Pi_{\leftarrow}^j(y) ^2} {1-(1-t) \Pi_{\leftarrow}^j(y)}\Big] \,d\nu_1(y).\]
As  in the proof of Theorem \ref{thmcomplete}, this inequality  implies for any $t\in(0,1)$
 \begin{equation*}
\int \widetilde A^2 \, d\widehat Q_t \geq \xi_\leftarrow''(t),
\end{equation*}
with
\begin{equation*}
\xi_\leftarrow(t) :=\frac12 \int  \max\Big[\sum_{i\in J_0(y)} h\Big((1-t)\Pi^i_\leftarrow(y) \Big), \sum_{j\in J_1(y)}   h\Big((1-t)\Pi^j_\leftarrow(y) \Big)\Big] \,d\nu_1(y).
 \end{equation*}
Identically, one proves that 
 \begin{equation*}
\int \widetilde B^2 \, d\widehat Q_t \geq \xi_\rightarrow''(t),
\end{equation*}
where 
\[\xi_\rightarrow(t):=\frac12\int \max\Big[\sum_{i\in J_0(x)}  \,h\Big(t  \,\Pi^i_\rightarrow(x)\Big),\sum_{j\in J_1(x)} \,h\Big(t  \,\Pi^j_\rightarrow(x)\Big) \Big]\,d\nu_0(x).\]
From \eqref{oouf1'} and  the two last estimates, applying Lemma \ref{lemconvex} provides the third lower bound of $C_t(\widehat \pi)$ in Theorem \ref{thmslicecube}. 
\end{proof}


\section{Appendix A : Basic lemmas}\label{appA}
\begin{lemma}\label{cubegauss}
The transport-entropy inequality \eqref{talcub} implies  the $W_2$ transport-entropy inequality \eqref{tal} for the standard Gaussian measure $\gamma$.
\end{lemma}

\begin{proof} The result follows from the  transport-entropy inequality \eqref{talcub} for the uniform probability measure $\mu$ on the hypercube  ($\alpha_i=1/2$ for all $i\in[n]$), and by using the  central limit Theorem with the projection map  
\[T_n(x):=\frac{2}{\sqrt n} \Big(\sum_{i=1}^n x_i-\frac n2\Big), \quad x,y\in \{0,1\}^n.\]

By density, it is sufficient to prove \eqref{tal}  for any probability measure $\nu$ on $\R$ with  continuous density $f$ and compact support $K$. Let $\nu^n$ denotes the probability measure on $\{0,1\}^n$  with density $f_n$  with respect to $\mu$ given by 
\[ f_n(x):= \frac {f(T_n(x))} {\int f\circ T_n \,d\mu},\qquad x\in\{0,1\}^n.\]
Applying \eqref{talcub} with $\nu_0:=\mu$ and $\nu_1:= \nu^n$, one gets 
\begin{eqnarray}\label{talapp}
\frac2n {T_{c_2} }(\mu,\nu^n)\leq  H(\nu^n|\mu).
\end{eqnarray}
By the weak convergence of $T_n \#\mu$ to the standard Gaussian law $\gamma$, one has 
\begin{eqnarray}\label{liment}
\lim_{n\to \infty} H(\nu^n |\mu)= H(\nu|\gamma),
\end{eqnarray}
and for $k=1$ or $k=2$,
\begin{eqnarray}\label{limmom}
\lim_{n\to \infty} \int |w|^k \, d(T_n\#\nu^n)(w)=\lim_{n\to \infty} \int |T_n(x)|^k \, f_n(T_n(x)) \,d\mu(x)= \int |w|^k \,d\nu(w).
\end{eqnarray} 
Since $ d(x,y)\geq \frac{\sqrt n}2 \left|T_n(x)-T_n(y)\right|$ and the monotonicity property of the function  $c_2:\R\to \R^+$ on $[2,+\infty)$ implies
\[\frac2n\, c_2(d(x,y))\geq \frac2n c_2\left(\frac{\sqrt n}2 \left|T_n(x)-T_n(y)\right| \right)\1_{\frac{\sqrt n}2 \left|T_n(x)-T_n(y)\right|\geq 2},\]
and therefore 
\[\frac2n\,{ T_{c_2}}(\mu,\nu^n)^2\geq \frac12 \inf_{\pi_n \in\Pi(T_n\#\mu,T_n\#\nu^n)}\iint 
c_n(z,w) \,d\pi_n(z,w),\]
where for any $z,w\in \R$
\begin{align*}
c_n(z,w)&:=  4n c_2\left(\frac{\sqrt n\,|z-w|}2\right)\1_{ |z-w|\geq 4/\!\sqrt n}\\ &=\left[|z-w|^2-\frac{4(1+\log(\sqrt n /2))}{\sqrt n}\,|z-w| - \frac{4}{\sqrt n} \, |z-w|\log|z-w|\right]\1_{|z-w|\geq 4/\!\sqrt n} .
\end{align*}
Let $c(z,w):=|z-w|^2$, $z,w\in \R$. One has, for any $z,w\in \R$, $c(z,w)\geq c_n(z,w)$ and  
\begin{align*}
c(z,w)-c_n(z,w)&= |z-w|^2\1_{|z-w|< 4/\!\sqrt n} +\left[ \frac{4(1+\log(\sqrt n /2))}{\sqrt n}\,|z-w| + \frac{4}{\sqrt n} \, |z-w|\log|z-w|\right] \,1_{|z-w|\geq 4/\!\sqrt n}\\
&
\leq \frac{16}n + \frac{4(1+\log n)}{\sqrt{n}} \left[|z|+|w| +1+2|z|^2+2|w|^2\right],
\end{align*}
where the last inequality follows from  $|u\log u|\leq 1+u^2, u>0$.
Since 
\[\int |z|\,d(T_n \#\mu)(z)\leq \left(\int |z|^2 d(T_n \#\mu)(z)\right)^{1/2}= \left(\int T_n^2 d\mu\right)^{1/2}=1,\]
it follows that for any $\pi_n \in\Pi(T_n\#\mu,T_n\#\nu^n)$,
\begin{align*}
\iint c_n \,d\pi_n 
&\geq \iint c\, d\pi_n-\frac{16}n - \frac{4(1+\log n)}{\sqrt{n}}  \iint \left[|z|+|w| +1+2|z|^2+2|w|^2\right]  \,d\pi_n(z,w)\\
&\geq  \iint c\, d\pi_n - \frac{32(1+\log n)}{\sqrt{n}} \left[ 1+ \int |w| \, d(T_n\#\nu^n)(w) +\int |w|^2 \, d(T_n\#\nu^n)(w)\right],
\end{align*}
and therefore 
\begin{equation*}
\frac2n\,{T_{c_2} }(\mu,\nu^n)^2\geq \frac12\,W^2_2(T_n\#\mu,T_n\#\nu^n)  - \frac{16(1+\log n)}{\sqrt{n}}\left[1+  \int |w| \, d(T_n\#\nu^n)(w) +\int |w|^2 \, d(T_n\#\nu^n)(w)\right] .
\end{equation*}
From the weak convergence in $\Pc_2(\R)$ of the sequences $(T_n\#\mu)$ and $(T_n\#\nu^n)$ and using   \eqref{limmom}, the last inequality implies as $n$ goes to infinity
\[\liminf_{n\to +\infty} \frac2n\,{T_{c_2} }(\mu,\nu^n)\geq \frac12 \,W_2^2(\nu,\gamma).\]
Finally, Talagrand's inequality $W_2^2(\nu,\gamma)\leq 2 H(\nu|\gamma)$, follows from \eqref{talapp} and  \eqref{liment}.
\end{proof}

\begin{lemma}\label{logsobcube}  If the convexity property  \eqref{deplacebis} holds, then for any $\nu_0,\nu_1\in \Pb_b(\X)$, 
\begin{equation*}
H(\nu_0|\mu)\leq H(\nu_1|\mu) +\sum_{x\in \X}\sum_{x'\in\X, x'\sim x}   \left(\log(f(x)-\log f(x')\right)\, \Pi^{x'}_\rightarrow(x)  \,\nu_0(x)  
-\frac12\liminf_{t\to 0} C_t(\widehat \pi).
\end{equation*}
where $\Pi^{x'}_\rightarrow(x):=\int \1_{x'\in [x,y]} d(x,y) r(x,x',x',y) \, d\widehat \pi_\rightarrow(y|x)$.  
\end{lemma}
\begin{proof}  
The convexity property  \eqref{deplacebis} implies, for any  $\nu_0,\nu_1\in\Pb_b(\X)$ and for any $t\in(0,1)$
\begin{equation}\label{tata} 
H(\nu_0|\mu)\leq H(\nu_1|\mu) -\frac{H(\widehat Q_t|\mu)-H(\nu_0|\mu)}{t} - \frac{(1-t)}2 C_t(\widehat \pi).
\end{equation}
The first step is to compute the left-hand side of this inequality as $t$ goes to zero. 
According to the expression \eqref{pontxycube} of ${Q_t}\!^{ x,y}$, for any $x,y,z\in\{0,1\}^n$,
\[\partial_t {Q_t}\!^{ x,y}(z)=r(x,z,z,y)\,\binom{d(x,y)}{d(x,z)} \,\1_{[x,y]}(z)  \left(d(x,z)t^{d(x,z)-1}(1-t)^{d(z,y)} -d(z,y)   t^{d(x,z)}(1-t)^{d(z,y)-1}\right),\] 
and therefore 
\begin{multline*}\partial_t Q_t^{x,y}(z)_{|t=0}=r(x,z,z,y)\,\binom{d(x,y)}{d(x,z)} \left( \1_{[x,y]}(z)\1_{z\sim x}-d(x,y)\1_{x=z}\right)\\
=\sum_{x'\in[x,y], x'\sim x} d(x,y) 
r(x,x',x',y)\left(\delta_{x'}(z)-\delta_x(z)\right).
\end{multline*}

Since $\partial_t \widehat Q^\gamma_t(z)_{|t=0} =\sum_{x,y\in \X} \partial_t Q_t^{x,y}(z)_{|t=0} \, \widehat \pi(x,y)$, it follows that
\begin{align*}
\lim_{t\to 0} \frac{ H(\widehat Q_t|\mu)-H(\nu_0|\mu)}{t}&=\partial_t H(\widehat Q^\gamma_t|\mu)_{|t=0}=\sum_{z\in \X} \partial_t \widehat Q_t(z)_{|t=0} \log f(z)\,\mu(z)\\
&=\sum_{x,y\in \X} \sum_{x'\in[x,y], x'\sim x} d(x,y)  \left(\log f(x')-\log f(x)\right)\, d(x,y)\,r(x,x',x',y)\, \widehat \pi(x,y)\\
&=\sum_{x\in \X}\sum_{x'\in\X, x'\sim x}   \left(\log(f(x')-\log f(x)\right)\, \left(\sum_{y\in \X, x'\in [x,y]} d(x,y) r(x,x',x',y) \, \widehat \pi_{_\rightarrow}(y|x) \right) \,\nu_0(x)
\end{align*}
The proof of Lemma \ref{logsobcube} ends from \eqref{tata}  as $t$ goes to 0.
\end{proof}

\begin{lemma}\label{opt}  Let $\X$ be a graph with graph distance $d$. 
Let  $\nu_0,\nu_1\in \Pc(\X)$ and assume that  $\widehat \pi\in\Pc(\X\times \X)$ is a $W_1$-optimal coupling of $\nu_0$ and $\nu_1$, namely
\[W_1(\nu_0,\nu_1)=\iint d(x,y) \,d\widehat \pi(x,y).\]
\begin{enumerate}[label=(\roman*)]
\item \label{mon0} Let 
\[C_{_\rightarrow}:=\Big\{(z,w)\in\X\times \X\,\Big|\,z\neq w, \exists (x,y)\in \supp(\widehat{\pi}), (z,w)\in [x,y]\Big\}.\]
If $(z_1,w)\in C_{_\rightarrow}$ and $(w,z_2)\in C_{_\rightarrow}$ then $d(z_1,z_2)\geq 2$ and $w\in[z_1,z_2]$.

\item \label{mon1} Let 
\[ C_{_\leftarrow}:=\Big\{(z,w)\in\X\times \X\,\Big|\, (w,z) \in C_{_\rightarrow} \Big\}.\]
The  sets $C_{_\rightarrow}$ and $C_{_\leftarrow}$ are disjoint.
 
\item \label{mon2}
If $d$ is the Hamming distance then the  following sets $D_{_\rightarrow}$ and $D_{_\leftarrow}$ are disjoint, 
\[D_{_\leftarrow}:=\Big\{ w\in \supp(\nu_1)\,\Big|\, \exists x\in \X, w\neq x, (x,w)\in \supp(\widehat{\pi}) \Big\},\] 
and 
\[D_{_\rightarrow}:=\Big\{ w\in \supp(\nu_0)\,\Big|\, \exists y\in \X, w\neq y, (w,y)\in \supp(\widehat{\pi}) \Big\}.\]  
\end{enumerate}
\end{lemma}

\begin{proof} 
\begin{enumerate}[label=(\roman*)]
\item Let $(z_1,w)\in C_{_\rightarrow}$ and $(w,z_2)\in C_{_\rightarrow}$. 
There exists $(x,y)\in \supp(\widehat{\pi})$ such that  $(z_1,w)\in [x,y]$ and there exists $(x',y')\in \supp(\widehat{\pi})$ such that  $(w,z_2)\in [x',y']$. 
One has 
\[d(z_1,w)+d(w,z_2)=\big((d(x,y)-d(x,z_1)-d(w,y)\big)+\big(d(x',y')-d(x',w)-d(z_2,y')\big).\]
It is well known that the support of any optimizer of $W_1(\nu_0,\nu_1)$ is $d$-cyclically monotone (see \cite[Theorem 5.10]{Vil09}. By definition, it means that for any family $(x_1,y_1),\ldots ,(x_N,y_N)$ of points in the support of $\widehat  \pi$
\[\sum_{i=1}^N d(x_i,y_i) \leq \sum_{i=1}^N d(x_i,y_{i+1}),\]
with the convention $y_{N+1}=y_1$.
It follows that 
\[ d(x,y)+d(x',y')\leq d(x,y')+d(x',y),\]
and therefore, from the above identity,
\[d(z_1,w)+d(w,z_2)\leq d(x,y')+d(x',y) -d(x,z_1)-d(w,y)-d(x',w)-d(z_2,y').\]
By the triangular inequality, it follows that 
\begin{multline*}
2\leq d(z_1,w)+d(w,z_2)\leq \big(d(x,z_1)+d(z_1,z_2)+d(z_2,y')\big)\\
+\big(d(x',w)+d(w,y)\big) -d(x,z_1)-d(w,y)-d(x',w)-d(z_2,y')=d(z_1,z_2) .
\end{multline*}
This implies that $d(z_1,z_2)\geq 2$ and $w\in[z_1,z_2]$.

\item
Assume there exists $(z,w)\in C_{_\rightarrow}\cap C_{_\leftarrow}$. Then $(w,z)\in C_{_\rightarrow}$ and therefore, according to \ref{mon0}, $z\in[w,w]=\{w\}$. This is impossible since $z\neq w$.

\item We assume that $d(x,y)=\1_{x\neq y}$ for any $x,y\in \X$. If the two sets $D_{_\rightarrow}$ and   $D_{_\leftarrow}$ intersect, then there exists 
$(x,w)\in  C_{_\rightarrow}$ and $(w,y)\in  C_{_\rightarrow}$. Point \ref{mon0} implies $w\in[x,y]$, and since $d(x,y)=1$, we get either $w=x$ or $w=y$, which is impossible.

\end{enumerate}
\end{proof}

\begin{lemma}\label{lemmetech} Let $\nu_0$ and $\nu_1$ some probability measures in ${\mathcal P}(\X)$ with bounded support. 
\begin{enumerate}[label=(\roman*)]
\item\label{item1} If \eqref{unifbounded1} holds ($\exists S\geq 1,\sup_{x\in\X}|L(x,x)|\leq S$), then for any $x,y\in \X$ and any integer $k$, \[L^k(x,y)\leq (2S)^k.\]
\item\label{item2} If \eqref{unifbounded2} holds ($\exists I\in(0,1],\inf_{x,y\in\X,x\sim y} L(x,y)\geq I$), then for any $x,y\in \X$, $L^{d(x,y)}(x,y)\geq I^{d(x,y)}.$
\item\label{item3} If \eqref{unifbounded1} and \eqref{unifbounded2} hold, then for any $x,y\in \X$, any $t\in [0,1]$, and any $\gamma\in(0,1)$, one has  
\[\P_t^{\gamma}(x,y) = \frac{ L^{d(x,y)}(x,y)}{d(x,y)!} \, (\gamma t)^{d(x,y)}\left( 1+ \gamma K^{d(x,y)}O(1)\right),\]
where $K:=2S/I$ and $O(1)$ denotes a quantity uniformly  bounded in $x,y,t$ and $\gamma$. 
\item \label{item4} If \eqref{unifbounded1} holds  then for any $x,y,z\in \X$ and for any $t\in[0,1]$
\[\lim_{\gamma\to 0}{Q_t^\gamma}^{x,y}(z)={Q_t}\!^{ x,y}(z):=\1_{[x,y]}(z)\,r(x,z,z,y) \,\B_t^{d(x,y)}(d(x,z)).\] 

\item \label{item5} If \eqref{unifbounded1} holds  then for any $x,y\in \X$,
\[\P_t^{\gamma}(x,y) \geq \frac{ L^{d(x,y)}(x,y) }{d(x,y)!} \,(t\gamma)^{d(x,y)}e^{-\gamma t S}.\]
For a fixed $x_0\in \X$, let $\displaystyle D:=\max_{x\in {\rm supp}(\nu_0),y\in {\rm supp}(\nu_1)}(d(x_0,x), d(x_0,y))$. It follows that if \eqref{unifbounded1} and \eqref{unifbounded2} hold then   for any $\gamma\in(0,1)$ and $t\in(0,1)$,
 \begin{equation}\label{encadrementbis}
0<e^{-S}      \left(\frac{t\gamma I}{d(x_0,z)+1+D}\right)^{d(x_0,z)+1+D} \min_{w\in\supp(\nu_0)} f^\gamma(w)  \leq \P_t^\gamma f^\gamma(z)\leq \max_{w\in\supp(\nu_0)} f^\gamma(w).
 \end{equation}

\item \label{item6} If $\eqref{unifbounded1}$ holds then ${\mathbb E}_{R^\gamma}[\ell|X_0=x,X_1=y]\leq \frac{\gamma S}{\P^\gamma_1(x,y)}$.

\item  \label{item7} Assume $\eqref{unifbounded1}$ and $\eqref{unifbounded2}$ hold. For a fixed $x_0\in \X$, let $\displaystyle D:=\max_{x\in {\rm supp}(\nu_0),y\in {\rm supp}(\nu_1)}(d(x_0,x), d(x_0,y))$. For any $x\in {\rm supp}(\nu_0) $    and $y\in {\rm supp}(\nu_1) $, one has for any $t \in(0,1)$ and any $\gamma\in(0,1)$
\[
{Q_t^\gamma}^{x,y}(z)\leq O(1) \left( \1_{[x,y]}(z)+ \left(1-\1_{[x,y]}(z)\right)\gamma \left(\gamma K^2\right)^{[2d(x_0,z)-4D-1]_+}\right),
\]
 where $K:=2S/I$ and $O(1)$ denotes a  constant that only depends on $S,I,D$ and $K:=2S/I$. 
 
 As a consequence, setting 
 \[B:=\bigcup_{x\in\supp(\nu_0),y\in\supp(\nu_1)} [x,y],\]
 one has 
 \begin{eqnarray}\label{majQt}
  \widehat Q_t^\gamma (z)\leq O(1) \,\gamma \left(\gamma K^2\right)^{[2d(x_0,z)-4D-1]_+},\qquad \forall z\in \X\setminus B.
  \end{eqnarray}
 
 \item \label{item8} Assume $\eqref{unifbounded1}$ and $\eqref{unifbounded2}$ hold. Let $x_0\in \X$, $t\in(0,1)$ and $\gamma\in(0,1)$. For any $w,z,z'\in \X$ with $d(z,z')\leq 2$ and $w\in \supp(\nu_0)$ one has 
 \[
 \frac{P_t^\gamma(z',w)}{P_t^\gamma(z,w)}\leq \frac{\max\left(1, d(x_0,z)^{d(z,z')}\right)  K^{d(x_0,z)}\,O(1)}{(\gamma t)^{d(z,z')}} ,
 \]
 where $K:=2S/I$ and $O(1)$ is a positive constant that does not depend on $z,z',\gamma,t$. 
It follows that  
\begin{equation}\label{encadrement}
 \frac{(\gamma t)^{d(z,z')}} {\max\left(1, d(x_0,z)^{d(z,z')}\right)  K^{d(x_0,z)}\,O(1)}\leq \frac{\P_t^\gamma f^\gamma(z')}{\P_t^\gamma f^\gamma(z)} \leq \frac{\max\left(1, d(x_0,z)^{d(z,z')}\right)  K^{d(x_0,z)}\,O(1)}{(\gamma t)^{d(z,z')}}.
 \end{equation}

 \item  \label{item9} Let $(\gamma_\ell )_{\ell\in \N}$ be a  sequence of positive numbers converging to zero. If \eqref{unifbounded0}, \eqref{unifbounded1},  $\eqref{unifbounded2}$ and \eqref{convhyp} hold, then 
 for any $t\in[0,1]$
 \[\lim_{\gamma_\ell \to 0} H(\widehat Q_t^{\gamma_\ell }|m) =H(\widehat Q_t^{0}|m).\]
\end{enumerate}
\end{lemma}

\begin{proof}
\begin{enumerate}[label=(\roman*)]
\item Given \eqref{unifbounded1}, we want to show that for any $x\in \X$, 
$S_k(y):= \sup_{x\in \X}  |L^k(x,y)|\leq (2S)^k$. It follows by induction on $k$ from the inequality 
\[S_{k+1}(y)=\sup_{x\in \X} \Big|L(x,x) L^k(x,y)+\sum_{z, z\sim x} L(x,z) L^k(z,y) \Big|\leq 2\sup_{x\in \X}|L(x,x)| \,\, S_{k}(y).\]

\item For $x=y$, one has $L^{d(x,y)}(x,y)=1$ and  by definition for $x\neq y$,
\[L^{d(x,y)}(x,y):= \sum_{\alpha}L_\alpha,\]
where the sum is over all path $\alpha$  from $x$ to $y$ of length $d(x,y)$, $\alpha=(z_0,\ldots, z_{d(x,y)})$ with $z_0=x$ and $z_{d(x,y)}=y$, and  
\[L_\alpha:=L(z_0,z_1)L(z_1,z_2)\ldots L(z_{d(x,y)-1},z_{d(x,y)}).\]
Such a path $\alpha$ is a geodesic.
Since we assume in this paper that $L(x,y)>0$ if and only if $x$ and $y$ are neighbour, one has $L_\alpha>0$.
By irreducibility it always exists at most one geodesic path from $x$ to $y$, and from assumption \eqref{unifbounded1}, for  such a   path $\alpha$, $L_\alpha\geq I^{d(x,y)}$. As a consequence we get $L^{d(x,y)}(x,y)\geq I^{d(x,y)}$.

\item According to  \eqref{formePt}, for any $x,y\in \X$, 
\begin{eqnarray*}\P_t^{\gamma}(x,y) = \frac{ L^{d(x,y)}(x,y)}{d(x,y)!} \, (\gamma t)^{d(x,y)}\left( 1+ \gamma\sum_{k,k\geq d(x,y)+1}  \frac{ L^k(x,y)}{L^{d(x,y)}(x,y)}\,\frac{d(x,y)!}{k!}\,t^{k-d(x,y)}\gamma^{k-d(x,y)-1}\right).
\end{eqnarray*}
Applying Lemma \ref{lemmetech} \ref{item1} and  \ref{item2}, we get 
\begin{align*}
\Big|\P_t^{\gamma}(x,y) - & \frac{ L^{d(x,y)}(x,y)}{d(x,y)!} \, (\gamma t)^{d(x,y)}\Big|
\\&\leq   \gamma   \, \frac{ L^{d(x,y)}(x,y)}{d(x,y)!} \, (\gamma t)^{d(x,y)} \sum_{k,k\geq d(x,y)+1} K^{d(x,y)} \frac{(2S)^{k-d(x,y)}}{(k-d(x,y))!}\\
&\leq\gamma   \, \frac{ L^{d(x,y)}(x,y)}{d(x,y)!} \, (\gamma t)^{d(x,y)} K^{d(x,y)} e^{2S},
\end{align*}
from which the expected result follows.

\item Let $x,y,z\in \X$ and $t\in[0,1]$. If \eqref{unifbounded1} holds, according to 
\eqref{formePt}, the Taylor expansion of $\P_t^{\gamma}(x,y)$ as $\gamma$ goes to zero is given by 
\[\P_t^{\gamma}(x,y) = \frac{ L^{d(x,y)}(x,y)}{d(x,y)!} \, (\gamma t)^{d(x,y)} + o(\gamma^{d(x,y)}),\]
As a consequence, the Taylor expansion of  ${Q_t^\gamma}^{x,y}(z)$, defined by \eqref{definut}, is 

\begin{multline*}{Q_t^\gamma}^{x,y}(z)=\gamma^{d(x,z)+d(z,y)-d(x,y)} \frac{L^{d(x,z)}(x,z) L^{d(z,y)}(z,y)}{L^{d(x,y)}(x,y)} \,\frac{d(x,y)!}{d(x,z)!d(z,y)!} \,t^{d(x,z)}(1-t)^{d(z,y)} \\+ o(\gamma^{d(x,z)+d(z,y)-d(x,y)}).
\end{multline*}
The expected result follows since one has $\gamma^{d(x,z)+d(z,y)-d(x,y)}=1$ if $z\in[x,y]$, and  \\$\lim_{\gamma\to 0} \gamma^{d(x,z)+d(z,y)-d(x,y)}=0 $ otherwise. 

\item  On some probability space $(\Omega',\mathcal{ A}, \mathbb{ P})$, let $(N_s)_{s\geq 0}$ be a Poisson process with parameter $\gamma S$ and $(Y_n)_{n\in \N}$ be a Markov chain on $\X$   with transition matrix $K$  given by 
\[{\rm K}(z,w): =\frac {L^\gamma (x,w)}{\gamma S}, \quad \mbox{for}\, w\neq z\in \X, \mbox{ and } \quad{\rm K}(z,z):=\frac{\gamma  S+L^\gamma(z,z)}{\gamma S}.\]
We assume that $(Y_n)_{n\in \N}$ and $(N_s)_{s\geq 0}$ are independent.
 It is well known that the law of the process $(X_t)_{t\geq 0}$ under $R^{\gamma}$ given $X_0=x$ is the same as the law of the process $(\widetilde X_t)_{t\geq 0}$ under $\mathbb{ P}$ given $\widetilde X_0=x$   defined by 
 $\widetilde X_t:=Y_{N_t}$. 
 As a consequence, one has for any $y\in \X$,
 \[\P^\gamma_t(x,y)=R^{\gamma}\left(X_t=y\,|\,X_0=x\right)={\mathbb P}\left(\widetilde X_t=y\,|\,\widetilde X_0=x\right).\]
 Let $n=d(x,y)$ and  $\widetilde N_t$ denotes the number of jumps of the process $\widetilde X_t$,
 one has 
 \begin{align*}
 \P_t^\gamma(x,y)&\geq {\mathbb P}\left(\widetilde X_t=y, \widetilde N_t =n \,|\,\widetilde X_0=x\right)\\
 &= {\mathbb P}\left(Y_1,\ldots, Y_{n} \mbox{ are all different}, Y_{n}=y,  N_t =n\, |\,\widetilde X_0=x\right)\\
 &={\mathbb P}\left(N_t =n)\, {\mathbb P}(Y_1,\ldots, Y_{n} \mbox{ are all different}, Y_{n}=y \,|\,\widetilde X_0=x\right)\\
 &=\frac{(\gamma t S)^{n}}{n!} \,e^{-\gamma t S} \sum_{\alpha=(x_0,\ldots,x_{n}),\,\alpha \, {\rm geodesic} \, {\rm from}\,  x\, {\rm to} \, y} {\rm K}(x_0,x_1)\cdots {\rm K}(x_{n-1},x_n)\\
 &=\frac{(\gamma t )^{n}}{n!} \,e^{-\gamma t S} L^{d(x,y)}(x,y).
 \end{align*}
 This ends the proof of the first part of (v). Observe that from the Schr\"odinger system \eqref{SS}, $f^{\gamma}(w)>0$ if and only if $w\in\supp(\nu_0)$. Since $\nu_0$ has bounded support, it follows that for any $w\in \supp(\nu_0)$,
 \[0<\min_{u\in \supp(\nu_0)} f^{\gamma}(u)\leq f^{\gamma}(w)\leq \max_{u\in \supp(\nu_0)} f(u),\]
 and therefore for any $z\in\X$,
 \[\min_{u\in \supp(\nu_0)} f(u) \min_{w\in \supp(\nu_0)} \P^\gamma_t(z,w) \leq \sum_{w\in \supp(\nu_0)} f^{\gamma}(w)  \P^\gamma_t(z,w) = \P^\gamma_t f^{\gamma}(z)\leq \max_{u\in \supp(\nu_0)} f(u).\]
 From $\eqref{unifbounded2}$ and (ii) and since $d(z,w)\leq d(z,x_0)+1+D$ for any $w\in \supp(\nu_0)$, one gets 
 \[\min_{w\in \supp(\nu_0)} \P^\gamma_t(z,w)\geq e^{-S}\left(\frac{t\gamma I}{d(x_0,z)+1+D}\right)^{d(x_0,z)+1+D},\]
from which the second part of (v) follows. 
 
\item The length $\ell(\omega)$ of a path $\omega\in \Omega$ represents the number of jumps of the process $X_t$ between times 0 and 1. Therefore according to the definition of the process $(\widetilde X_t)_{t\geq 0}$ above, 
\begin{align*}
{\mathbb E}_{R^\gamma}&[\ell\,|\,X_0=x,X_1=y]={\mathbb E}_{\mathbb P}\left[\widetilde N_1\,| \,\widetilde X_0=x,  \widetilde X_1=y\right] \\
&\leq  {\mathbb E}_{\mathbb P}\left[ N_1\,|\, \widetilde X_0=x,  \widetilde X_1=y\right] =  \frac{ {\mathbb E}_{\mathbb P}\left[ N_1\1_{\widetilde X_1=y} \,| \,\widetilde X_0=x\right]}{{\mathbb P}\left(\widetilde X_1=y\,|\, \widetilde X_0=x\right)}\leq  \frac{ {\mathbb E}_{\mathbb P}\left[ N_1\right]}{\P^\gamma_1(x,y)},
\end{align*}
which ends the proof since $ {\mathbb E}_{\mathbb P}\left[ N_1\right]=\gamma S$.

\item From \ref{item3} and \ref{item5}, one gets for any $x,z,y\in \X$,
\begin{eqnarray}\label{hic}
 {Q_t^\gamma}^{x,y}(z)&=&\nonumber \frac{\P_t^\gamma(x,z)\P_{1-t}^\gamma(z,y)}{P_1^\gamma(x,y)}\\ \nonumber
 &\leq& \gamma^{d(x,z)+d(z,y)-d(x,y)} r(x,z,z,y)\, \frac{d(x,y)!}{d(x,z)!d(z,y)!} t^{d(x,z)}(1-t)^{d(z,y)}\,e^{\gamma S}\\
 &~&\qquad\qquad\qquad\qquad\qquad\qquad\qquad\left(1+\gamma K^{d(x,z)}O(1)\right)\left(1+\gamma K^{d(z,y)}O(1)\right).
\end{eqnarray}
If $z\in [x,y]$ then thanks to \ref{item1} and  \ref{item2}, the right-hand side of this inequality is bounded from above by 
\[\left(\frac{2S}{I}\right)^{d(x,y)}e^{d(x,y)}e^{\gamma S}4 K^{2d(x,y)}O(1),\]
and the maximum of this quantity over all $x\in {\rm supp}(\nu_0)$ and $y\in {\rm supp}(\nu_1)$ is a constant $O(1)$, independent of $x,z,y$ and $\gamma$.\\
If $z\not \in [x,y]$, then $d(x,z)+d(z,y)-d(x,y)\geq \max\{1, 2d(x_0,z)-4D\}$, and the right-hand side of \eqref{hic} is bounded by
\begin{multline*}
\gamma^{d(x,z)+d(z,y)-d(x,y)}\,\frac{(2S)^{d(x,z)+d(z,y)}}{I^{d(x,y)}}\,d(x,y)!\,e^{\gamma S}4K^{d(x,z)+d(z,y)}O(1)\\ \leq \gamma^{1+[2d(x_0,z)-4D-1]_+} \frac{(2S)^{2d(x_0,z)+ 2D}}{I^{d(x,y)}}\,d(x,y)!\,e^{\gamma S}4K^{2d(x_0,z)+2D} O(1).
\end{multline*}
The maximum over all $x\in {\rm supp}(\nu_0)$ and $y\in {\rm supp}(\nu_1)$ of the right-hand side   quantity is bounded by $O(1)\,\gamma^{1+[2d(x_0,z)-4D-1]_+} K^{4d(x_0,z)}$. This ends the proof of the first inequality of \ref{item7}. The second inequality easily follows since 
\[\widehat Q_t^\gamma(z)= \sum_{x\in \supp(\nu_0),y\in\supp(\nu_1) } {Q_t^\gamma}^{x,y}(z) \,\,\widehat \pi^\gamma(x,y).\]

\item Using \ref{item3} and \ref{item5}, one gets for any $z,z'\in \X$ and any $w\in \supp(\nu_0)$,
\begin{align*}
&\frac{\P_t^\gamma(z',w)}{\P_t^\gamma(z,w)}\leq \frac{L^{d(z',w)}(z',w)}{L^{d(z,w)}(z,w)}\, \frac{d(z,w)!}{d(z',w)!}\,\left(\frac{1}{\gamma t}\right)^{d(z,w)-d(\tz,w)} e^{\gamma t S}\left(1+\gamma K^{d(z',w)} O(1)\right)\\
&\leq K^{d(z,z')+ d(z,x_0)+d(x_0,w)}\max\left(1,d(z,w)^2\right) \left(\frac{1}{\gamma t}\right)^{d(z,z')}2e^{ S}K^{d(z,z')+ d(z,x_0)+d(x_0,w)}O(1) \\
&\leq \frac{K^{2d(z,x_0)}\max\left(1,d(z,x_0)^2\right)O(1)}{(\gamma t)^{d(z,z')}}, 
\end{align*}
where one maximizes over all $w\in \supp(\nu_0)$ to get the last inequality.
Inequality \eqref{encadrement} follows since 
\[ 
\frac{\P_t^\gamma f^\gamma(z')}{\P_t^\gamma f^\gamma(z)}= \sum_{w\in \supp(\nu_0) }\frac{\P_t^\gamma(z',w)}{\P_t^\gamma(z,w)} \frac{f^\gamma(w)\P_t^\gamma(z,w)}{\P_t^\gamma f^\gamma(z)},
\]
with $\displaystyle\sum_{w\in \supp(\nu_0) } \frac{f^\gamma(w)\P_t^\gamma(z,w)}{\P_t^\gamma f^\gamma(z)}=1$.

\item Recall that 
\[H(\widehat Q^{\gamma_\ell}_t|m)= \sum_{z\in \X}\log \frac{\widehat Q^{\gamma_\ell }_t(z)}{m(z)}\, \widehat Q^{\gamma_\ell }_t(z).
\]
Let us consider  the finite set $B$ defined in Lemma \ref{lemmetech} \ref{item7}. From the weak convergence of the sequence $(\widehat Q^{\gamma_\ell }_t)$ to $\widehat Q^{0}_t$ and since $\supp (\widehat Q^{0}_t)\subset B$, one has 
\[\lim_{\gamma_\ell \to 0} \sum_{z\in B}\log \frac{\widehat Q^{\gamma_\ell }_t(z)}{m(z)}\, \widehat Q^{\gamma_\ell }_t(z)= H(\widehat Q^{0}_t|m).\]
Therefore it remains to prove that 
\[\lim_{\gamma_\ell \to 0} \sum_{z\in \X\setminus B}\log \frac{\widehat Q^{\gamma_\ell }_t(z)}{m(z)}\, \widehat Q^{\gamma_\ell }_t(z)=0.\]
From Lemma \ref{lemmetech} \ref{item7} and hypothesis \eqref{unifbounded0}  one has, for any $z\in \X\setminus B$, 
\[\frac{\widehat Q^{\gamma_\ell }_t(z)}{m(z)}\leq \frac {O(1)\,\gamma_\ell  \,\left(\gamma_\ell  K^2\right)^{[2d(x_0,z)-4D-1]_+}}{\inf_{z\in\X} m(z)}.\]
Using the inequality $|v\log v|\leq \sqrt v$ for $v\in(0,1]$, we get for $0<\gamma_\ell \leq \min\left(\frac{\inf_{z\in X}m(z)}{O(1)},\frac1{K^2}\right)$,
\[
\sum_{z\in \X\setminus B}\log \frac{\widehat Q^{\gamma_\ell }_t(z)}{m(z)}\, \widehat Q^{\gamma_\ell }_t(z)
\leq O(1)\sup_{z\in \X} m(z) \sqrt {\gamma_\ell} \sum_{z\in \X} \left(\gamma_\ell K^2\right)^{[2d(x_0,z)-4D-1]_+/2}. 
\]
Hypothesis \eqref{convhyp} then implies that there exists $\tilde\gamma>0$ such that for any $0<\gamma_\ell<\tilde \gamma$
\[
\sum_{z\in \X\setminus B} \log \frac{\widehat Q^{\gamma_\ell }_t(z)}{m(z)}\, \widehat Q^{\gamma_\ell }_t(z)
\leq  O(1) \sqrt\gamma_\ell, 
\] 
and the expected result follows.
\end{enumerate}
\end{proof}


\section{Appendix B : Proofs of Lemmas \ref{lemconvex}, \ref{deriv1phi}, \ref{deriv2phibis},  and \ref{lemintaibis}}\label{appB}
\begin{proof}[Proof of Lemma \ref{deriv1phi} and Lemma \ref{deriv2phibis}.] Let $\gamma$  denotes a fixed parameter of temperature that can be chosen as small as we want. To simplify the notations,  the dependence in the temperature parameter $\gamma$ is sometimes omitted. 
For $t\in(0,1)$, let us note $f_t:= P_t^\gamma f^\gamma$ and $g_t:=P_{1-t}^\gamma g^\gamma$ and  recall that  $F_t :=\log f_t$,  $G_t:=\log   g_t$ and 
\[\varphi(t)=\int F_t   f_t \, g_t \,dm, \qquad \psi(t)=\int G_t  f_t \,g_t \,dm.\]
Observe that for $\gamma$ sufficiently small,  these two functions are well defined on $(0,1)$ since  \eqref{encadrementbis} and   \eqref{majQt} implies
\begin{multline*}
 \int |F_t|   f_t \, g_t \,dm =\sum_{z\in \X} \left|\log(P_t^\gamma f^\gamma(z))\right| \widehat Q_t^\gamma(z)\\
\leq O(1) +O(1)\sum_{z\in \X\setminus B} (d(x_0,z)+1+D) \left(\log\frac 1{t\gamma I}+\log\left({d(x_0,z)+1+D}\right)\right) \,\gamma \left(\gamma K^2\right)^{[2d(x_0,z)-4D-1]_+}.
\end{multline*}
According to hypothesis \eqref{convhyp}, the right-hand side of this inequality is finite if $(\gamma K^2)^2<\gamma_o$.
Identically,  one could check that $\int |G_t|  f_t \,g_t \,dm$ is finite for $\gamma$ sufficiently small.

The proof is based on $\Gamma_2$-calculus by using  backward equations, $\partial_t f_t =L f_t$,  $\partial_t g_t=-L g_t$. 
We  only present the proof of the expression of $\varphi'(t)$ and $\varphi''(t)$. Same arguments  provide the expression of $\psi'(t)$ and $\psi''(t)$. 
We start with a general statement that we will apply twice.
Let $(t,z)\in(0,1)\times\X \to V_t(z)\in \R$ denotes some differentiable function in $t$ (that also depends of the parameter $\gamma$) satisfying for any $\varepsilon \in (0,1/2)$, and any $x_0\in\X$,
\begin{eqnarray}\label{coro1}
\sup_{t\in (\varepsilon,1-\varepsilon)} |V_t(z)|\leq O(1) \frac{A^{d(x_0,z)}}{\gamma^{10}},
\end{eqnarray}
and 
\begin{eqnarray}\label{coro2}
\sup_{t\in (\varepsilon,1-\varepsilon)} |\partial_tV_t(z)|\leq O(1) \frac{B^{d(x_0,z)}}{\gamma^{10}},
\end{eqnarray}
for all $z\in\X$
where $O(1),A,B$ denote constants that do not depend on $t,\gamma$ and $z$. 
Then  the following identity holds: 
for any $t\in (0,1)$,
 \begin{align}\label{escala}
\partial_t \left (\int V_t  f_t \,g_t \,dm\right)&=\int \partial_t(V_t  f_t \,g_t)\,dm \nonumber\\
&=\int (\partial_t V_t) \,f_t  \, g_t+V_t \,(Lf_t)\,  g_t - V_t \,f_t  \,(Lg_t)\,dm \nonumber\\
&=\int (\partial_t V_t) \,f_t  \,  g_t+V_t \,(Lf_t)\,    g_t - L(V_t f_t ) g_t \,dm\nonumber \\
&= \int \Big[\partial_t V_t(z) - \sum_{\tz, \,\tz\sim z} e^{\nabla F_t(z,\tz) }\nabla V_t(z,\tz)\,L(z,\tz)\Big]f_t(z) g_t(z)\,dm(z).
\end{align}
It suffises to justify this identity for any $\varepsilon \in (0,1/2)$ and any $t\in (\varepsilon,1-\varepsilon)$.
The second equality of \eqref{escala} is due to the backward equations. The first equality of \eqref{escala} is justified by applying  Lebesgue's theorem with hypothesis \eqref{convhyp}, provided that for $\gamma$ sufficiently small, one has
\[\sup_{t\in (\varepsilon,1-\varepsilon)}|\partial_t(V_t  f_t \,g_t)(z) \,m(z)|\leq O(1) \,\gamma_o^{d(x_0,z)}.\]
This is indeed the case, since for any $z\in \X$,
\[\partial_t(V_t  f_t \,g_t)(z) \,m(z)= \left[(\partial_t V_t)(z)+V_t(z) \frac{L\P_t^\gamma f^\gamma(z)}{\P_t^\gamma f^\gamma(z)}- V_t(z) \frac{L\P_{1-t}^\gamma g^\gamma(z)}{\P_{1-t}^\gamma g^\gamma(z)}\right]\widehat Q_t^\gamma(z),\]
with according to \eqref{encadrement}, for any $t\in(\varepsilon,1)$,
\[\left|\frac{L\P_t^\gamma f^\gamma(z)}{\P_t^\gamma f^\gamma(z)}\right|\leq S d_{\rm max} \left(1+\max_{\tz, \tz\sim z}\left|\frac{\P_t^\gamma f^\gamma(z')}{\P_t^\gamma f^\gamma(z)}\right|\right)\leq Sd_{\rm max}  
\frac{\max(1, d(x_0,z))  K^{d(x_0,z)}\,O(1)}{\gamma \varepsilon}\leq O(1)  \,\frac{K^{d(x_0,z)}}\gamma .\]
One identically shows that $\left|\frac{L\P_{1-t}^\gamma g^\gamma(z)}{\P_{1-t}^\gamma g^\gamma(z)}\right|\leq O(1)  \frac{K^{d(x_0,z)}}\gamma,$ for any $t\in(0,1-\varepsilon)$ and $z\in \X$.
Together with \eqref{majQt}, we get the bound, for any $z\in\X$ and $t\in(\varepsilon,1-\varepsilon)$,
\[|\partial_t(V_t  f_t \,g_t)(z) m(z)|\leq O(1)\left(B^{d(x_0,z)}+ (AK)^{d(x_0,z)}\right)\frac{\left(\gamma K^2\right)^{2d(x_0,z)}}{\gamma^{11}}\leq O(1) \,\gamma_o^{d(x_0,z)},\]
for any $\gamma>0$ with $\gamma^2(B+AK) K^4\leq \gamma_o$.
The third equality of  \eqref{escala}  is due to Fubini's theorem together with the reversibility property of  $m$ with respect to $L$. The last equality of  \eqref{escala} is a simple rearrangement of the terms. 

At first, one  applies   \eqref{escala} with $V_t=F_t$, since according to \eqref{encadrementbis}, for any $t\in(\varepsilon, 1-\varepsilon)$, for any $z\in\X$,
\[|F_t(z)| \leq O(1)  \left({d(x_0,z)+1+D}\right)\left(\log\frac 1{\varepsilon \gamma I}+\log\left({d(x_0,z)+1+D}\right)\right)\leq O(1) \,\frac{2^{d(x_0,z)}}\gamma,\]
and 
\[|\partial_t F_t(z)|= \left|\frac{L\P_t^\gamma f^\gamma(z)}{\P_t^\gamma f^\gamma(z)}\right|\leq 
O(1) \, \frac{K^{d(x_0,z)}}\gamma .\]

\begin{equation*}
\partial_t F_t (z)=\sum_{\tz\in \X} e^{\nabla F_t(z,z') }L(z,z')= \sum_{\tz, \,\tz\sim z} \left(e^{\nabla F_t(z,z') }-1\right)L(z,z'),\qquad z\in\X, 
\end{equation*}
one gets the expected result
\begin{align*}
\varphi'(t)&=\int \sum_{\tz,\, \tz\sim z} \left(e^{\nabla F_t(z,z') }-1- \nabla F_t(z,z') e^{\nabla F_t(z,z') }\right)L(z,\tz) f_t(z) g_t(z)\,dm(z)\\
&=-\int \sum_{\tz,\, \tz\sim z}\zeta\left( e^{\nabla F_t(z,z') }\right) L(z,\tz) \,d\widehat Q^\gamma_t(z). 
\end{align*}

We want  now to apply  again \eqref{escala} with $V_t(z)=\sum_{\tz, \tz\sim z}\zeta\left( e^{\nabla F_t(z,z') }\right) L(z,\tz) $, $z\in \X$. 
From the inequality,  $|\zeta(a)|\leq 2 +a^2,a>0$ and  using \eqref{encadrement},  one may check as above that \eqref{coro1} holds.
The backward equations ensure that  
\begin{align*}
\partial_t V_t(z)&= \sum_{\tz,\, \tz\sim z} \left(\frac{Lf_t(z')}{f_t(z)} -\frac{f_t(z')Lf_t(z)}{f_t^2(z)}\right)\zeta'\left( e^{\nabla F_t(z,z') }\right)L(z,z')\\
&= \sum_{\tz,\, \tz\sim z} e^{\nabla F_t(z,z') }\left(\frac{Lf_t(z')}{f_t(\tz)} -\frac{Lf_t(z)}{f_t(z)}\right)\nabla F_t(z,z') \,L(z,z')\\
&=\sum_{\tz,\,\ttz,\, z\sim \tz\sim \ttz} \nabla F_t(z,z') \,e^{\nabla F_t(z,z') }\left(e^{\nabla F_t(\tz,\ttz)}-1\right) L(z,z')\,L(\tz,\ttz) \\
&\qquad-\sum_{\tz,\,\tw, \,\tz\sim z, \,\tw \sim z} \nabla F_t(z,z')\, e^{\nabla F_t(z,z') }\left(e^{\nabla F_t(z,\tw)} -1\right) L(z,z')\,L(z,\tw).
\end{align*}
Simple computations  together with \eqref{encadrement} show that \eqref{coro2} holds too.

Applying the identity \eqref{escala}, since 
\begin{align*}
\sum_{\tz,\, \tz\sim z} e^{\nabla F_t(z,\tz) }\nabla V_t(z,\tz)\,L(z,\tz)&=
\sum_{\tz,\,\ttz,\, z\sim \tz\sim \ttz}e^{\nabla F_t(z,\tz) } \zeta\left( e^{\nabla F_t(\tz,\ttz) }\right) L(z,z')\,L(\tz,\ttz) \\
&- \sum_{\tz,\,\tw, \,\tz\sim z, \,\tw \sim z} e^{\nabla F_t(z,\tz)} \zeta\left( e^{\nabla F_t(z,\tw) }\right) 
 L(z,z')\,L(z,\tw),
\end{align*}
one gets for any $t\in(0,1)$, 
\begin{align*}
\varphi''(t)
&=-\int \bigg[\sum_{\tz,\,\tw, \,\tz\sim z, \,\tw \sim z} \!\!\!\! \!\!\!\!\Big[ \zeta\left( e^{\nabla F_t(z,\tw) }\right)- \nabla F_t(z,z') \left(e^{\nabla F_t(z,\tw)} -1\right)
\Big] e^{\nabla F_t(z,\tz)}L(z,z')\,L(z,\tw)\\
&+ \!\!\!\! \!\!\!\!\sum_{\tz,\,\ttz,\, z\sim \tz\sim \ttz}  \!\!\!\! \!\!\!\!\Big[\nabla F_t(z,z') \left(e^{\nabla F_t(\tz,\ttz)}-1\right) - \zeta\left( e^{\nabla F_t(\tz,\ttz) }\right)\Big]e^{\nabla F_t(z,z') } L(z,z')\,L(\tz,\ttz)\bigg] d\widehat Q^\gamma_t(z)\\ 
&=-\int \Big[\sum_{\tz,\,\tw,\, \tz\sim z,\, \tw \sim z} \Big( \left(\nabla F_t(z,\tw) - \nabla F_t(z,\tz)\right) -1\Big)  \,e^{\nabla F_t(z,\tw)+\nabla F_t(z,\tz)} L(z,z')\,L(z,\tw)\\
&\qquad\qquad\left.+\sum_{\tz,\,\tw,\, \tz\sim z,\, \tw \sim z} \left(\nabla F_t(z,\tz)+1\right)e^{\nabla F_t(z,\tz)}L(z,z')\,L(z,\tw)\right.\\
&\qquad\qquad\left.
-\sum_{\tz,\,\ttz, \,z\sim \tz\sim \ttz}\left(\nabla F_t(z,\tz)+1\right)e^{\nabla F_t(z,\tz)}L(z,\tz)\,L(\tz,\ttz)
\right.\\
&\qquad\qquad-\sum_{\tz,\,\ttz, \,z\sim \tz\sim \ttz} \rho\left( e^{\nabla F_t(z,\tz) },e^{\nabla F_t(z,\ttz) }\right) L(z,z')\,L(\tz,\ttz)\Big] d\widehat Q^\gamma_t(z),
\end{align*}
where the last equality holds since $\nabla F_t(z,\tz)+\nabla F_t(\tz,\ttz)=\nabla F_t(z,\ttz)$.
The expected expression of $\varphi''(t)$ follows by symmetrization of the first sum in $\tz$ 
and $\tw$,
and since $\sum_{\tw,\, \tw \sim z} L(z,\tw)=-L(z,z)$.
\end{proof}

\begin{proof}[Proof of Lemma \ref{lemconvex}]
Let $\varepsilon \in(0,1/2)$. We first prove that if  \eqref{unifbounded1}, \eqref{unifbounded2} and \eqref{convhyp} hold then $\varphi''_\gamma(t)$ is  uniformly lower bounded  over all $t\in [\varepsilon, 1]$ and $\gamma\in(0,\bar \gamma]$  for some $\bar\gamma\in(0,1)$. 
According to \eqref{decomp} and inequality \eqref{majM} and \eqref{majR}, for any $t\in [\varepsilon, 1]$ and $\gamma>0$,
\begin{align*}
  \varphi''_\gamma(t)&\geq -O(1)\left[\frac{|\gamma\log\gamma|}{\varepsilon}\int d^2(x_0,z)K^{d(x_0,z)} d\widehat Q_{t}^\gamma(z)+\frac1{\varepsilon^2}\int \Big(d^2(x_0,z)+1\Big) K^{2 d(x_0,z)}d\widehat Q_{t}^\gamma(z)\right]\\
  &\geq -O(1) \int d^2(x_0,z) K^{2 d(x_0,z)}d\widehat Q_{t}^\gamma(z),
\end{align*}
where $O(1)$ denotes a positive constant that only depends on $\bar \gamma$ and $\varepsilon$.
Using  Lemma \ref{lemmetech} \ref{item7} and the fact that $\nu_0$ and $\nu_1$ have bounded support, it follows that
\begin{align*}
  \varphi''_\gamma(t)&\geq-  O(1) \!\! \!\! \!\! \!\!\sum_{x\in \supp(\nu_0), y\in\supp(\nu_1)}\!\! \max_{ z\in [x,y]}\left(d^2(x_0,z) K^{2 d(x_0,z)}\right) - O(1) \sum_{z\in \X}d^2(x_0,z)\left(\gamma K^3\right)^{[2 d(x_0,z)-4D-1]_+} \\
  &=-O(1) -O(1)\sum_{z\in \X}d^2(x_0,z)\left(\gamma K^3\right)^{[2 d(x_0,z)-4D-1]_+}
\end{align*}
From hypothesis \eqref{convhyp}, choosing $\bar \gamma$ so that 
$(\bar \gamma K^3)^2<\gamma_o$, one gets 
\[\inf_{\gamma\in(0,\bar \gamma), t\in [\varepsilon, 1]} \varphi''_\gamma(t) \geq -O(1) .\]
One may similarly proved by symmetry that if \eqref{unifbounded1}, \eqref{unifbounded2} and  \eqref{convhyp} hold, then  $-\psi''_\gamma(t)$ is also uniformly lower bounded, namely 
\[\inf_{\gamma\in(0,\bar \gamma), t\in [0, 1-\varepsilon]} \psi''_\gamma(t)\geq - O(1) .\]

Let $\varepsilon\in(0,1/2)$,  and for $\gamma\in [0,1)$, let 
\[F_\gamma^\varepsilon(t)= H(\widehat Q^\gamma_{(1-\varepsilon)t+\varepsilon(1-t)}|m), \qquad t\in [0,1].\]
We will first prove a convexity property for the function $F_0^\varepsilon$  from a convexity property  of  $F^{\gamma_\ell }_\varepsilon$ as the sequence $(\gamma_\ell )$ goes to zero.  
We use the identity, for any $t\in(0,1)$
\begin{equation}\label{identconv}
(1-t) F_{\gamma_\ell }^\varepsilon(0)+tF_{\gamma_\ell }^\varepsilon(1)- F_{\gamma_\ell }^\varepsilon(t)=\frac{t(1-t)}2\int_0^1 K_t(s){(F_{\gamma_\ell }^\varepsilon)}''(s)\,ds,
\end{equation}
where the kernel $K_t$ is defined by \eqref{noyau}.
Observe that  
\[\int_0^1 K_t(s){(F_{\gamma_\ell }^\varepsilon)}''(s)\,ds=(1-2\varepsilon)  \int_\varepsilon^{1-\varepsilon} K_t\left(\frac{u-\varepsilon}{1-2\varepsilon}\right)\left(\varphi_{\gamma_\ell }''(u) +\psi_{\gamma_\ell }''(u)\right) du.\]
The above uniform bounds on $\varphi_{\gamma}''$ and  $\psi_{\gamma}''$  for $\gamma\in(0,\bar \gamma)$ allow to apply Fatou's Lemma. Together with  Lemma \ref{lemmetech} \ref{item9} it implies, for any $\varepsilon\in(0,1/2)$
\begin{multline}\label{equa9}
    (1-t) F_{0}^\varepsilon(0)+tF_{0}^\varepsilon(1)- F_{0}^\varepsilon(t) \geq \frac{t(1-t)}2   (1-2\varepsilon)\int_\varepsilon^{1-\varepsilon} K_t\left(\frac{u-\varepsilon}{1-2\varepsilon}\right)\liminf_{\gamma_\ell \to 0}\left(\varphi_{\gamma_\ell }''(u) +\psi_{\gamma_\ell }''(u)\right) du.
\end{multline}
For any $t\in[0,1] $ the support of the measure $\widehat Q_t$ is finite, included in the set $B$ defined  Lemma \ref{lemmetech} \ref{item7}. As a consequence, the function $t\in[0,1]\to H(\widehat Q_t|m)$ is  continuous  as a finite sum of continuous functions. It follows that for any $t\in[0,1]$,
\[\lim_{\varepsilon \to 0} F_{0}^\varepsilon(t)= H(\widehat Q_t|m).\]
Consequently, using hypothesis \eqref{condlemconvex} and applying Fatou's Lemma  as $\varepsilon$ goes to zero,   equality \eqref{equa9} provides 
\begin{align*}
      (1-t) H(\nu_0|m)+tH(\nu_1|m)- H(\widehat Q_t|m)&\geq  \frac{t(1-t)}2\int_0^1 K_t\left(u\right)\Big(\liminf_{\gamma_\ell \to 0}\varphi_{\gamma_\ell }''(u) +\liminf_{\gamma_\ell \to 0}\psi_{\gamma_\ell }''(u)\Big) \,du\\
      &\geq  \frac{t(1-t)}2\int_0^1 K_t\left(u\right) \xi''(u)  \,du\\
      &= (1-t)\xi(0)+t\xi(1)-\xi(t)
\end{align*}
were the last equality is a consequence of identity \eqref{identconv} applied with  $\xi$.
\end{proof}

\begin{proof}[Proof of Lemma \ref{lemintaibis}]
Let $z\in \widehat Z$ and $\tz\in V(z)$. One will only compute  the expression of $\lim_{\gamma_\ell \to 0} \left(\gamma_\ell  A_t^{\gamma_\ell }(z,z')\right)$ and similar calculations  provide $\lim_{\gamma_\ell \to 0} \left(\gamma_\ell  B_t^{\gamma_\ell }(z,z')\right)$.
For any $\gamma>0$,
let 
\[a_t^\gamma(z,y) :=\widehat Q^\gamma(X_t=z|X_1=y)= \int  {Q_t^{\gamma}}^{w,y}(z) \,d\widehat{\pi}^{\gamma}_{_\leftarrow}(w|y)
,\]
and 
\[
{\mathrm a}_t^\gamma(z,\tz,y):=\int  \alpha_t^\gamma(y,z,\tz,w) \,d\widehat{\pi}^\gamma_{_\leftarrow}(w|y),\quad \mbox{with}\quad \alpha_t^\gamma(y,z,\tz,w)=\, \frac{P_{1-t}^\gamma(y,z)P^\gamma_t(\tz,w)}{P_1^\gamma(y,w)}.\]
Using equality  \eqref{defnoyaubis} and since  $P^\gamma_1 f^\gamma(y)>0$ for any $\gamma>0$, one easily check that for any $\gamma>0$, 
\[ A_t^\gamma(z,z') = \,\frac{P_t^\gamma f^\gamma (\tz)}{P_t^\gamma f^\gamma(z)}=\frac{{\mathrm a}_t^\gamma(z,\tz,y)}{a_t^\gamma(z,y)}.\]

From the expression \eqref{a_t} of $a_t(z,y)$ and since $\supp(\widehat{\pi}^{\gamma_\ell }_{_\leftarrow}(\cdot|y))\subset\supp(\nu_0)$, one  has 
\[\left|\, a^{\gamma_\ell }_t(z,y)-a_t(z,y)\,\right|\leq \sup_{w\in \supp(\nu_0)} \left|\,{Q_t^{\gamma_\ell }}^{w,y}(z)-{Q_t}^{w,y}(z)\,\right| + \sum_{w\in \supp(\nu_0)} \left|\,\widehat{\pi}^{\gamma_\ell }_{_\leftarrow}(w|y)-\widehat{\pi}_{_\leftarrow}(w|y)\,\right|.\]
Therefore, the weak  convergence of $(\widehat{\pi}^{\gamma_\ell })_{k\in\N}$ to $\widehat{\pi}$ and  Lemma \ref{lemmetech} (iv) imply
\begin{equation}\label{lima}
\lim_{\gamma_\ell \to 0} a_t^{\gamma_\ell }(z,y)=a_t(z,y).
\end{equation}

Let us now consider the behaviour  of $\gamma_\ell  {\mathrm a}_t^{\gamma_\ell }(z,\tz,y)$ as $\gamma_\ell $ goes to zero.
 Lemma \ref{lemmetech} (iii) provides the following  Taylor expansion,
\begin{multline*}
\gamma\alpha_t^\gamma(y,z,\tz,w)=
\gamma^{d(y,z)+1+d(\tz,w)-d(y,w)}r(y,z,\tz,w)\,\frac{d(y,w)!}{d(y,z)! d(\tz,w)!}\,(1-t)^{d(y,z)}t^{d(\tz,w)}\\
\cdot \left(
1+\gamma \left(K^{d(y,z)}+K^{d(\tz,w)}+K^{d(y,w)}\right)O(1) \right),
\end{multline*}
where $O(1)$ is a quantity uniformly bounded in $t,\gamma,z,\tz,x,y$.
By the triangular inequality and since $z\sim\tz$, one has 
$d(y,w)\leq  d(y,z)+1+d(\tz,w),$
with equality if and only if $(z,\tz)\in[y,w]$. 
Therefore, one gets 
\[\lim_{\gamma\to 0} \gamma\alpha_t^{\gamma}(y,z,\tz,w)= \alpha_t(y,z,\tz,w),\]
with
\begin{eqnarray*}
 \alpha_t(y,z,\tz,w) :=\1_{(z,\tz)\in[y,w]}\,r(y,z,\tz,w)d(y,w)\B_t^{d(y,w)-1}(d(z,w)-1).
\end{eqnarray*}
Moreover,  Lemma \ref{lemmetech}  (i), (ii) and (iii) ensures that for any $w\in \supp(\nu_0)$ and $y\in \supp(\nu_1)$,
\begin{align*}
\gamma \alpha_t^\gamma(y,z,\tz,w)&\leq O(1) \,\gamma^{d(y,z)+1+d(\tz,w)-d(y,w)}\;{(2S)^{d(y,z)+d(\tz,w)-d(y,w)}}\;K^{d(y,z)+d(\tz,w)}\\
&\qquad\qquad\qquad\qquad \cdot
\,\max_{w\in \supp(\nu_0),y\in \supp(\nu_1)}\frac{(2S)^{d(y,w)} d(y,w)! K^{d(y,w)}}{I^{d(y,w)}}\\
&\leq O(1) \,(\gamma 2SK)^{d(y,z)+d(\tz,w)+1-d(y,w)},
\end{align*}
where $O(1)$ is a constant independent of $t,y,z,\tz,w$. 
Therefore $\gamma\alpha_t^\gamma(y,z,\tz,w)\leq O(1)$ as soon as $\gamma<1/(2SK)$.
As a consequence, for any $\gamma_\ell <1/(2SK)$, it holds
\begin{multline*}
\left|\, \gamma_\ell  {\mathrm a}^{\gamma_\ell }_t(z,\tz, y)-{\mathrm a}_t(z,\tz, y)\,\right|\\ \leq \sup_{w\in \supp(\nu_0)} \left|\,\gamma_\ell \alpha_t^{\gamma_\ell }(y,z,\tz,w)-\alpha_t(y,z,\tz,w)\,\right| + O(1) \sum_{w\in \supp(\nu_0)} \left|\,\widehat{\pi}^{\gamma_\ell }_{_\leftarrow}(w|y)-\widehat{\pi}_{_\leftarrow}(w|y)\,\right|,
\end{multline*}
 As  $\gamma_\ell $ goes to 0, this inequality with the weak convergence of $\widehat{\pi}^{\gamma_\ell }$ to  $\widehat{\pi}^{0}$ implies 
 \[\lim_{\gamma_\ell \to 0} \gamma_\ell  \,{\mathrm a}^{\gamma_\ell }_t(z,\tz, y)={\mathrm a}_t(z,\tz, y),\]
 The set $\widehat Y_z$ is not empty since  $z \in \widehat Z$. Since for any $y\in \widehat Y_z$,  $a_t(z,y)\neq 0$, it follows from \eqref{lima} that $\gamma_\ell  A_t^{\gamma_\ell }(z,\tz)$ converges as $\gamma_\ell $ goes to zero with for any $y\in \widehat Y_z$,
  \[\lim_{\gamma_\ell \to 0} \gamma_\ell  A_t^{\gamma_\ell }(z,\tz)=\frac{{\mathrm a}_t(z,\tz, y)}{a_t(z,y)}.\]
  The proof of the first part of Lemma \ref{lemintaibis} is completed.
  
  We now turn to the  proof of the second part of Lemma \ref{lemintaibis}. One will only compute  $\lim_{\gamma_\ell \to 0} \left(\gamma_\ell ^2 A_t^{\gamma_\ell }(z,\ttz)\right)$ for $z\in \widehat Z,\ttz\in \V(z)$  and the expression of  $\lim_{\gamma_\ell \to 0} \left(\gamma_\ell ^2 B_t^{\gamma_\ell }(z,\ttz)\right)$ follows from similar calculations.
 For any $y\in \X$ and any $t>0$, one has 
\[ A_t^{\gamma}(z,\ttz) =\frac{ {\mathrm a}_t^\gamma(z,\ttz,y)}{a^\gamma_t(z,y)}, \]
with 
\[ {\mathrm a}_t^\gamma(z,\ttz,y):=\int  \, \alpha_t^{\gamma}(y,z,\ttz,w)\,d\widehat{\pi}^\gamma_{_\leftarrow}(w|y)
.\]
It  remains to compute $\lim_{\gamma_\ell \to 0} \gamma_\ell ^2 {\mathrm a}_t^{\gamma_\ell }(z,\ttz,y)$ to prove \eqref{conv2}.  As above, Lemma \ref{lemmetech} (iii) provides 
\begin{multline*}
\gamma\alpha_t^\gamma(y,z,\ttz,w)=
\gamma^{d(y,z)+2+d(\ttz,w)-d(y,w)}r(y,z,\ttz,w)\,\frac{d(y,w)!}{d(y,z)! d(\ttz,w)!}\,(1-t)^{d(y,z)}t^{d(\ttz,w)}\\
\cdot \left(
1+\gamma \left(K^{d(y,z)}+K^{d(\ttz,w)}+K^{d(y,w)}\right)O(1) \right),
\end{multline*}
where $O(1)$ is a quantity uniformly bounded in $t,\gamma,z,\ttz,x,y$.
Since $d(y,w)\leq  d(y,z)+2+d(\ttz,w)$
with equality if and only if $(z,\ttz)\in[y,w]$, it follows that 
\begin{eqnarray*}
\lim_{\gamma\to 0}\gamma^2 \alpha_t^\gamma(y,z,\ttz,w)
=\alpha_t(y,z,\ttz,w):={\1_{(z,\ttz)\in[y,w]}} \,r(y,z,\ttz,w)\, d(y,w) (d(y,w)-1)\B_t^{d(y,w)-2}(d(z,w)-2).
\end{eqnarray*}
Moreover,  Lemma \ref{lemmetech} (i), (ii) and (iii) gives  that for any $w\in \supp(\nu_0)$ and $y\in \supp(\nu_1)$,
\[
 \gamma^2\alpha_t^\gamma(y,z,\ttz,w) \leq O(1)\, (\gamma 2SK)^{d(y,z)+d(\tz,w)+2-d(y,w)},
\]
where $O(1)$ is a constant independent of $t,y,z,\ttz,w$.
As above, the proof ends as $\gamma_\ell $ goes to 0 from the inequality 
\begin{multline*}
\left|\, \gamma_\ell ^2 {\mathrm a}^{\gamma_\ell }_t(z,\ttz, y)-{\mathbbm a}_t(z,\ttz, y)\,\right|\\ \leq \sup_{w\in \supp(\nu_0)} \left|\,\gamma_\ell ^2\alpha_t^{\gamma_\ell }(y,z,\ttz,w)-\alpha_t(y,z,\ttz,w)\,\right| + O(1) \sum_{w\in \supp(\nu_0)} \left|\,\widehat{\pi}^{\gamma_\ell }_{_\leftarrow}(w|y)-\widehat{\pi}_{_\leftarrow}(w|y)\,\right|,
\end{multline*}
for all $\gamma_\ell <1/(2SK)$. The end of the proof of the second part of Lemma \ref{lemintaibis} is identical to the one the first part.
\end{proof}




\bibliographystyle{plain} 
\bibliography{bib-new-courbure-discrete}       

\begin{thebibliography}{10}

\bibitem{AGS08}
L.~Ambrosio, N.~Gigli, and G.~Savar{\'e}.
\newblock {\em Gradient flows in metric spaces and in the space of probability
  measures}.
\newblock Lectures in Mathematics ETH Z\"urich. Birkh\"auser Verlag, Basel,
  second edition, 2008.

\bibitem{AGS14}
L.~Ambrosio, N.~Gigli, and G.~Savar{\'e}.
\newblock Calculus and heat flow in metric measure spaces and applications to
  spaces with {R}icci bounds from below.
\newblock {\em Invent. Math.}, 195(2):289--391, 2014.

\bibitem{Bak94}
D.~Bakry.
\newblock L'hypercontractivit\'e et son utilisation en th\'eorie des
  semigroupes.
\newblock In {\em Lectures on probability theory ({S}aint-{F}lour, 1992)},
  volume 1581 of {\em Lecture Notes in Math.}, pages 1--114. Springer, Berlin,
  1994.

\bibitem{BHLLMY15}
F.~Bauer, P.~Horn, Y.~Lin, G.~Lippner, D.~Mangoubi, S.-T. Yau, et~al.
\newblock Li-yau inequality on graphs.
\newblock {\em Journal of Differential Geometry}, 99(3):359--405, 2015.

\bibitem{BB00}
J.-D. Benamou and Y.~Brenier.
\newblock A computational fluid mechanics solution to the monge-kantorovich
  mass transfer problem.
\newblock {\em Numerische Mathematik}, 84(3):375--393, 2000.

\bibitem{BHT06}
S.~Bobkov, C.~Houdr{\'e}, and P.~Tetali.
\newblock The subgaussian constant and concentration inequalities.
\newblock {\em Israel Journal of Mathematics}, 156(1):255--283, 2006.

\bibitem{BS09}
A.I. Bonciocat and K.T. Sturm.
\newblock Mass transportation and rough curvature bounds for discrete spaces.
\newblock {\em J. Funct. Anal.}, 256(9):2944--2966, 2009.

\bibitem{CDP09}
P.~Caputo, P.~Dai~Pra, and G.~Posta.
\newblock Convex entropy decay via the bochner-bakry-emery approach.
\newblock {\em Annales de l'I.H.P. Probabilit\'es et statistiques},
  45(3):734--753, 2009.

\bibitem{Con18}
G.~Conforti.
\newblock A second order equation for schr{\"o}dinger bridges with applications
  to the hot gas experiment and entropic transportation cost.
\newblock {\em Probability Theory and Related Fields}, 174(1-2):1--47, 2019.

\bibitem{EHMT17}
M.~Erbar, C.~Henderson, G.~Menz, and P.~Tetali.
\newblock Ricci curvature bounds for weakly interacting markov chains.
\newblock {\em Electronic Journal of Probability}, 22, 2017.

\bibitem{EM12}
M.~Erbar and J.~Maas.
\newblock Ricci curvature of finite {M}arkov chains via convexity of the
  entropy.
\newblock {\em Arch. Ration. Mech. Anal.}, 206(3):997--1038, 2012.

\bibitem{EM14}
M.~Erbar and J.~Maas.
\newblock Gradient flow structures for discrete porous medium equations.
\newblock {\em Discrete \& Continuous Dynamical Systems-A}, 34(4):1355--1374,
  2014.

\bibitem{EMT15}
M.~Erbar, J.~Maas, and P.~Tetali.
\newblock Discrete ricci curvature bounds for bernoulli-laplace and random
  transposition models.
\newblock In {\em Annales de la Facult{\'e} des sciences de Toulouse:
  Math{\'e}matiques}, volume~24, pages 781--800, 2015.

\bibitem{FM16}
M.~Fathi and J.~Maas.
\newblock Entropic ricci curvature bounds for discrete interacting systems.
\newblock {\em The Annals of Applied Probability}, 26(3):1774--1806, 2016.

\bibitem{GRST14}
N.~Gozlan, C.~Roberto, P.-M. Samson, and P.~Tetali.
\newblock Displacement convexity of entropy and related inequalities on graphs.
\newblock {\em Probability Theory and Related Fields}, 160(1-2):47--94, 2014.

\bibitem{GRST17}
N.~Gozlan, C.~Roberto, P.-M. Samson, and P.~Tetali.
\newblock Kantorovich duality for general transport costs and applications.
\newblock {\em J. Funct. Anal.}, 273(11):3327--3405, 2017.

\bibitem{GRST19}
N.~Gozlan, C.~Roberto, P.-M. Samson, and P.~Tetali.
\newblock Transport proofs of some discrete variants of the
  pr{\'e}kopa-leindler inequality.
\newblock {\em Annali della Scuola Normale Superiore di Pisa. Classe di
  scienze}, 22(3):1207--1232, 2021.

\bibitem{HKS19}
D.~Halikias, B.~Klartag, and B.~A Slomka.
\newblock Discrete variants of brunn--minkowski type inequalities.
\newblock In {\em Annales de la Facult{\'e} des sciences de Toulouse:
  Math{\'e}matiques}, volume~30, pages 267--279, 2021.

\bibitem{Hil14}
E.~Hillion.
\newblock {$W_{1,+}$}-interpolation of probability measures on graphs.
\newblock {\em Electron. J. Probab.}, 19:no. 92, 29, 2014.

\bibitem{Hil17}
E.~Hillion.
\newblock Interpolation of probability measures on graphs.
\newblock In {\em Convexity and concentration}, volume 161 of {\em IMA Vol.
  Math. Appl.}, pages 3--32. Springer, New York, 2017.

\bibitem{KKRT16}
B.~Klartag, G.~Kozma, P.~Ralli, and P.~Tetali.
\newblock Discrete curvature and abelian groups.
\newblock {\em Canadian Journal of Mathematics}, 68(3):655--674, 2016.

\bibitem{KL19}
B.~Klartag and J.~Lehec.
\newblock Poisson processes and a log-concave {B}ernstein theorem.
\newblock {\em Studia Math.}, 247(1):85--107, 2019.

\bibitem{Leo12}
C.~L{\'e}onard.
\newblock From the schr{\"o}dinger problem to the monge--kantorovich problem.
\newblock {\em Journal of Functional Analysis}, 262(4):1879--1920, 2012.

\bibitem{Leo14}
C.~L{\'e}onard.
\newblock Some properties of path measures.
\newblock In {\em S{\'e}minaire de Probabilit{\'e}s XLVI}, pages 207--230.
  Springer, 2014.

\bibitem{Leo13}
C.~L{\'e}onard.
\newblock A survey of the schr{\"o}dinger problem and some of its connections
  with optimal transport.
\newblock {\em Discrete \& Continuous Dynamical Systems-A}, 34(4):1533--1574,
  2014.

\bibitem{Leo16}
C.~{L\'eonard}.
\newblock Lazy random walks and optimal transport on graphs.
\newblock {\em Ann. Probab.}, 44(3):1864--1915, 2016.

\bibitem{Leo17}
C.~L{\'e}onard.
\newblock On the convexity of the entropy along entropic interpolations.
\newblock In {\em Measure Theory in Non-Smooth Spaces}, pages 194--242. Sciendo
  Migration, 2017.

\bibitem{LV09}
J.~Lott and C.~Villani.
\newblock Ricci curvature for metric-measure spaces via optimal transport.
\newblock {\em Ann. of Math. (2)}, 169(3):903--991, 2009.

\bibitem{Maa11}
J.~Maas.
\newblock Gradient flows of the entropy for finite {M}arkov chains.
\newblock {\em J. Funct. Anal.}, 261(8):2250--2292, 2011.

\bibitem{Mar96b}
K.~Marton.
\newblock Bounding {$\overline d$}-distance by informational divergence: a
  method to prove measure concentration.
\newblock {\em Ann. Probab.}, 24(2):857--866, 1996.

\bibitem{McC97}
R.~J. McCann.
\newblock A convexity principle for interacting gases.
\newblock {\em Adv. Math.}, 128(1):153--179, 1997.

\bibitem{Mie13}
A.~Mielke.
\newblock Geodesic convexity of the relative entropy in reversible markov
  chains.
\newblock {\em Calculus of Variations and Partial Differential Equations},
  48(1-2):1--31, 2013.

\bibitem{Mik04}
T.~Mikami.
\newblock Monge’s problem with a quadratic cost by the zero-noise limit of
  h-path processes.
\newblock {\em Probability theory and related fields}, 129(2):245--260, 2004.

\bibitem{Oll09}
Y.~Ollivier.
\newblock Ricci curvature of {M}arkov chains on metric spaces.
\newblock {\em J. Funct. Anal.}, 256(3):810--864, 2009.

\bibitem{Oll13}
Y.~Ollivier.
\newblock A visual introduction to {R}iemannian curvatures and some discrete
  generalizations.
\newblock In {\em Analysis and geometry of metric measure spaces}, volume~56 of
  {\em CRM Proc. Lecture Notes}, pages 197--220. Amer. Math. Soc., Providence,
  RI, 2013.

\bibitem{OV12}
Y.~Ollivier and C.~Villani.
\newblock {A curved Brunn-Minkowski inequality on the discrete hypercube}.
\newblock {\em {Siam Journal on Discrete Mathematics}}, 26(3):983--996, 2012.

\bibitem{Sam17b}
P.-M. Samson.
\newblock Concentration of measure principle and entropy-inequalities.
\newblock In {\em Convexity and concentration}, pages 55--105. Springer, 2017.

\bibitem{Sam17}
P.-M. Samson.
\newblock Transport-entropy inequalities on locally acting groups of
  permutations.
\newblock {\em Electronic Journal of Probability}, 22, 2017.

\bibitem{Stu06a}
K.T. Sturm.
\newblock On the geometry of metric measure spaces. {I}.
\newblock {\em Acta Math.}, 196(1):65--131, 2006.

\bibitem{Stu06b}
K.T. Sturm.
\newblock On the geometry of metric measure spaces. {II}.
\newblock {\em Acta Math.}, 196(1):133--177, 2006.

\bibitem{Tal95}
M.~Talagrand.
\newblock Concentration of measure and isoperimetric inequalities in product
  spaces.
\newblock {\em Publications Math{\'e}matiques de l'Institut des Hautes Etudes
  Scientifiques}, 81(1):73--205, 1995.

\bibitem{Tal96c}
M.~Talagrand.
\newblock Transportation cost for {G}aussian and other product measures.
\newblock {\em Geom. Funct. Anal.}, 6(3):587--600, 1996.

\bibitem{Vil09}
C.~Villani.
\newblock {\em Optimal transport}, volume 338 of {\em Grundlehren der
  Mathematischen Wissenschaften [Fundamental Principles of Mathematical
  Sciences]}.
\newblock Springer-Verlag, Berlin, 2009.
\newblock Old and new.

\end{thebibliography}

%
%
%

\end{document}